\documentclass{amsart}

\usepackage[utf8]{inputenc}
\usepackage[english]{babel}

\usepackage{xcolor}
\usepackage[colorlinks,
    linkcolor={red!50!black},
    citecolor={blue!50!black},
    urlcolor={blue!80!black}]{hyperref}

\usepackage{tikz}
\usetikzlibrary{calc, matrix, arrows, cd}

\usepackage{amsfonts,amssymb,amsmath}

\theoremstyle{plain}
\newtheorem{theorem}{Theorem}[section]

\newtheorem{corollary}[theorem]{Corollary}
\newtheorem{lemma}[theorem]{Lemma}
\newtheorem{proposition}[theorem]{Proposition}

\newtheorem{innercustomthm}{Theorem}
\newenvironment{customthm}[1]
  {\renewcommand\theinnercustomthm{#1}\innercustomthm}
  {\endinnercustomthm}

\theoremstyle{definition}
\newtheorem{definition}[theorem]{Definition}

\newtheorem{example}[theorem]{Example}
\newtheorem{construction}[theorem]{Construction}
\newtheorem{notation}[theorem]{Notation}

\theoremstyle{remark}
\newtheorem{remark}[theorem]{Remark}

\newtheorem{claim}{Claim}

\newcommand{\claimbox}{$\diamond$}

\newcommand{\N}{\mathbb{N}}
\newcommand{\Z}{\mathbb{Z}}

\newcommand{\R}{\mathbb{R}}
\newcommand{\C}{\mathbb{C}}

\newcommand{\F}{\mathbb{F}}
\newcommand{\id}{\operatorname{id}}
\newcommand{\sign}{\operatorname{sign}}

\newcommand{\LF}{\mathbb{F}[t^{\pm 1}]}
\newcommand{\LR}{\mathbb{R}[t^{\pm 1}]}
\newcommand{\LC}{\mathbb{C}[t^{\pm 1}]}
  \newcommand{\bsm}{\left(\begin{smallmatrix}}
\newcommand{\esm}{\end{smallmatrix}\right)}
\newcommand{\GL}{\operatorname{GL}}
\newcommand{\U}{\operatorname{U}}

\newcommand{\sm}{\setminus}

\newcommand{\PD}{\operatorname{PD}}

\newcommand{\Hom}{\operatorname{Hom}}

\newcommand{\wt}[1]{\widetilde {#1}}
\newcommand{\ol}[1]{\overline {#1}}
\newcommand{\op}[1]{\operatorname{#1}}
\newcommand{\Int}{\operatorname{Int}}

\newcommand{\lf}{\text{lf}}
\newcommand{\cps}{\text{cs}}
\newcommand{\triv}{\text{triv}}

\newcommand{\tr}{\intercal}

\newcommand{\limone}{\operatorname{lim}^1}
\DeclareMathOperator*{\colim}{colim}
\DeclareMathOperator*{\invlim}{lim}
\newcommand{\clim}{ \colim \limits_{\rightarrow} }
\newcommand{\ilim}{ \invlim \limits_{\leftarrow} }
\newcommand{\limd}{ \lim \limits_{\rightarrow} }

\DeclareMathOperator\re{Re}

\begin{document}

\title{Twisted signatures of fibered knots}
\author{Anthony Conway}
\address{Max-Planck-Institut für Mathematik, Bonn, Germany}
\email{anthonyyconway@gmail.com}

\author{Matthias Nagel}
\address{Department of Mathematics, ETH Zurich, Switzerland}
\email{matthias.nagel@math.ethz.ch}

\begin{abstract}
This paper concerns twisted signature invariants of knots and 3-manifolds.
In the fibered case, we reduce the computation of these invariants to the study of the intersection form and monodromy on the twisted homology of the fiber surface.
Along the way, we use rings of power series to obtain new interpretations of the twisted Milnor pairing introduced by Kirk and Livingston.
This allows us to relate these pairings to twisted Blanchfield pairings.
Finally, we study the resulting signature invariants, all of which
are twisted generalisations of the Levine-Tristram signature.
\end{abstract}
\maketitle

\section{Introduction}
In the seventies, Rudolph asked whether the set of algebraic knots is linearly independent in the knot concordance group. 
Evidence for a positive answer to this question has been provided for torus knots~\cite{LitherlandIterated}, using the Levine-Tristram signature, and for larger family of algebraic knots, using Casson-Gordon invariants~\cite{HeddenKirkLivingston} and twisted Blanchfield pairings~\cite{ConwayKimPolitarczyk}.
In a similar vein, Baker has conjectured that if two strongly quasipositive fibered knots~$K_0$ and~$K_1$ are concordant, then they must be equal~\cite{Baker}.
Baker's conjecture has attracted increasingly much attention~\cite{AbeTagami, Misev, MisevSpano, ConwayKimPolitarczyk, HeHubbardTruong}.

In light of these recent developments, our goal is to improve the understanding of twisted signature invariants of fibered knots.
In the abelian case, the Levine-Tristram signature, which is usually defined using Seifert matrices, can also be understood using the Blanchfield pairing or the Milnor pairing.
In the metabelian case, signature invariants are vastly more difficult to compute.
However, while Seifert matrices are not available, signatures 
can nevertheless be extracted from twisted generalisations of 
the aforementioned pairings.
This was carried out for the twisted Blanchfield pairing in~\cite{BorodzikConwayPolitarczyk} (including explicit computations), and this paper investigates the twisted Milnor signatures defined in~\cite{KirkLivingston}; see also~\cite{Nosaka}.

\subsection{The twisted Milnor pairing of fibered \texorpdfstring{$3$--manifolds}{3-manifolds}}
Twisted signature invariants of knots are obtained from their $0$-framed surgery.
For this reason, given a field $\F$
with involution $x \mapsto \overline{x}$, we let $X$ be a closed $3$-manifold with infinite cyclic cover~$X^\infty$ classified by a map~$f \colon X \to S^1$, and let $M = (\F^k, \rho)$ be a unitary representation
with~$\rho \colon \pi_1(X) \to \U(k, \F)$. Here, $\U(k,\F) \subset \GL(k, \F)$ denotes
the subgroup of matrices~$A$ with the property that $A\cdot \ol A^\tr = \id_k$.
We use~$H_1(X^\infty;M)$ to denote the twisted homology of $X^\infty$ with coefficients in $M$. 
Deck transformations equip this $\F$-vector space with the structure of an $\F[t^{\pm 1}]$-module.
Provided~$H_1(X^\infty;M)$ is $\LF$-torsion, Kirk and Livingston~\cite[Section 7]{KirkLivingston} generalised work of Milnor~\cite{MilnorInfiniteCyclic} to define a skew-Hermitian \emph{twisted Milnor pairing}
\[ \mu_{X,M} \colon H_1(X^\infty;M) \times H_{1}(X^\infty;M) \to \F. \]
Before describing~$\mu_{X,M}$ in more detail, we first return to fibered knots.

Let $f\colon X \to S^1$ be a fibered $3$-manifold with fiber surface $\Sigma$. Denote the monodromy by~$\varphi \colon \Sigma \to \Sigma$.
Our first theorem 
expresses~$\mu_{X,M}$ in terms of both the induced automorphism~$\varphi_M \colon H_1(\Sigma;M) \to H_1(\Sigma;M)$ and the twisted intersection form $\lambda_{\Sigma,M}$ on $H_1(\Sigma;M)$. 
We refer to Subsection~\ref{sub:Fibered} for the definition of $\varphi_M$.
To state our theorem, we also use $\iota \colon H_1(\Sigma;M) \to H_1(X^\infty;M)$ to denote the $\F$-linear isomorphism induced by the inclusion.

\begin{customthm}{\ref{thm:MilnorIntersectionMonodromy}}
\label{thm:Main}
Let $f \colon X \to S^1$ be a closed fibered $3$-manifold with fiber surface~$\Sigma$ and monodromy~$\varphi$.
Let $M$ be a unitary representation such that $H_1(X^\infty;M)$ is~$\F[t^{\pm 1}]$-torsion.
Then, for any $k \in \Z$, the twisted Milnor pairing can be expressed~as
\[ \mu_{X,M}(t^k \iota(x),\iota(y))=\lambda_{\Sigma,M}(\varphi_M^{-k}(x),y).  \]
\end{customthm}

Theorem~\ref{thm:MilnorIntersectionMonodromy} therefore reduces the study of the twisted Milnor pairing and its signature invariants to the study of the twisted intersection form  and monodromy of the fiber surface.
We briefly describe a simple example of Theorem~\ref{thm:MilnorIntersectionMonodromy}.
\begin{example}\label{ex:Erle}
Let $K$ be a fibered knot with a Seifert surface $F$ as a fiber surface, and set $\Sigma:=F \cup_\partial D^2$.
The $0$-framed surgery $M_K$ is now a fibered $3$-manifold with fiber surface $\Sigma$.
If $V$ denotes an invertible Seifert matrix for $F$, then the monodromy of $F$ is $V^{-1}V^T$ and the intersection form is $V-V^T$.
For the trivial representation~$M=\C$, the complex Alexander module~$H_1(M_K^\infty;\C)$ is endowed with the classical (untwisted) Milnor pairing $\mu_K$~\cite{MilnorInfiniteCyclic}.
Theorem~\ref{thm:MilnorIntersectionMonodromy} then implies that the symmetric form $b_K(x,y):=\mu_K(tx,y)-\mu_K(x,ty)$ is represented by the matrix
\[  \left( V^{-1} V^\tr \right)^\tr \left( V- V^\tr\right) - \left( V- V^\tr\right)\left( V^{-1} V^\tr\right). \]

As we show in Example~\ref{ex:ErleLT}, this matrix turns out to be similar to $V+V^T$, whose signature is the classical Murasugi signature of a knot.
This recovers a result of Erle~\cite{Erle} in the case of fibered knots and illustrates the type of example that we wish to generalise to the twisted setting. 
\end{example}

\subsection{Consequences for signatures of fibered knots}
Next, we recall how to extract signatures from skew-Hermitian pairings~\cite{MilnorInfiniteCyclic}.
A \emph{skew-isometric structure}~$(H,\mu,t)$ consists of a 
finite dimensional $\F$-vector space $H$, a non-singular skew-Hermitian pairing $\mu \colon H \times H \to \F $ and an isometry $t \colon H \to H$ of $\mu$.
In particular, $H$ is an $\LF$-module.
Signature invariants are obtained by considering the restriction of the following Hermitian form on $H$ to the primary summands of~$H$:
\[ b(x,y):=\mu(tx,y)+ \mu(x,ty). \]
For a $3$-manifold $X$ and a module $M$ as in Theorem~\ref{thm:MilnorIntersectionMonodromy}, Kirk and Livingston considered the skew-Hermitian structure $(H,\mu,t):=(H_1(X^\infty;M),\mu_{X,M},t)$.
While these authors defined signature invariants by relying on the primary decomposition, computations of related signature invariants for low crossing knots did not appear until the preprint of Nosaka~\cite{Nosaka}.
A framework to extract and study signatures from twisted Blanchfield pairings was set up in~\cite{BorodzikConwayPolitarczyk}
and applied to the study of Rudolph's conjecture in~\cite{ConwayKimPolitarczyk}.
We compare these signature invariants in Section~\ref{sec:MilnorBlanchfieldSignatures}.

One advantage of the Milnor pairing over the Blanchfield pairing is that it contains a signature invariant that can be defined without using the primary decomposition (and therefore without computing twisted Alexander polynomials), namely
\[ \sigma(X,M):=\sign(b_{X,M}).\]
Recall from Example~\ref{ex:Erle} that if $X=M_K$ is the $0$-framed surgery along a knot~$K$, and~$M=\C$ the trivial representation,
then $\sigma(X,M)$ is the classical Murasugi signature of the knot $K$~\cite{Erle}. 
On the other hand, considering the primary decomposition recovers (jumps of) the Levine-Tristram signature of $K$~\cite{Matumoto}.

As a corollary of Theorem~\ref{thm:MilnorIntersectionMonodromy}, we obtain the following description of the signature $\sigma(X,M)$ in terms of the fiber surface of a fibered $3$-manifold.

\begin{corollary}\label{cor:SignatureFibered}
Let $f \colon X\to S^1$ a closed fibered $3$-manifold with fiber~$\Sigma$ and monodromy~$\varphi$.
	Let $M$ be a unitary representation such that $H_1(X^\infty;M)$ is~$\F[t^{\pm 1}]$-torsion.
    Then the symmetrization $b_{X,M}$ of the twisted Milnor pairing can be expressed as
\[ b_{X,M}(x,y)=\lambda_{\Sigma,M}(\varphi_M^{-1}(x),y)-\lambda_{\Sigma,M}(x,\varphi_M^{-1}(y)).\]
In particular, the twisted signature $\sigma(X,M)$ can be expressed purely in terms of the fiber surface~$\Sigma$ and the induced map~$\varphi_M$.
\end{corollary}

We outline how these results are relevant to knot concordance.
As mentioned in~\cite[p. 660]{KirkLivingston}, if $H_1(X^\infty;M)$ has no $(t \pm 1)$-primary torsion, $X$ bounds a $4$-manifold $W$ and the coefficient system extends, then the pairing~$b_{X,M}$ must be metabolic (this also follows from Theorem~\ref{thm:BlanchfieldMilnor} and~\cite[Proposition 7.7]{BorodzikConwayPolitarczyk}). 
By \emph{metabolic}, we mean that there is a $t$-invariant subspace $L \subset H_1(X^\infty;M)$ such that~$L=L^\perp$ with respect to $b_{X,M}$; we refer to Subsection~\ref{sub:WittSkewIsom} for details.

We apply this fact to metabelian signature invariants as follows.
Given a knot~$K$, set~$X:=M_n(K)$, where~$M_n(K) \to M_K$ is the $n$-fold cyclic cover of the $0$-framed surgery~$M_K$.
If~$K$ is fibered, then~$M_K$ is fibered with fiber $\Sigma$ a capped-off Seifert surface~$\Sigma=F\cup_\partial D^2$.
We use~$\varphi$ to denote the monodromy of~$F$ (and $\Sigma$), so that~$M_n(K)$ is fibered with fiber~$\Sigma$ and monodromy~$\varphi^n$.
Additionally, we let~$\Sigma_n(K)$ be the $n$-fold branched cover of~$K$.
The infinite cyclic cover of $M_n(K)$ is classified by the projection induced map~$H_1(M_n(K);\Z) \to H_1(M_K;\Z)\cong\Z$, and Casson-Gordon theory~\cite{CassonGordon2} ensures that for a character $\chi \colon H_1(\Sigma_n(K)) \to~\C^*$, the resulting twisted homology $\C$-vector space $H_1(M_n^\infty(K);\C^\chi)$ is finite dimensional, i.e. torsion as a~$\LC$-module.
Here, we used the isomorphism $H_1(M_n(K)) \cong H_1(\Sigma_n(K)) \oplus \Z$ to view $\chi$ a character on $\pi_1(M_n^\infty(K)) \subset \pi_1(M_n(K))$.
We use~$\sigma(M_n(K),\C^\chi)$ to denote the resulting signature.
Combining the work of Casson and Gordon with these considerations and Theorem~\ref{thm:MilnorIntersectionMonodromy} yields the following~result. 

\begin{corollary}\label{cor:CGCorollary}
Let $n$ be a prime power and let $K$ be a slice fibered knot with Seifert surface $F$.
Assume that $H_1(M_n^\infty(K);\C^\chi)$ has no $(t \pm 1)$-primary torsion. 
There exists an invariant metaboliser $G$ of $H_1(\Sigma_n(K);\Z)$ such that for every non-trivial prime power
order character $\chi$ vanishing on $G$,
\[ \sigma(M_n(K),\C^\chi)=0. \]
Furthermore, the twisted signature $\sigma(M_n(K),\C^\chi)$ can be explicitly computed from the induced monodromy~$\varphi_{\C^\chi}$ of $F$ and the twisted intersection form $\lambda_{F,M}$.
\end{corollary}

In Corollary~\ref{cor:CGCorollary}, an \emph{invariant metaboliser} refers to a $\Z[\Z_n]$-invariant subgroup $G \subset H_1(\Sigma_n(K);\Z)$ such that $G=G^\perp$ with respect to the linking form; we refer to~\cite{CassonGordon1,KirkLivingston} for further details.

\subsection{Definitions of the twisted Milnor pairing}
We now describe the Milnor pairing in more details and how its signatures relate to previously studied metabelian signatures~\cite{KirkLivingston, BorodzikConwayPolitarczyk}. 
Given a $3$-manifold $X$, Milnor used the two ends of~$X^\infty$ to show that $H^2(X^\infty;\F)\cong H^3_{\op{cs}}(X^\infty;\F)$, where $H^\bullet_{\op{cs}}(X^\infty;\F)$ denotes the compactly supported cohomology of~$X^\infty$~\cite[Section 4]{MilnorInfiniteCyclic}.
Since Poincar\'e duality gives~$H^3_{\op{cs}}(X^\infty;\F) \cong H_0(X^\infty;\F)=\F$, Milnor then used this result to define his pairing as the following cup product on cohomology:
\[ \mu_X \colon H^1(X^\infty;\F) \times H^1(X^\infty;\F) \to H^2(X^\infty;\F)\cong H^3_{\op{cs}}(X^\infty;\F)\cong \F. \]
Later, Litherland outlined a definition of $\mu_X$ on homology that involves the locally finite homology of $X^\infty$~\cite{LitherlandCobordism}.
The identification with Milnor's original pairing, though implicit, is not proved.
Litherland then used a \emph{trace map} to relate the (untwisted) Blanchfield pairing $\operatorname{Bl}_X \colon H_1(X^\infty;\F) \times H_1(X^\infty;\F) \to \F(t)/\F[t^{\pm 1}] $ to his version of the Milnor pairing~\cite[Theorem A.1]{LitherlandCobordism}.
More precisely, he showed that $\chi \circ \operatorname{Bl}_X =\mu_X$, where the trace map
\[ \chi \colon \F(t)/\F[t^{\pm 1}] \to \F \]
will be described in more details in Section~\ref{sec:BlanchfieldMilnor}.
Upon defining their twisted Milnor pairing $\mu_{X,M}$, Kirk and Livingston built on Milnor's cohomological construction, rather than Litherlands'~\cite[Section 7]{KirkLivingston}.
Consequently, if one wishes to relate~$\mu_{X,M}$ to the twisted Blanchfield pairing, then one must work with a locally finite version of twisted homology.
A second difficulty then appears: Litherland's argument involves explicit chain computations, and these arguments have to be translated to twisted homology~\cite{FriedlPowell}.

Our idea is therefore to replace the use of twisted cohomology involving the ends of $X^\infty$ by twisted homology of~$X$ with coefficients involving rings of power series.
Featuring prominently in this approach is the ring $\Gamma$ of power series defined by
\[ \Gamma := \F[[t^{\pm 1}]] = \left\{ \sum_{i = -\infty}^{\infty} a_i t^i \right\}.\]  
Milnor's arguments, which involve compactly supported 
cohomology of $X^\infty$ (and its ends), are then replaced by arguments which make use of the twisted homology of $X$ as well as the two rings~$\Lambda:=\F[t^{\pm 1}]  \subset \Gamma^\pm \subset \Gamma$, given by 
\begin{align*}
\Gamma^+ &:= \left\{ p \in \Gamma : p = \sum_{i = k}^{\infty} a_i t^i \text{ for a } k \in \Z \right\} \\
\Gamma^- &:= \left\{ p \in \Gamma : p = \sum_{i = -\infty}^{k} a_i t^i \text{ for a } k \in \Z\right\}.
\end{align*}
Namely, we employ twisted homology with coefficients both in $M(\Gamma):=\Gamma \otimes_\F~M$ as well as in~$M(\Gamma^+)$, $M(\Gamma^-)$ and $M(\Lambda)$, which are analogously defined. 
In Definition~\ref{def:MilnorSignature}, we use this set-up and a Bockstein homomorphism to define a twisted Milnor pairing 
\[ \widetilde{\mu}_{X,M} \colon H_1(X;M(\Lambda)) \times H_1(X;M(\Lambda))  \to \F. \]
While $H_1(X;M(\Lambda))$ is known to be isomorphic to $H_1(X^\infty;M)$~\cite[Theorem~2.1]{KirkLivingston}, the corresponding statements for $M(\Gamma^-),M(\Gamma^+)$ and $M(\Gamma)$ take some more work to establish: they involve the ends of $X^\infty$ and locally finite homology.
As an outgrowth of these considerations, we obtain the following new interpretation of the twisted Milnor pairing.

\begin{theorem}\label{thm:PairingsAgree}
Let $X$ be a closed $3$-manifold with an epimorphism $\pi_1(X) \twoheadrightarrow \Z$, and let $M$ be a unitary representation such that $H_1(X^\infty;M)$ is~$\LF$-torsion.
The twisted Milnor pairing $\mu_{X,M}$ is isometric to $\widetilde{\mu}_{X,M}$: there is an~$\F[t^{\pm 1}]$-linear isomorphism
$\varphi \colon H^1(X^\infty;M)  \to H_1(X;M(\Lambda))$
such that, for all~$x,y \in H^1(X^\infty;M)$
\[ \mu_{X,M}(x,y)=\widetilde{\mu}_{X,M}(\varphi(x),\varphi(y)).\]
\end{theorem}

To the best of our knowledge, this interpretation of the Milnor pairing is new, even in the untwisted case.
Theorem~\ref{thm:PairingsAgree} is proved in three steps.
Firstly, in Definition~\ref{def:TwistedMilnorPairingCohomology}, we define an additional Milnor pairing $\widetilde{\mu}^\cup_{X,M}$, this time on $H^1(X;M(\Lambda))$.
Secondly, Proposition~\ref{prop:CohomologicalPairingsAgree} relates the Kirk-Livingston pairing $\mu_{X,M}$ on~$H^1(X^\infty;M)$ to $\widetilde{\mu}^\cup_{X,M}$; this makes use of the fact that both pairings are cohomological.
Thirdly, Proposition~\ref{prop:IsometryTwistedPairings} relates $\widetilde{\mu}^\cup_{X,M}$ to the pairing~$\widetilde{\mu}_{X,M}$ on $H_1(X;M(\Lambda))$; here since both pairings are defined via power series, the argument boils down to relations between twisted cap and cup products.

\begin{remark}\label{rem:Module}
Since $\Lambda:=\LF$ is a PID, a $\Lambda$-module is $\Lambda$-torsion if and only if it is finite dimensional as an $\F$-vector space.
In particular, the assumptions in Theorems~\ref{thm:MilnorIntersectionMonodromy} and~\ref{thm:PairingsAgree} are equivalent to requiring that the twisted Alexander polynomial $\operatorname{ord}_\Lambda H_1(X;M(\Lambda))$ is non-zero.
\end{remark}

These various definitions of $\mu_{X,M}$ are useful depending on the setting.
For instance, while Theorem~\ref{thm:MilnorIntersectionMonodromy} requires geometric input from the infinite cyclic cover (and in fact a fourth definition of $\mu_{X,M}$; see Definition~\ref{def:GeometricTwistedMilnor}), the relation to the twisted Blanchfield pairing is best understood from the purely algebraic description of $\widetilde{\mu}_{X,M}$.
A related fact: a multivariable version of the Milnor pairing for links currently seems difficult to define (contrarily to the Blanchfield pairing), and we hope that the new perspectives afforded by Theorem~\ref{thm:PairingsAgree} will shed light on this issue and on  algebraic concordance of links.
The univariate case of the Milnor pairing of a link has been studied in~\cite{KawauchiOnQuadratic, KawauchiSignature, Nosaka}.

\subsection{Relation to the twisted Blanchfield pairing}
Our new approach makes the relation between the Milnor pairing and the Blanchfield pairing more transparent, even in the untwisted case.
Indeed, Litherland's trace map $\chi$ factors through~$\Gamma^- \oplus \Gamma^+$, and this fact is particularly well suited to the definition of~$\widetilde{\mu}_{X,M}$. 
In order to state the result in the twisted case and using the same notation as above, we recall that the twisted Blanchfield pairing~\cite{BorodzikConwayPolitarczyk,MillerPowell,PowellBlanchfield} is a sesquilinear Hermitian non-singular pairing
\[ \operatorname{Bl}_{X,M} \colon H_1(X;M(\Lambda)) \times H_1(X;M(\Lambda)) \to \F(t)/\Lambda.  \]
While we will recall the definition of $\operatorname{Bl}_{X,M}$ in Section~\ref{sec:BlanchfieldMilnor}, the upshot of this discussion is that a purely homological argument now relates the twisted Blanchfield pairing to the twisted Milnor pairing.

\begin{theorem}\label{thm:BlanchfieldMilnorIntro}
	Let $X$ be a closed $3$-manifold with an epimorphism $\pi_1(X) \twoheadrightarrow \Z$, and let $M$ be a unitary representation such that $H_1(X^\infty;M)$ is $\LF$-torsion.
The twisted Milnor pairing $\widetilde{\mu}_{X,M}$ is the image of the twisted Blanchfield pairing $ \operatorname{Bl}_{X,M}$ under the trace map: 
\[ \chi \circ \operatorname{Bl}_{X,M} =\widetilde{\mu}_{X,M}.\]
\end{theorem}

Theorem~\ref{thm:BlanchfieldMilnorIntro} generalises Litherland's result~\cite[Theorem A.1]{LitherlandCobordism} to the twisted setting.
Furthermore, since the twisted Blanchfield pairing is known to be Hermitian and non-singular (since $X$ is closed)~\cite{PowellBlanchfield}, Theorem~\ref{thm:BlanchfieldMilnorIntro} and the properties of the trace map provide a quick proof that the twisted Milnor pairing is skew-Hermitian and non-singular; see Corollary~\ref{cor:SkewSym}.
While this fact is not obvious from the definition (in fact from neither of our definitions), it was already known to Kirk and Livingston and stated in~\cite[Theorem 7.1]{KirkLivingston}. 
As noted by Nosaka~\cite[Subsection 2.2 and Example 4.2]{Nosaka}, twisted Milnor pairings can be singular for manifolds with boundary.
This is explained by the fact that the inclusion induced map~$H_1(X;M) \to H_1(X,\partial X;M)$ need not be an isomorphism.

\begin{remark}\label{rem:SatelliteIntro}
As satellite formulas are known for the twisted Blanchfield pairing~\cite[Theorem 7.11]{BorodzikConwayPolitarczyk}, novel satellite formulas for the twisted Milnor pairing (and their signatures) can be deduced from Theorem~\ref{thm:BlanchfieldMilnorIntro}.
This generalises the untwisted version stated in~\cite{KeartonCompound} and proved in~\cite{LivingstonMelvin}.
\end{remark}

\subsection{Twisted Milnor signatures and signatures jumps}
Theorem~\ref{thm:BlanchfieldMilnor} is also the starting result needed to describe the relation between Kirk and Livingston's \textit{twisted Milnor signatures} with the \emph{signatures jumps} of~\cite{BorodzikConwayPolitarczyk}. 
To state our theorem, we recall some terminology and notation from~\cite{BorodzikConwayPolitarczyk}.
If $\F=\R$, then we set $p_\xi(t):=t-2\operatorname{Re}(\xi)+t^{-1}$, with~$\xi \in S^1 \cap \lbrace \operatorname{Im}(z)>0 \rbrace$.
If $\F=\C$, then we set~$p_\xi(t):=t-\xi$ with $\xi \in S^1 \setminus \lbrace 1 \rbrace$.
Such polynomials are referred to as \emph{basic polynomials}.

Let $X$ be a 3-manifold, and let $M$ be a unitary representation.
The work of~\cite{BorodzikConwayPolitarczyk} uses twisted Blanchfield pairings to associate to this data a \emph{signature jump} $\delta \sigma_{X,M}(\xi)$ whose definition we recall in Subsection~\ref{sub:WittLinkingForm}.\footnote{For $\F=\C$, our definition of $\delta \sigma_{X,M}(\xi)$ differs from the one in~\cite{BorodzikConwayPolitarczyk} by a sign; see Remark~\ref{rem:BCPSign}. For $\F=\R$, both definitions agree.}
They key advantage of these signature jumps is that they can be computed if one knows the primary decomposition of~$H_1(X;M)$ and the Blanchfield pairing of any two of its elements; see~\cite[Section 9]{BorodzikConwayPolitarczyk} for an example of this process in the Casson-Gordon setting.
On the other hand, as we alluded to above, the restriction of $b_{X,M}$ to the $p_\xi$-primary components of~$H_1(X;M(\Lambda))$ produces twisted Milnor signatures which we denote by~$\sigma_{(X,M)}(\xi)$.
These are the invariants defined in~\cite[Subsection 7.2]{KirkLivingston} (and which, in the fibered case, are now also known to be computable in terms of the twisted intersection form and monodromy).

\begin{theorem}\label{thm:SignaturesAreEqual}
Let $X$ be a closed $3$-manifold with an epimorphism~$\pi_1(X) \twoheadrightarrow \Z$, and let $M$ be a unitary representation such that $H_1(X^\infty;M)$ is~$\LF$-torsion.
Up to multiplication by a constant, for $\xi \in S^1 \setminus \lbrace 1 \rbrace$, the twisted Milnor signatures agree with the signatures jumps of the twisted Blanchfield pairing:
\[
\sigma_{(X,M)}(\xi)=
\begin{cases}
	-2 (\delta \sigma_{X,M})(\xi)  &\quad \text{ if } \F=\R, \\
\operatorname{sign}(\operatorname{Im}(\xi)) \cdot \delta \sigma_{X,M}(\xi) &\quad \text{ if } \F=\C.  
\end{cases}
\]
\end{theorem}

In the untwisted case, Kearton~\cite{KeartonSignature} (see also~\cite{LevineMetabolicHyperbolic}) extracted signatures invariants from the classical Blanchfield by a different method from~\cite{BorodzikConwayPolitarczyk} and showed that his signatures coincide with those of Milnor; his proof is topological~\cite[Section 11]{KeartonSignature}. 
Theorem~\ref{thm:SignaturesAreEqual} is obtained by combining Theorem~\ref{thm:MainMain} with Theorem~\ref{thm:BlanchfieldMilnorIntro}.
Theorem~\ref{thm:MainMain} constitutes the core of the proof and is purely algebraic: relying on the theory of Witt groups, we show that the signatures jumps of a linking form $(H, \lambda)$ are related to the signatures of the skew-isometric structure~$(H,\chi \circ \lambda,t)$.
Thanks to Theorem~\ref{thm:BlanchfieldMilnor}, Theorem~\ref{thm:SignaturesAreEqual} then follows by applying this result to the twisted Blanchfield pairing.
Some additional work shows that Kearton's approach to signatures is equivalent to the one in~\cite{BorodzikConwayPolitarczyk}, and in particular, Theorem~\ref{thm:SignaturesAreEqual} recovers Kearton's result in the untwisted case. 

\subsection*{Organisation}
Section~\ref{sec:TwistedHomology} reviews twisted homology and cohomology.
Section~\ref{sec:Milnor} defines the twisted Milnor pairing using coefficient systems involving rings of power series.
Section~\ref{sec:GeometricMilnor} reformulates the twisted Milnor pairing via infinite cyclic covers, proving Theorem~\ref{thm:PairingsAgree}.
Section~\ref{sec:Fibered} describe the twisted Milnor pairing of fibered $3$-manifolds, proving Theorem~\ref{thm:MilnorIntersectionMonodromy}.
Section~\ref{sec:BlanchfieldMilnor} relates the twisted Milnor pairing to the twisted Blanchfield pairing, proving Theorem~\ref{thm:BlanchfieldMilnorIntro}. 
Section~\ref{sec:Isometric} reviews signatures of skew-isometric structures and Section~\ref{sec:MilnorBlanchfieldSignatures} proves Theorem~\ref{thm:SignaturesAreEqual}.

\subsection*{Acknowledgments}
We thank Stefan Friedl for suggesting that we work with rings of power series in the study of the Milnor pairing. This facilitated the comparison with the Blanchfield pairing.
We also thank the University of Geneva and the University of Regensburg at which part of this work was conducted.
AC thanks the Max Planck Institute for Mathematics for its support.
MN gratefully acknowledges support by the SNSF Grant 181199.

\section{Preliminaries}
\label{sec:TwistedHomology}

\subsection{Twisted homology and cohomology}
\label{sub:homomology}
 
We quickly recall some concepts related to modules over a ring~$R$ with an involution.
An \emph{involution} on a ring~$R$ is a map~$r \mapsto \overline r$ fulfilling the properties
\begin{align*}
\overline 1 &= 1&
\overline {rs} &= \ol s \ \ol r\\
\overline {r + s} &= \ol r + \ol s & \ol {\ol r} &= r
\end{align*}
for all $r,s \in R$. Natural examples are group rings~$\Z[G]$ of a group~$G$, which
has involution~$\ol g = g^{-1}$. Given a left $R$--module~$M$, denote the \emph{transposed module} by~$M^\tr$: 
this is a right $R$--module that has the same underlying abelian group as~$M$, but the following right $R$--multiplication:
\begin{align*}
M^\tr \times R \rightarrow M^\tr\\
m \cdot r = \ol r \cdot m.
\end{align*}
The transposed module of a right $R$--module is defined similarly.
The transposed module~$M^\tr$ of an $(R,S)$--bimodule~$M$ is an $(S,R)$--bimodule. 

We will mainly deal will the following chain complexes associated to 
a connected finite CW-complex~$X$ with fundamental group~$\pi_1(X)$. We also consider
a subcomplex $Y \subset X$. We denote by~$p\colon \widetilde{X} \rightarrow X$ the universal cover of~$X$. The left action of~$\pi_1(X)$ on~$\widetilde{X}$ and on~$p^{-1}(Y)$ turns the 
cellular chain complexes~$C_\bullet(\widetilde X)$, $C_\bullet( p^{-1}(Y) )$, and $C_\bullet(\widetilde X, p^{-1}(Y))$
into chain complexes of $\Z[\pi_1(X)]$--modules.

\begin{definition}
\label{def:TwistedHomology}
For a commutative ring~$R$, a~$(R,\mathbb{Z}[\pi_1(X)])$--bimodule~$M$ and
a subspace $Y \subset X$ containing the base point,
define the \emph{twisted (co)chain complexes} as the chain complexes
\begin{align*}
C_\bullet(X,Y;M)&=M \otimes_{\mathbb{Z}[\pi_1(X)]}C_\bullet\left(\widetilde X, p^{-1}(Y) \right),\\
C_\bullet(Y \subset X; M) &= M \otimes_{\Z[\pi_1(X)]} C_\bullet\left(p^{-1}(Y) \right),\\
	C^\bullet(X,Y;M) &=\operatorname{Hom}_{\text{right-}\mathbb{Z}[\pi_1(X)]}\left( {C_\bullet(\widetilde{X},p^{-1}(Y)) }^\tr, M \right),\\
	C^\bullet(Y \subset X;M)& = \operatorname{Hom}_{\text{right-}\mathbb{Z}[\pi_1(X)]}\left( C_\bullet(p^{-1}(Y))^\tr , M \right).
\end{align*}
of left $R$--modules. Denote the corresponding homology
and cohomology left $R$-modules by $H_\bullet(X,Y;M), H_\bullet(Y \subset X; M)$ and
$H^\bullet(X,Y;M), H^\bullet(Y \subset X;M)$.
\end{definition}
Given two $(R, \Z[\pi])$--modules $M$ and $N$, we will often need to take twisted (co)homology with coefficients in the tensor product $M \otimes_{\Z}N^\tr$.
For this to be possible however, we need to endow this left $R$--module with the structure of a right~$\Z[\pi]$--bimodule.
This is done by using the diagonal action, as spelled out in the next remark.

\begin{remark}
\label{rem:Box}
Consider a commutative ring~$R$ and two $(R, \Z[\pi])$--modules $M$ and $N$. 
Since $R$ is commutative, we can consider the abelian group $M \otimes_\Z N$. 
We now equip this $\Z$-module with the structure of an~$(R, \Z[\pi] \times R)$--bimodule, denoted $M\boxtimes N$.
The left and right actions of $R$ are~via 
\[ r \cdot (m \otimes n) \cdot s = ( r \cdot m ) \otimes ( \ol s \cdot n )\]
for $r ,s \in R$ and $m \in M$ and $n \in N$.
The right $\pi$--action is the diagonal action
\[ (m\otimes n) \cdot g = (m \cdot g) \otimes (n \cdot g) \]
for $r \in R$ and $m \in M$ and $n \in N$.
It is to stress the diagonal action that we use the symbol~$\boxtimes$ instead of the more common $\otimes$.
\end{remark}

Let $X$ be a space and let $M$ and $N$ be $(R,\Z[\pi_1(X)])$-bimodules.
As $M \boxtimes~N$ is in particular an $(R, \Z[\pi_1(X)])$-bimodule, we obtain the left $R$-modules $H_\bullet(X;M \boxtimes N)$ and $H^\bullet(X;M \boxtimes N)$.
The right $R$-action on $M \boxtimes N$, turns these left $R$-modules into~$(R,R)$-bimodules.
The next lemma describes $H_0\big(X; A\boxtimes B \big)$.

\begin{lemma}\label{lem:Ccoinvariants}
Let $X$ be a connected space, and let $A$ and $B$ be $(R,\Z[\pi_1(X)])$--bimodules.
	The map~$(m\boxtimes n) \otimes \op{pt} \mapsto m\otimes n$ defines an isomorphism of $(R, R)$--modules
\[ H_0\big(X; A\boxtimes B \big) \to A \otimes_{\Z[\pi_1(X)]} B^\tr.\]
\end{lemma}
\begin{proof}
	The $(R,R)$-bimodule $H_0\big(X; A \boxtimes B \big)$ is known to coincide with the $(R,R)$-bimodule of coinvariants
$ A\boxtimes B \Big/ \Z \langle 1 -g \colon g \in \pi_1(X)\rangle$~\cite[Chapter VI.3]{HiltonStammbach}.
As~$A\boxtimes B $ carries the diagonal $\Z[\pi_1(X)]$-action (recall Remark~\ref{rem:Box}), we can identify~$a \boxtimes b$ with~$ag \boxtimes bg$, which is exactly the relation needed to build $A \otimes_{\Z[\pi_1(X)] }~B^\tr$. 
\end{proof}
\begin{remark}
	For finitely generated $(R,\Z[\pi])$-bimodules $M$ and $N$, the abelian group~$M\boxtimes N$ is not necessarily a finitely generated left $R$--module, even though it is finitely generated as an $(R, R)$--bimodule. For example, $\C \otimes_\Z \C$ is not a finite dimensional left $\C$--vector space.
\end{remark}

\subsection{The evaluation map}
\label{sub:EvaluationMaps}
We review evaluation maps of twisted cohomology classes on twisted homology classes.
References include~\cite{KirkLivingston,FriedlKim}.
\medbreak
Let~$\pi$ be a group and let $R$ be a commutative ring with involution. 
Let~$M,N$ be $(R, \Z[\pi])$--bimodules, and let $C$ be a chain complex of left $\Z[\pi]$--modules.
The following map is a chain map of chain complexes of left $R$--modules: 
\begin{align}
\otimes_{\Z[\pi]} N^\tr  \colon 
	\operatorname{Hom}_{\text{right-}\Z[\pi]}\left(C^\tr, M \right)
	&\rightarrow \operatorname{Hom}_{\text{right-}R}\left( (N \otimes_{\Z[\pi]}C)^\tr, M \otimes_{\Z[\pi]} N^\tr \right)\\
f &\mapsto  \Big(  n \otimes c \mapsto f(c) \otimes n \Big). \nonumber
\end{align}
By construction of the boundary maps of the dual complex, this induces an $R$--module homomorphism $\big( \otimes_{\Z[\pi]} N^\tr \big)_*$ in homology.
Consider furthermore the following canonical left $R$--module homomorphism:
\begin{equation} 
	\kappa \colon H^i\Big( \operatorname{Hom}_{\text{right-}R}\left( (N \otimes_{\Z[\pi]}C)^\tr, M \otimes_{\Z[\pi]} N^\tr \right)\Big) \rightarrow 
\Hom_{\text{right-}R}\left( H_i\left(  N \otimes_{\Z[\pi]} C \right)^\tr , M \otimes_{\Z[\pi]} N^\tr \right).
\end{equation}
The main definition of this subsection is the following. 
\begin{definition} \label{def:EvaluationMap}
Let~$\pi$ be a group, and let $R$ be a ring with involution. 
Let~$M,N$ be $(R, \Z[\pi])$--bimodules, and let $C$ be a chain complex of left $\Z[\pi]$--modules.
The \emph{evaluation map} is defined as the composition $\op{ev}_N:=\kappa \circ \big( \otimes_{\Z[\pi]} N^\tr \big)_*$:
\[ \op{ev}_N \colon H^i( \operatorname{Hom}_{\text{right-}\Z[\pi]}\left(C^\tr, M \right) ) \to\Hom_{\text{right-}R}\left( H_i\left(  N \otimes_{\Z[\pi]} C \right)^\tr , M \otimes_{\Z[\pi]} N^\tr \right).\]
\end{definition}

In practice, the module $M \otimes_{\Z[\pi]} N^\tr$ is not a convenient target for an evaluation pairing, and so we have to make additional assumptions. Let us assume that~$R = \F$ is a field and~$M = \F^d$. We consider the elements of~$\F^d$ as column vectors. The right action defines a homomorphism~$\rho_M \colon \pi \to \GL(\F, d)$ by the formula~$v \cdot g = \rho_M(g)^\tr v$. We say $\rho_M$ is \emph{unitary} if it takes values in~$\U(\F, d)$. Here~$\U(\F,d) \subset \GL(\F,d)$ consists of the matrices~$A$ with $A \cdot \ol A^\tr = \id_d$. 
By a slight abuse of notation, we also refer to the $(\F, \Z[\pi])$--bimodule~$M$ as a \emph{unitary representation}, if $\rho_M$ is unitary.

Define a sesquilinear pairing~$b \colon M \times M \to \F$ by~$b(v,w) = v^\tr \cdot \ol w$. Since $M$ is unitary, this pairing fulfills the relation~$b(v\cdot g, w\cdot g) = b(v,w)$, and so descends to a linear map~$b \colon M \otimes_{\Z[\pi]} M^\tr \to \F$. Also, note that the adjoint~$M \to \Hom_{\text{right}-\F} (M^\tr, \F)$ is an $\F$--isomorphism, that is $b$ is non-singular.

\begin{remark}
	The fact, that we only consider~$M = \F^d$ and the standard pairing~$b$ is not a restriction. Given an $(\F, \Z[\pi])$--bimodule~$M$ and a non-singular pairing~$b' \colon M \otimes_{\Z[\pi]} M^\tr \to \F$, apply Gram-Schmidt to find an isometry~$(M,b') \cong (\F^d, b)$.
Finally, when explicit reference to the pairing $b$ is needed, we will refer to $(M,b)$ as the unitary representation.
\end{remark}

We conclude this subsection by recalling the well known analogue of the universal coefficient theorem in the twisted setting.
\begin{lemma}\label{lem:nonSingularEval}
	Let~$(M,b)$ be a unitary representation, and let $C$ be a chain complex of left $\Z[\pi]$--modules.
Then the following composition is an isomorphism of left $\F$-vector spaces:
\begin{align*} b_* \circ \op{ev}_M \colon H^i \Big( \operatorname{Hom}_{\op{right-}\Z[\pi]}(C^\tr, M ) \Big) &\to\Hom_{\op{right-}\F}\Big( H_i\left(  M \otimes_{\Z[\pi]} C \right)^\tr , \F \Big) \\
	f &\mapsto \Big( m \otimes c \mapsto b(m,f(c) ) \Big).
\end{align*}
\end{lemma}
\begin{proof}
This is a consequence of the universal coefficient spectral sequence~\cite[Theorem 2.3]{LevineKnotModules} and the fact that the pairing $b$ is non-singular; we refer to~\cite[Theorem 5.4.4]{ConwayThesis} for details.
\end{proof}

\subsection{Twisted cup and cap products}\label{sub:Cup}
In this subsection, we fix our conventions on twisted cup and cap products.
\medbreak

Let $X$ be a space with universal cover $\widetilde{X}$, let $R$ be a commutative ring with involution, and let~$A,B$ be $(R,\Z[\pi])$-bimodules. 
Given an $(i+j)$-chain~$\sigma \in C_{i+j}(\widetilde{X};\Z)$, we use~$\lrcorner \sigma \in C_i(\widetilde{X};\Z)$ and~$\llcorner \sigma \in C_j(\widetilde{X};\Z)$ to denote the front-face $i$-chain and back-face $j$-chain of $\sigma$. 
They are defined on a singular simplex~$s_{i+j} \colon \Delta^{i+j} \to \widetilde X$ as follows and then extended linearly to chains: in barycentric coordinates, define projections~$p_\lrcorner \colon \Delta^{i+j} \to \Delta^i$ by $[v_0, \ldots, v_{i+j}] \mapsto [v_0, \ldots, v_i]$ and $p_\llcorner \colon \Delta^{i+j} \to \Delta^j$ by $[v_0, \ldots, v_{i+j}] \mapsto [v_i, \ldots, v_{i+j}]$. Now define~$\lrcorner s_{i+j} = s_{i+j} \circ p_\lrcorner$ and~$\llcorner s_{i+j} = s_{i+j} \circ p_\llcorner$.
Given twisted cochains $f \in C^i(X;A)=\Hom_{\text{right-}\Z[\pi]}(C_i(\widetilde{X})^\tr,A)$ and~$g \in C^j(X;B)= \Hom_{\text{right-}\Z[\pi]}(C_j(\widetilde{X})^\tr,B)$, a chain $\sigma \in C_{i+j}(\widetilde{X})$, and a twisted chain $b \otimes \sigma' \in C_k(X;B)$, one sets 
\begin{align*} 
 \big( f\cup g \big) (\sigma)&:=f( \lrcorner \sigma) \boxtimes g(   \llcorner \sigma )
	\in A\boxtimes B ,  \\
 f\cap (b \otimes \sigma' )  &:= \big( f( \lrcorner \sigma') \boxtimes b \big) \otimes  \llcorner \sigma' 
	\in C_{k-i}(X; A\boxtimes B). 
	\end{align*}
Using Remark~\ref{rem:Box}, a verification shows that $\cup$ and $\cap$ are $(R,R)$-bilinear.
It can be checked that $\cup$ and $\cap$ are chain maps and descend to (co)homology, leading to the main definition of this subsection. 

\begin{definition} \label{def:TwistedCup}
	Let $X$ be a space, and let $A,B$ be $(R,\Z[\pi])$-bimodules. The chain maps $\cup$ and $\cap$ defined above respectively induce the $(R,R)$-linear \emph{twisted cup product} and \emph{twisted cap product}  
\begin{align*}
\cup \colon H^i(X;A) \otimes_\Z H^j(X;B)^\tr \to H^{i+j}(X;A \boxtimes B), \\
\cap \colon H^i(X;A) \otimes_\Z H_k(X;B)^\tr \to H_{k-i}(X;A \boxtimes B).
\end{align*}
\end{definition}

We note that the relation $\big(\alpha \cup \beta \big) \cap \sigma = \alpha \cap \big( \beta \cap \sigma \big)$ holds between the cap and the cup product, as in the untwisted case~\cite[Theorem~5.2~(3)]{Bredon93}.

\subsection{Twisted Poincaré duality}\label{sub:PD}
Let $X$ be a closed $n$--dimensional manifold with fundamental group~$\pi$, and let~$M$ be an $(R, \Z[\pi])$--bimodule.
In this subsection, we recall Poincar\'e duality between $H^k(X; M)$ and $H_{n-k}(X;M)$. 
\medbreak
Let $\Z^\text{triv}$ be the $\Z[\pi]$--module with underlying abelian group~$\Z$ and equipped with the trivial action. 
Note that the fundamental class $[X]$ is a class in $H_n(X; \Z^\text{triv})$ and that $m \otimes a \mapsto am$ defines a natural identification of $M\boxtimes \Z^\text{triv}$ with $M$ as~$(R, \Z)$--bimodules.
We also write $[X] \in C_n(X)$ for the choice of a chain representative of the fundamental class.
The map $\cap [X] \colon C^\bullet(X;M) \to C_{n-\bullet}(X;M)$ is a ~$\Z[\pi]$-chain equivalence.
Tensoring with the~$(R,\Z[\pi])$-bimodule~$M$ leads to the following chain equivalence of left~$R$-modules:
\[ (\cap [X])_M \colon C^\bullet(X;M) \to C_{n-\bullet}(X;M).\]
This chain equivalence descends to homology, leading to the following definition.

\begin{definition}\label{def:twPoincare}
Let $X$ be a closed $n$--dimensional manifold with fundamental group~$\pi$, and let~$M$ be an $(R, \Z[\pi])$--bimodule.
\emph{Twisted Poincaré duality} is the inverse~$\PD_{X,M}$ of the isomorphism
\[ \cap [X] \colon H^k(X; M) \xrightarrow{\cong} H_{n-k}(X;M).\]
\end{definition}

The next lemma describes Poincar\'e duality when coefficients are changed.
\begin{lemma}\label{lem:PDCoeff}
If $\phi \colon M \rightarrow N$ is a morphism of $(R, \Z[\pi_1(X)])$--bimodules, then the following diagram of chain complexes of left $R$--modules commutes: 
\[ \begin{tikzcd}
C_\bullet(X; M) \ar[d,"\phi"]
& C^{n-\bullet}(X;M )  \ar[d,"\phi"] 	\ar[l, "\textcolor{black}{(\cap [X])_M}"]  \\
C_\bullet(X; N) & C^{n-\bullet}(X;N) \ar[l, "\textcolor{black}{(\cap [X])_N}"] .  
\end{tikzcd} \]
\end{lemma}
\begin{proof}
	This reduces to the fact $\otimes$ is a bifunctor, so that
\[ \left( \phi \otimes \id \right) \circ \left( \id_M \otimes \big(\cap [X] \big) \right) 
= \left( \id_N \otimes \big(\cap [X]\big) \right) \circ \left( \phi \otimes \id \right), \]
which is the commutativity of the diagram. 
\end{proof}

Next, we relate twisted Poincaré duality to Bockstein homomorphisms. 
\begin{lemma}\label{lem:BocksteinPD}
Let $X$ be a closed $n$-dimensional manifold, and let $A$, $B$ and~$C$ be~$(R, \Z[\pi])$--bimodules which fit into a short exact sequence
\[ 0 \to A \to B \to C \to 0. \]
If $\beta_{k+1} \colon H_{k+1}(X; C) \to H_{k}(X; A)$ and $\beta^k \colon H^{k}(X;C) \to H^{k+1}(X; A)$ denote the homology and cohomology Bockstein connecting homomorphisms, then the following diagram commutes:
	\[ \begin{tikzcd}
			H_{k+1}(X; C ) \ar[d, "\PD_{X, C}"]\ar[r, "\beta_{k+1}"] &  \ar[d, "\PD_{X, A}"] H_k(X; A)\\
			H^{n-k-1}(X; C )\ar[r, "\beta^{n-k}"] &  H^{n-k}(X; A).
	\end{tikzcd}\]
\end{lemma}
\begin{proof}
	The Poincar\'e duality chain equivalences fit into the commutative diagram
	\[ 
	\begin{tikzcd} 
		0 \ar[r]& C_\bullet(X; A)  \ar[r]& C_\bullet(X; B)
		\ar[r]& C_\bullet(X; C) \ar[r]& 0\\
0 \ar[r]& C^{n-\bullet}(X; A) \ar[r]\ar[u, "{(\cap [X])_A}"]& C^{n-\bullet}(X; B) \ar[r]\ar[u, "{(\cap [X])_B}"]& C^{n-\bullet}(X; C) \ar[r] \ar[u, "{(\cap [X])_C}"]& 0
	\end{tikzcd}
	\]
of chain complexes of left $R$-modules.
 Apply functoriality of the connecting homomorphisms to conclude the proof of the lemma.
\end{proof}

\section{The twisted Milnor pairing}
\label{sec:Milnor}

\subsection{Homological version}
\label{sub:TwistedMilnorHomological}
We define the twisted Milnor pairing using homology with coefficients in rings of power series.
\medbreak

Fix a field~$\F$ with an involution and consider the ring~$\F[t^{\pm 1}]$ of Laurent polynomials, which we abbreviate by~$\Lambda$. 
Recall that $\Lambda$ is equipped with the involution~$a_k t^k \mapsto \overline{a_k} t^{-k}$
for $a_k \in \F$ and $k \in \Z$.
Consider the $\Lambda$-module $\Gamma$ of power~series
\[ \Gamma := \F[[t^{\pm 1}]] = \left\{ \sum_{i = -\infty}^{\infty} a_i t^i \right\}.\]  
There are two rings~$\Lambda \subset \Gamma^\pm \subset \Gamma$, namely 
\begin{align}
\label{eq:GammaPlusMinus}
\Gamma^+ &= \left\{ p \in \Gamma : p = \sum_{i = k}^{\infty} a_i t^i \text{ for a } k \in \Z \right\},\\
\Gamma^- &= \left\{ p \in \Gamma : p = \sum_{i = -\infty}^{k} a_i t^i \text{ for a } k \in \Z\right\}. \nonumber
\end{align}

\begin{remark}\label{rem:GammaFields}
The rings~$\Gamma^+$ and $\Gamma^-$ are actually fields: one can use polynomial division to show that an element in~$\Gamma^+$ is invertible, if and only its lowest degree nonzero coefficient~$a_k \in \F$ is invertible~\cite[Section~II.1]{Lang85}. 
        Note that~$\Gamma^-$ is isomorphic to $\Gamma^+$ via the ring isomorphism~$t \mapsto t^{-1}$, and so $\Gamma^-$ is also a field.
Moreover, as right $\Lambda$--modules, the modules~$\Gamma_{\pm}$ are flat~\cite[Proposition 3.2]{Harvey}.
\end{remark}
We record the coefficient system that we shall use throughout this section. 

\begin{construction} \label{cons:LambdaCoeff}
Given a space $X$ and an epimorphism $\psi \colon \pi_1(X) \twoheadrightarrow \Z$, consider 
the cover~$q \colon X^\infty \rightarrow X$ associated to the kernel of~$\psi$.
Given an $(\F, \Z[\pi_1(X)])$--bimodule~$M$, consider the $(\Lambda,\Z[\pi_1(X)])$--bimodule~$M(\Lambda)$ with underlying left $\Lambda$-module~$\Lambda \otimes_\F M$, and where the right $\pi_1(X)$--action is diagonal and defined by
\[(p(t) \otimes m ) \cdot g := p(t) \cdot t^{\psi(g)} \otimes (m \cdot g),\] 
for $g \in G$, $m \in M$ and $p(t) \in \Lambda$. 
We consider the elements of $M(\Lambda)$ as polynomials with coefficients 
in the module~$M$, that is we use the shorthand~$\sum_k m_k t^k$ for~$\sum_k t^k \otimes~m_k$.
We define the $(\Lambda,\Z[\pi_1(X)])$-bimodules~$M(\Gamma^+),M(\Gamma^-)$ and~$M(\Gamma)$ analogously.
Note that~$M(\Gamma^+)$ and~$M(\Gamma^-)$ are flat over $\Lambda$: 
use that~$\Gamma^\pm$ is flat over~$\Lambda$ and that, since~$M \cong \F^d$, there is an isomorphism $M(\Gamma^\pm) \cong \big( \Gamma^\pm \big)^d$ of $\Lambda$-modules.
\end{construction}

Let $X$ be a closed $3$-manifold and let $(M,b)$ a unitary representation
such that the~$\Lambda$-module~$H_1(X;M(\Lambda))$ is torsion.
In order to define the twisted Milnor pairing, consider the $\F$--linear 
 composition 
\begin{align*}
	H_1(X;M(\Lambda)) 
&\xrightarrow{ \big( \beta^\Gamma \big)^{-1} } H_2(X; M(\Gamma)) \\
&\xrightarrow{ \PD_X } H^1(X; M( \Gamma)) \\
&\xrightarrow{ \op{ev}_{M(\Lambda)} } 
\Hom_{\text{right-}\Lambda}\left( H_1(X; M(\Lambda))^\tr , M(\Gamma) \otimes_{\Z[\pi_1(X)]} M(\Lambda)^\tr \right)\\
&\xrightarrow{\wt b_*} \Hom_{\text{right-}\Lambda}\left( H_1(X; M(\Lambda)) , \Gamma \right) \\
&\xrightarrow{( \op{const})_*} \Hom_{\text{right-}\F}\left( H_1(X; M(\Lambda)) , \F \right)
\tag{AlgMil} \label{eqn:MilnorPairing}
\end{align*}
of the five maps that we now define.
Firstly,~$\beta^\Gamma \colon H_{k+1}(X;M(\Gamma)) \to H_k(X;M(\Lambda))$ is the connecting $\Lambda$--homomorphism that arises from the following exact sequence of $\Lambda$--modules 
\[ 0 \to C_\bullet(X; M(\Lambda)) \to C_\bullet(X; M(\Gamma^-)) \oplus C_\bullet(X; M(\Gamma^+)) \to C_\bullet(X;M( \Gamma)) \to 0.\]
This sequence is exact since the $C_k(X; \Z[\pi])$ are free $\Z[\pi]$--modules and the sequence
$0 \to M(\Lambda) \xrightarrow{\bsm 1\\1 \esm} M(\Gamma^-) \oplus M(\Gamma^+) \xrightarrow{\bsm -1&1 \esm} M(\Gamma) \to 0$
is exact; $\beta^\Gamma$ is an isomorphism thanks to our assumptions on $X$ and $M$. 
Secondly, $\operatorname{PD}_X$ is the Poincar\'e duality~$\Lambda$--isomorphism described in Subsection~\ref{sub:PD}.
Thirdly, $\operatorname{ev}_{M(\Lambda)}$ is the $\Lambda$--linear evaluation map described in Subsection~\ref{sub:EvaluationMaps}.
Fourthly, $\wt b \colon M(\Gamma) \otimes_{\Z[\pi]} M(\Lambda)^\tr \to \Gamma$ is the $(\Lambda, \Lambda)$--homomorphism obtained by extending $b$ as
\begin{align*}
	\wt b	(m p(t), n q(t) ) := b(m, n) p(t)q(t^{-1}) .
\end{align*}
Fifthly, the map $\op{const} \colon \Gamma \to \F$ sends a series to 
its constant term: $\sum_k a_i t^i \mapsto a_0$.
This is the only $\F$-linear map in this composition that is not $\Lambda$--linear.

\begin{definition}\label{def:TwistedMilnorPairing}
Let $X$ be a closed $3$-manifold, and let $M$ be 
a unitary representation such that $H_1(X;M(\Lambda))$ is $\Lambda$-torsion.
The \emph{twisted Milnor pairing} 
\[ \widetilde{\mu}_{X,M} \colon H_1(X;M(\Lambda)) \times H_1(X;M(\Lambda)) \to \F \]
is the $\F$--sesquilinear pairing defined by the composition in \eqref{eqn:MilnorPairing}.
\end{definition}

In Corollary~\ref{cor:SkewSym} below, we will show that $\op{Mil}$ is non-singular and skew-Hermitian with respect to the involution of $\F$.

\subsection{The cohomological Milnor pairing with coefficients}\label{sub:TwistedMilnorCoHomological}
We provide a second definition of the Milnor pairing using twisted cup products.
This second interpretation will prove useful when relating Definition~\ref{def:TwistedMilnorPairing} to the definition given by Kirk and Livingston~\cite[Section 7]{KirkLivingston}.
\medbreak

Let $X$ be a closed $3$--manifold, and let $(M,b)$ be a 
unitary representation
such that~$H_1(X;M(\Lambda))$ is $\Lambda$-torsion.
Consider the $\F$--linear map obtained as the composition
\begin{align*}
	H^1(X; M(\Lambda)) \times H^1(X;M(\Lambda)) 
	&\xrightarrow{ \id \times (\beta_\Gamma)^{-1} } H^1(X;M(\Gamma)) \times H^2(X;M(\Lambda)) \\ 
	& \xrightarrow{\cup} H^3(X; M(\Gamma) \boxtimes M(\Lambda))\\
	& \xrightarrow{\PD_X} H_0(X;M(\Gamma) \boxtimes M(\Lambda))\\
	& \xrightarrow{\cong} M(\Gamma) \otimes_{\Z[\pi_1(X)]}M(\Lambda)^\tr\\
	& \xrightarrow{\wt b} \Gamma\\
	& \xrightarrow{\op{const}} \F \tag{AlgMilC}\label{eqn:algmilcohom}
\end{align*}
of the six maps defined as follows.
Firstly,~$\beta_{\Gamma} \colon H^{k+1}(X;M(\Gamma)) \to H^k(X;M(\Lambda))$ is the connecting $\Lambda$-homomorphism that arises from the following exact sequence of cochain complexes:
\[ 0 \to C^\bullet(X; M(\Lambda)) \to C^\bullet(X; M(\Gamma^-)) \oplus C^\bullet(X; M(\Gamma^+)) \to C^\bullet(X;M( \Gamma)) \to 0.\]
Secondly $\cup$ is the $(\Lambda,\Lambda)$-linear twisted cup product described in Subsection~\ref{sub:Cup}. 
Thirdly~$\operatorname{PD}_X$ is the Poincar\'e duality $\Lambda$-isomorphism defined in Subsection~\ref{sub:PD}.
Fourthly, we use the $\Lambda$-isomorphism which is described in Lemma~\ref{lem:Ccoinvariants}.
Fifthly, we use the~$(\Lambda,\Lambda)$-map~$\wt b \colon M(\Gamma)\otimes_{\Z[\pi]} M(\Lambda)\to \Gamma$ obtained as an extension of the pairing $b$.
Sixthly, we use the $\F$-linear constant map, which sends a power series to its constant term.

\begin{definition} \label{def:TwistedMilnorPairingCohomology}
Let $X$ be a closed $3$-manifold, and let $M$ be a 
unitary representation such that $H_1(X;M(\Lambda))$ is $\Lambda$-torsion.
The \emph{cohomological twisted Milnor pairing} 
\[ \widetilde{\mu}_{X,M}^\cup \colon H^1(X;M(\Lambda)) \times H^1(X;M(\Lambda)) \to \F\]
is the $\F$--sesquilinear pairing defined by the composition in \eqref{eqn:algmilcohom}.
\end{definition}

\subsection{Relating homological and cohomological pairings with coefficients} \label{sub:RelatingPairingsSeries}
The goal of this subsection is to relate the homological Milnor pairing of Definition~\ref{def:TwistedMilnorPairing} to the cohomological Milnor pairing of Definition~\ref{def:TwistedMilnorPairingCohomology}.
\medbreak

Fix a closed $3$-manifold $X$, and a 
unitary representation $(M,b)$ such that 
$H_1(X;M(\Lambda))$ is $\Lambda$-torsion.
The extension of~$b$ to~$M(\Gamma)$ is denoted by $\wt b \colon M(\Gamma) \otimes_{\Z[\pi]}  M(\Gamma)^\tr \to \Gamma$.
To relate the pairings $\widetilde{\mu}_{X,M}$ and $\widetilde{\mu}_{X,M}^\cup$, we define two $\F$-linear maps~$f$ and $g$.
Firstly,~$f$ is defined as the composition
\begin{align*}
f \colon	H^1(X; M(\Gamma)) 
&\xrightarrow{ \op{ev}_{M(\Lambda)} } 
	\Hom_{\text{right-}\Lambda}\left( H_1(X; M(\Lambda))^\tr , M(\Gamma) \otimes_{\Z[\pi_1(X)]} M(\Lambda)^\tr \right)\\
	&\xrightarrow{\wt b_*} \Hom_{\text{right-}\Lambda}\left( H_1(X; M(\Lambda)) , \Gamma \right)\\
	&\xrightarrow{( \op{const})_*} \Hom_{\text{right-}\F}\left( H_1(X; M(\Lambda) ) , \F \right).
\end{align*}
Secondly, $g$ is defined as the following $\F$-linear composition:
\begin{align*}
g \colon H^1(X;M(\Gamma)) \times H^2(X;M(\Lambda)) & \xrightarrow{\cup} H^3(X; M(\Gamma) \boxtimes M(\Lambda))\\
	& \xrightarrow{\PD_X} H_0(X;M(\Gamma) \boxtimes M(\Lambda))\\
	& \xrightarrow{\cong} M(\Gamma) \otimes_{\Z[\pi_1(X)]}M(\Lambda)^\tr\\
	& \xrightarrow{\wt b} \Gamma\\
	& \xrightarrow{\op{const}} \F.
\end{align*}

We start with a lemma that involves the evaluation map of Definition~\ref{def:EvaluationMap}, as well as the description of $H_0$ from Lemma~\ref{lem:BocksteinPD}.
\begin{lemma}\label{lem:EvalAndCup}
	The following diagram commutes: 
	\[ 
		\begin{tikzcd}
			H^1(X; M(\Gamma))\times H_1(X; M(\Lambda))^\tr \ar[rr, "\op{ev_{M(\Lambda)}}"] \ar[rrd, "\cap"] \ar[dd,"\id \times \PD_{X,M(\Lambda)}"] 
			& & M(\Gamma) \otimes_{\Z[\pi_1(X)]} M(\Lambda)^\tr\\
		        & & H_0(X; M(\Gamma) \boxtimes M(\Lambda)) \ar[u, "\cong"]\\
			H^1(X; M(\Gamma)) \times H^2(X; M(\Lambda))^\tr \ar[r,"\cup"] & H^3(X; M(\Gamma) \boxtimes M(\Lambda)) \ar[r,"\PD_X^{-1}"] & H_0(X; M(\Gamma)\boxtimes M(\Lambda) ). \ar[u, "="]
		\end{tikzcd}
	\]
    In particular, we have $f \circ (\id \times \PD_{X,M(\Lambda)})=g$.
\end{lemma}
\begin{proof}
	The lower pentagon of the diagram is commutative thanks to the relation~$\big(\alpha \cup \beta\big) \cap [X] = \alpha \cap \big( \beta \cap [X] \big)$ that we recalled in Subsection~\ref{sub:Cup}.
	The commutativity of the upper triangle amounts to the observation that, on the chain level, the cap product $\alpha \cap \big( m \otimes c \big) =  \big( \alpha(c) \boxtimes m \big) \otimes \op{pt}$ is identified with~$\alpha(c) \otimes m$; recall Lemma~\ref{lem:Ccoinvariants}. This also happens to be $\op{ev}(\alpha)\Big(m\otimes c\Big)$.
\end{proof}

The main result of this subsection is the following.
\begin{proposition}\label{prop:IsometryTwistedPairings}
	The map $h \colon H^1(X; M(\Gamma)) \to H_1(X; M(\Lambda))$ given by~$\PD_{X, M(\Lambda)} \circ (\beta_\Gamma)^{-1}$
	induces a $\Lambda$-isometry~$h^\#$ of the pairings $\widetilde{\mu}_{X,M}$ and $\widetilde{\mu}_{X,M}^\cup$.
\end{proposition}
\begin{proof}
	Consider the diagram
\[ \begin{tikzcd}
		H_1(X; M(\Lambda)) \times H_1(X; M(\Lambda)) \ar[d,"(\beta^\Gamma)^{-1} \times \id"]  \ar[ddr, out=0, in=90, "\widetilde{\mu}_{X,M}"]\\
		H_2(X; M(\Gamma)) \times H_1(X; M(\Lambda)) \ar[d,"\PD_{X, M(\Gamma)} \times \id"] \\
		H^1(X;M(\Gamma)) \times H_1(X; M(\Lambda))\ar[r, "f"] & \F. \\
		H^1(X; M(\Gamma)) \times H^2(X; M(\Lambda)) \ar[u, "\id \times \PD_{X, M(\Lambda)}"] \ar[ur, out=0, in=-135, "g"]\\
		H^1(X; M(\Gamma)) \times H^1(X; M(\Gamma)) \ar[u, "\id \times (\beta_\Gamma)^{-1}"] \ar[uur, out=0, in=-90, "\widetilde{\mu}_{X,M}^\cup"] \ar[uuuu, out=-180, in=180, "h\times h"]
\end{tikzcd} \]
The triangle involving $f$ and $g$ commutes by Lemma~\ref{lem:EvalAndCup}.
 The other two triangles on the right commute by definition of the pairings~$\widetilde{\mu}_{X,M}$ and $\widetilde{\mu}_{X,M}^\cup$.
 The left-hand side commutes by Lemma~\ref{lem:BocksteinPD}.

The map $h$ is a composition of isomorphisms. Thus, it is an isomorphism itself.
The commutativity of the diagram shows that $h$ respects the two pairings and is therefore an isometry. 
This concludes the proof of the proposition.
\end{proof}

\section{The Milnor pairing via infinite cyclic covers.}
\label{sec:GeometricMilnor}

\subsection{Ends and twisted (co)homology}\label{sub:Ends}
In this subsection, we study twisted homology groups of the infinite cyclic cover $X^\infty$ of a $3$--manifold $X$. 
More precisely, fix a field~$\F$, an $(\F, \Z[\pi_1(X)])$--bimodule~$M$, and suppose that $M$ has the property that $H_\bullet(X^\infty; M)$ is a finite dimensional $\F$--vector space.
By analogy with the homology with $\Gamma^+$, $\Gamma^-$ and $\Gamma$--coefficients, we consider locally finite twisted homology and introduce $\Lambda$--modules $H_\bullet^+(X^\infty;M)$, $H_\bullet^-(X^\infty;M)$ and~$H_\bullet^\lf(X^\infty;M)$.
The main result of this technical subsection is the fact that the modules~$H_\bullet^{+}(X^\infty;M)$ and~$H_\bullet^{-}(X^\infty;M)$ vanish; see Proposition~\ref{prop:PMAcyclic}.
This involves adapting arguments due to Milnor~\cite[Section 4]{MilnorInfiniteCyclic} to the twisted setting.

\medbreak
Fix a compact $3$--manifold~$X$ together with a surjection~$\pi_1(X) \twoheadrightarrow \Z$.
Pick a smooth map $f\colon X \rightarrow S^1$ inducing this homomorphism. 
The cover $X^\infty$ that corresponds to the kernel also has the description
as the pull-back of the diagram 
\begin{center}
	\begin{tikzpicture}
		\matrix (m) [matrix of math nodes, row sep=3em, column sep=4em
					, text height=1.5ex, text depth=0.25ex]
  		{
		     X^\infty & \R \\
		     X & S^1. \\
		};
		\path[->] (m-1-1) edge node[above] {$F$} (m-1-2);
		\path[->] (m-1-1) edge (m-2-1);
		\path[->] (m-1-2) edge node[right] {$\exp$} (m-2-2);
		\path[->] (m-2-1) edge (m-2-2);
	\begin{scope}[shift={($(m-1-1)+(0.5,-0.5)$)}]
                \draw +(-0.2,0) -- +(0,0) -- (0,0.2);
        \end{scope}
	\end{tikzpicture}
\end{center}

We describe the various homology groups that can be associated to $X^\infty$.
First, we can use the inclusion $\pi_1(X^\infty) \subset \pi_1(X)$ to restrict an~$(\F,\Z[\pi_1(X)])$-bimodule~$M$ to an $(\F, \Z[\pi_1(X^\infty)])$--bimodule, which we also denote by~$M$.

Write $t_g \colon X^\infty \to X^\infty$ for the deck transformation of $g \in \pi_1(X)$. 
Note that~$t_g$ only depends on the class of $g$ in~$\pi_1(X)/\pi_1(X^\infty) = \Z$. 
We use this action to endow the~$\F$-chain complex $C_\bullet(X^\infty;M)$ with the structure of an $\Lambda$-chain complex.
\begin{construction}\label{const:LambdaModule}
	Given an $(\F,\Z[\pi_1(X)])$-module $M$, we endow the $\F$-chain complex $C_\bullet(X^\infty; M)$ with a left $\Lambda$-module structure. 
For $m \otimes_{\Z[\pi_1(X^\infty)]} c \in C_\bullet(X^\infty; M)$, and~$g \in \pi_1(X)$, we define $t_g \colon C_\bullet(X^\infty;M) \to C_\bullet(X^\infty;M)$ by
	\[ g \cdot \big( m \otimes_{\Z[\pi_1(X^\infty)]} c \big) \mapsto m\cdot g^{-1} \otimes_{\Z[\pi_1(X^\infty)]} (g \cdot c).\] 
In order to see that this action is well defined, we use the fact that $\pi_1(X^\infty) \subset \pi_1(X)$ is a normal subgroup as well as the following computation:
	\begin{align*}
		g \cdot (mh^{-1} \otimes_{\Z[\pi_1(X^\infty)]} hc) 
		&=  m\cdot h^{-1}g^{-1} \otimes_{\Z[\pi_1(X^\infty)]} g h\cdot c \\
		&= m\cdot h^{-1}g^{-1} \otimes_{\Z[\pi_1(X^\infty)]} (ghg^{-1})g \cdot c) \\
		&= m\cdot h^{-1}g^{-1}(ghg^{-1}) \otimes_{\Z[\pi_1(X^\infty)]} (g \cdot c) \\
		&=m\cdot {g}^{-1} \otimes_{\Z[\pi_1(X^\infty)]} (g \cdot c)\\
		&=g \cdot (m \otimes_{\Z[\pi_1(X^\infty)]} c).
	\end{align*}
This left action descends to the quotient $\pi_1(X)/\pi_1(X^\infty)$ and so gives rise to a $\Z$--action.
Via this action, endow $C_\bullet(X^\infty; M)$ with a left $\Lambda$--module structure.
\end{construction}
\begin{remark}
Note that the $\F$-chain complex $C_\bullet(X^\infty; M)$ depends only on the structure of $M$ as a $(\F, \Z[\pi_1(X^\infty)])$--bimodule, whereas the $\Lambda$--action on $C_\bullet(X^\infty; M)$ requires that $M$ is in fact an $(\F, \Z[\pi_1(X)])$--bimodule.
\end{remark}

In Lemma~\ref{lem:weakshapiro} below, we will use Shapiro's lemma to relate the chain complex~$C_\bullet(X^\infty;M)$ described in Construction~\ref{const:LambdaModule} to the chain complex $C_\bullet(X;M(\Lambda))$ described in Construction~\ref{cons:LambdaCoeff}.

\begin{construction}\label{const:LambdaCohomology}
	We also equip the $\F$-cochain complex~$C^\bullet(X^\infty;M)$ with a $\Lambda$--action.
	For $\phi \in C^\bullet(X^\infty; M)$
	and $g \in \pi_1(X)$, define $(g \cdot \phi)$ by
	\[ (g \cdot \phi)(x) := \phi(g^{-1}x)g^{-1}. \]
	This action descends to $\pi_1(X)/\pi_1(X^\infty)$ and so gives a $\Z$--action, since $\phi$ is $\Z[\pi_1(X^\infty)]$--linear.
\end{construction}

The cover~$X^\infty$ is non-compact and so we now recall some of the natural chain complexes that appear in the study of such spaces. 
To the best of our knowledge such details have not been written down in the twisted case.
\medbreak
We will define chain complexes $C_\bullet^\lf(X^\infty;M)$, $C_\bullet^+(X^\infty;M) $, $C_\bullet^-(X^\infty;M)$ as limits of inverse systems: consider the category~$\N$, which has the natural numbers as objects and a unique map $a \rightarrow b$ if $a\leq b$.
 Thus, the category $\N^\text{op}$ looks like 
\[ \ldots \to 3 \to 2 \to 1 \to 0.\]
A \emph{diagram~$\{D(a)\}_a$ of type~$\N^\text{op}$}
in chain complexes is an inverse system of chain complexes, and we can consider its limit, a chain complex denoted by $\ilim \{ D(a) \}_a$, which is equipped with maps from the limit into the diagram.

Before we continue, we give a concrete description of the limits of these inverse systems in the category of modules to get the reader acquainted. 
As input, we have a collection~$f_{ij} \colon M_i \to M_j$ of module homomorphisms for $i \geq j \geq 0$ with $f_{ii}=\id$ and $f_{ik}=f_{jk} \circ f_{ij}$ for $i\geq j\geq k$, and the limit can be defined as:
\[ {\ilim}_i M_i = \{ (v_i) \in \prod_i M \colon f_{ij}(v_i) = v_j \} .\]

Note that $X^\infty$ comes with a sense of
negative and positive direction via the map~$F \colon X^\infty \to \R$. 
For each $t \in \R$, we then set
\begin{align*}
&L_t := F^{-1}(-\infty, -t), \\
&R_t := F^{-1}(t, \infty),
\end{align*}
and obtain a diagram of type~$\N^\text{op}$ by considering, for each $a \leq b$ in $\N$, the
inclusion induced map~$C_\bullet(X^\infty, L_b;M) \to C_\bullet(X^\infty, L_a; M)$. 
This map is well defined because the underlying map $X^\infty \to X^\infty$ is the identity, which is compatible with the twisted coefficients.
Similarly, define diagrams for the families~$R_t$ and $L_t \cup R_t$.
\begin{definition}\label{def:EndHomology}
Let~$M$ be an $(\F, \Z[\pi_1(X)])$--bimodule.
The chain complexes $C^\lf_\bullet(X^\infty;M)$, $C^+_\bullet(X^\infty;M)$, $C^-_\bullet(X^\infty;M)$ of $\F$--modules 
are defined as the following limits of the inverse systems constructed above:
\begin{align*}
C^-_\bullet(X^\infty;M) &= \ilim C_\bullet(X^\infty, L_t; M), \\
C^+_\bullet(X^\infty;M) &= \ilim C_\bullet(X^\infty, R_t; M), \\
C^\lf_\bullet (X^\infty;M) &= \ilim C_\bullet(X^\infty, L_t \cup R_t; M).
\end{align*}
\end{definition}

Informally, these chain complexes respectively allow for chains that tend towards one or two of the ends of $X^\infty$. 
We describe an equivalent definition of $C^\lf_\bullet (X^\infty;M) $.

\begin{remark}
\label{rem:FundamentalClassInfiniteCyclic}
	Note that $C^\lf_\bullet(X^\infty;M)$ is also the inverse limit of the complements of compact sets $K \subset X^\infty$, that is 
	\[ C^\lf_k(X^\infty; M) = \underset{K \subset X^\infty}{\ilim}  C_k(X^\infty, X^\infty \sm K; M). \]
	The corresponding homology groups are the locally finite homology groups, that is
\[ C_k^\lf(X^\infty; M) = \Big\{ \alpha = \sum_{i \in I} m\otimes_{\Z[\pi_1(X^\infty)]} c_i \in C_k(\widetilde{X};M) \colon \alpha \text{ locally finite}\Big\},\]   
where $\alpha$ is \emph{locally finite} if for every compact set $K$, $\{ i \in I \colon \pi^{\infty}(c_i) \cap K \neq \emptyset\}$ is finite,
and $\pi^\infty \colon \wt X \to X^\infty$ is the covering projection.
\end{remark}

We now equip these chain complexes with a $\Lambda$-module structure.

\begin{construction}\label{const:ICLambdaModule}
Recall from Construction~\ref{const:LambdaModule} that the chain complex~$C_\bullet(X^\infty; M)$ is a chain complex of~$\Lambda$--modules: an element~$g \in \pi_1(X)$ acts by
\[ g\cdot (m \otimes_{\Z[\pi_1(X^\infty)]} c) = m g^{-1} \otimes_{\Z[\pi_1(X^\infty)]} gc. \]
Now we consider the relative case of a pair $Y \subset X^\infty$. The chain map above induces a well defined chain map
\begin{align} \label{eq:Translation}
	t_g \colon C_k(X^\infty, Y; M) &\to C_k(X^\infty, t_g(Y); M)\\
	m \otimes c &\mapsto m g^{-1} \otimes gc. \nonumber
\end{align}	
Note that for elements $g \in \pi_1(X^\infty)$, the map $t_g$ is the identity.
By construction of the inverse limit, we have projections~$\pi_a$ as below, and can consider the composition
\[ C^-_k(X^\infty;M)={\ilim}_t C_k(X^\infty, L_t; M) \xrightarrow{\pi_a} C_k(X^\infty, L_a; M) \xrightarrow{t_g} C_k(X^\infty, t_g(L_a); M) \] 
for $a \in \N$. 
By the universal property of the inverse limit, we get an induced map~$t_g \colon C^-_k(X^\infty; M) \to  C^-_k(X^\infty; M)$. 
These maps~$t_g$ and the corresponding ones for the other complexes of Definition~\ref{def:EndHomology} endow the chain modules $C^\pm_k (X^\infty; M)$ and $C^\lf_k(X^\infty; M)$ with the structure of $\Lambda$--modules.
\end{construction}

Next, we discuss the analogous definitions in cohomology.
To do so, note that we obtain a diagram of type~$\N$ (and therefore a direct system) by 
considering the 
inclusion induced map~$C^\bullet(X^\infty, L_a; M) \to C^\bullet(X^\infty, L_b; M)$ for each $a \leq b$.

As for the limit, we quickly recall the colimit of directed systems in the category of modules. 
 Here, we have as input module homomorphisms~$f_{ij} \colon M_i \to M_j$ for $0\leq i \leq j$ with $f_{ii}=\id$ and $f_{ik}=f_{jk} \circ f_{ij}$ for $i \leq j \leq k$, and the colimit is
\[ {\colim}_i M_i = \bigoplus_i M_i/ \sim \text{, where } f_{ik}(v_i) \sim f_{jk}(v_j) \text{ for } v_i \in M_i, v_j \in M_j, k \geq i, k \geq j. \]
\begin{definition} \label{def:EndCoHomology}
Let~$M$ be a $(\F, \Z[\pi_1(X)])$--bimodule.
The cochain complexes $C_\cps^\bullet(X^\infty;M)$, $C_+^\bullet(X^\infty;M)$, $C_-^\bullet(X^\infty;M)$ of $\F$--vector spaces are defined as the following colimits:
\begin{align*}
C_-^\bullet(X^\infty;M) &= \clim C^\bullet(X^\infty, L_t; M), \\
C_+^\bullet(X^\infty;M) &= \clim C^\bullet(X^\infty, R_t; M), \\
C_\cps^\bullet(X^\infty;M) &= \clim C^\bullet(X^\infty, L_t \cup R_t; M).
\end{align*}
\end{definition}

We briefly discuss equivalent definitions of these cochain complexes.

\begin{remark}
\label{rem:CompactlySupportedCohomology}
	Similarly to Remark~\ref{rem:FundamentalClassInfiniteCyclic}, the cochain complex~$C_\cps^\bullet(X^\infty; M)$ can be viewed as the cochain complex of compactly supported cochains.
	Namely, taking the colimit over compact subsets of $X^\infty$, we have
    \[ C_\cps^\bullet(X^\infty; M)=\underset{K \subset X^\infty} \clim C^\bullet(X^\infty,X^\infty \setminus K; M).\]
Note furthermore that since taking (directed) colimits in $R$--modules is exact~\cite[Theorem~2.6.15]{Weibel94}, Definition~\ref{def:EndCoHomology} could have been made in cohomology instead of on the cochain complex level.
\end{remark}

Next, we endow these chain complexes with a $\Lambda$-module structure. 

\begin{construction}
\label{Cons:ModuleStructureCohomologyLimit}
As for homology (recall Construction~\ref{const:ICLambdaModule}), 
 the $\Lambda$--module structure on~$C^\bullet(X^\infty;M)$ from Construction~\ref{const:LambdaCohomology}, given by $(g \cdot \phi)(x) = \phi(g^{-1} \cdot x) g^{-1}$, passes to the direct limits and induces $\Lambda$-module structures on the chain complexes of Definition~\ref{def:EndCoHomology}.
\end{construction}

Applying  functoriality and the universal property of the (co)limit to the canonical inclusions allows us to define $\Lambda$-linear maps which fit into the following sequences:
\begin{align}
	0 \rightarrow C_\bullet(X^\infty; M) &\rightarrow C^-_\bullet(X^\infty;M) \oplus C^+_\bullet(X^\infty;M) 
	\rightarrow C^\lf_\bullet(X^\infty;M) \rightarrow 0,\label{eq:ShortHomology}\\
	0 \rightarrow C_\cps^\bullet(X^\infty; M) &\rightarrow C_-^\bullet(X^\infty;M) \oplus C_+^\bullet (X^\infty;M) 
	\rightarrow C^\bullet(X^\infty;M) \rightarrow 0.\label{eq:ShortCohomology}
\end{align}

\begin{remark}\label{rem:ConcreteLimOne}
Below, we take inverse limits of a short exact sequence of chain complexes. The sequence of limits will in general not be exact, and the defect is measured by the associated derived functor~$\limone$~\cite[Lemma 3.5.2]{Weibel94}. Although we have to refer the reader to the book of Weibel~\cite[Section~3.5]{Weibel94} for more details and more abstract properties, we now give a concrete construction of~$\limone_i M_i$ for a tower of modules~$f_{ij} \colon M_i \to M_j$ for $i\geq j \geq 0$. Consider the map
        \begin{align*} 
                \Delta \colon \prod_i M_i &\to \prod_i M_i\\
                \Delta(v_0, v_1, \ldots) &= (v_0 - f_{1,0} (v_1), \ldots, 
                v_i - f_{i+1,i}(v_{i+1}), \ldots).
        \end{align*}
Note that $\ker \Delta$ agrees on the nose with the previous description of $\ilim M_i$. One can define $\limone M_i = \op{coker} \Delta$.
\end{remark}

The next lemma proves that these sequences are exact.
\begin{lemma}\label{lem:PMShortExact}
	The sequences \eqref{eq:ShortHomology} and \eqref{eq:ShortCohomology} of $\Lambda$-modules are exact.
\end{lemma}
\begin{proof}
For every $t > 0$, we have the relative Mayer-Vietoris sequence 
\[ 0 \to C_\bullet(X^\infty; M) \to C_\bullet(X^\infty, L_t;M) \oplus C_\bullet(X^\infty, R_t;M) 
	\to C_\bullet(X^\infty, L_t\cup R_t;M) \to 0.\]
	Note that we used the fact $L_t \cap R_t = \emptyset$. 
We now pass to the limit.
The inverse limit is not an exact functor, and so we consider the associated derived functor~$\limone$~\cite[Lemma 3.5.2]{Weibel94}.
In our case, for each $k$, we thus obtain the following exact sequence of $\Lambda$-modules:
\[ 0 \rightarrow C_k(X^\infty; M) \rightarrow C_k^-(X^\infty;M) \oplus C_k^+(X^\infty;M) 
\rightarrow C_k^\lf(X^\infty;M) \rightarrow \limone C_{k+1}(X^\infty; M). \]
Since $C_{k+1}(X^\infty; M)$ is the constant inverse system, the module $\limone C_{k+1}(X^\infty; M)$ vanishes~\cite[Lemma~3.5.3]{Weibel94}.
This establishes the lemma for the first sequence.
Now we consider the second sequence. It is the colimit of the sequence of cochain complexes
\[ 0 \to C^\bullet(X^\infty, L_t \cup R_t; M) \to C^\bullet(X^\infty, L_t;M) \oplus C^\bullet(X^\infty, R_t;M) \to C^\bullet(X^\infty;M) \to 0.\]
Recall that taking (directed) colimits in $R$--modules is exact~\cite[Theorem~2.6.15]{Weibel94}, and thus the sequence in \eqref{eq:ShortCohomology} is exact as well.
\end{proof}

Our goal is now to show that the chain complexes $C_\bullet^\pm(X^\infty;M)$ and $C^\bullet_\pm(X^\infty;M)$ are acyclic.
As a first step however, we establish the following technical lemma.

\begin{lemma}[Milnor]\label{lem:FiniteGenerated}
If the $\F$--vector space $H_i(X^\infty;M)$ is
finite dimensional for every $i\in \N$, then the following assertions hold: 
\begin{enumerate}
	\item the $\F$-vector spaces $H_i(L_t \subset X^\infty;M)$ and $H_i(X^\infty, L_t;M)$ are finite dimensional;
	\item the $\F$-vector spaces~$H^i(X^\infty; M)$,  $H^i(L_t \subset X^\infty;M)$ and $H^i(X^\infty, L_t;M)$ are finite dimensional.
\end{enumerate}
\end{lemma}
\begin{proof}

The first assertion follows as in~\cite[p.125]{MilnorInfiniteCyclic}.
We now prove the second assertion.
Lemma~\ref{lem:nonSingularEval} shows that 
$H^k(X^\infty; M) \cong \op{Hom}_{\text{right-}\F} (H_k(X^\infty; M)^\tr, \F)$.
Since $H_k(X^\infty; M)$ is a finite dimensional $\F$--vector space,
so is $H^k(X^\infty; M)$. 
As in Milnor's argument, a Mayer-Vietoris sequence for the subspaces~$L_t$ and~$R_{-t-1}$, whose union is~$X^\infty$, and a relative sequence for the pair~$L_t \subset X^\infty$ show that the cohomology groups $H^i(L_t \subset X^\infty;M)$ and $H^i(X^\infty, L_t;M)$ are finite dimensional. 
\end{proof}

We can now prove that the chain complexes $C^-_\bullet(X^\infty; M)$ and $C^+_\bullet(X^\infty; M)$ of~$\F$-vector spaces are acyclic. 
An inverse system $\lbrace \phi_{ij} \colon A_j \to A_i \rbrace_{i \leq j}$ is \emph{Mittag-Leffler} if for all $k$, the image~$\phi_{ik}(A_k) \subset A_i$ stabilises~\cite[Definition 3.5.6]{Weibel94}.
If the image is eventually trivial, then the system is called \emph{trivially Mittag-Leffler}. 
Note that if an inverse system~$\lbrace \phi_{ij} \colon A_j \to A_i \rbrace_{i \leq j}$ is trivially Mittag-Leffler, then $\ilim A_i=0$.

\begin{remark}
\label{rem:IsMittagLeffler}
Observe that the inverse systems $C_\bullet(X^\infty, L_t; M)$, $C_\bullet(X^\infty, R_t; M)$ and $C_\bullet(X^\infty, L_t \cup R_t; M)$ of $\Lambda$-chain complexes are Mittag-Leffler.
Indeed, for each~$j \leq i$, the inclusion induced map~$\phi_{ij} \colon C_\bullet(X^\infty,L_i;M) \to C_\bullet(X^\infty,L_j;M)$ is a surjection.
The argument is identical for $C_\bullet(X^\infty, L_t \cup R_t; M)$.
\end{remark}

We can now prove the first main result of this subsection.

\begin{proposition}\label{prop:PMAcyclic}
The chain complexes~$C^-_\bullet(X^\infty;M)$ and $C^+_\bullet(X^\infty;M)$ are acyclic. 
\end{proposition}
\begin{proof}
We supplement the proof for $C^-_\bullet(X^\infty;M)$; the argument for~$C^+_\bullet(X^\infty;M)$ is identical.
	The proof of the next two claims follows the article of Milnor~\cite[p.125]{MilnorInfiniteCyclic}, but we adapt them for the reader's convenience.
\begin{claim} 
\label{claim:InclusionInducedIsZero}
For any $t \in \R$, there exists an integer $k > 0$ such that whenever $b > k$, then the inclusion induced map $H_i(X^\infty, L_t; M) \to H_i(X^\infty, L_{t-b};M)$ is zero.
\end{claim}
The $\F$-vector space $H_i(X^\infty, L_t; M)$ is finite dimensional by the first item of Lemma~\ref{lem:FiniteGenerated}.
We can thus pick a collection of cycles $\{c_1, \ldots, c_n \} \subset C_i(X^\infty; L_t;M)$ as $\F$-generators. 
For a suitably large integer $k > 0$, all the $c_i$ will be contained in~$C_i(L_{t-k}\subset X^\infty;M)$.
As a consequence, for all $b \geq k$, the inclusion induced map~$ H_i(X^\infty, L_t; M) \rightarrow H_i(X^\infty, L_{t-b};M)$ is the zero map, proving 
the claim. \claimbox

\begin{claim}
\label{claim:InclusionInducedIsZero2}
For any $t \in \R$, there exists an integer $k > 0$ such that whenever $b > k$, the inclusion induced map $H_i(X^\infty, L_{t+b}; M) \to H_i(X^\infty, L_t; M)$ is zero. 
Thus, the inverse system $H_i(X^\infty, L_t; M)$ is trivially Mittag-Leffler and $\ilim H_k(X^\infty, L_t; M) =~0$.
\end{claim}
Using the translation maps defined in~\eqref{eq:Translation}, we will reduce the first assertion to Claim~\ref{claim:InclusionInducedIsZero}.
Recall that $X$ is equipped with a map $f \colon X \to S^1$.
For a $t \in \R$, pick a $k$ as given by Claim~\ref{claim:InclusionInducedIsZero} and, since $f_* \colon \pi_1(X) \to \Z$ was assumed to be an epimorphism, choose~$g \in \pi_1(X)$ with $f_* (g) = k \in \Z$. 
Recall from~\eqref{eq:Translation} that the associated translation~$t_g$ is a map
\[ t_g \colon H_i(X^\infty, L_j; M) \to H_i(X^\infty, t_g(L_j); M),\]
and that $t_g(L_j) = L_{j-k}$ for any $j \in \R$. 
For a  given $b > k$, consider the diagram
\begin{center}
\label{eq:CommutTranslation}
\begin{tikzcd}
	H_i(X^\infty, L_{t+b-k}; M)  \ar[r, "0"]  & H_i(X^\infty, L_{t-k};M) \\
	H_i(X^\infty, L_{t+b}; M) \ar{u}{t_g}[swap]{\cong} \ar[r] & H_i(X^\infty, L_{t};M) \ar{u}{t_g }[swap]{\cong}.
\end{tikzcd}
\end{center}
The fact that this diagram commutes can already be checked on the chain level. 
The top map factors through $H_i(X^\infty, L_{t}; M) \to H_i(X^\infty, L_{t-k};M)$ which is zero by Claim~\ref{claim:InclusionInducedIsZero}, and so the top map is zero as well.
Therefore, by commutativity of the diagram, $H_i(X^\infty, L_{t+b}; M) \to H_i(X^\infty, L_{t};M)$ is also the zero map. 
This concludes the proof of the claim. \claimbox

We now slightly deviate from Milnor's original argument, since we work in homology instead of cohomology. 
To conclude the proof of the proposition, we have to argue that in our situation, the inverse limit commutes with homology, i.e. that~$\ilim H_k( X^\infty, L_t; M) = H_k ( C_\bullet^-(X^\infty; M))$.
The failure of this fact is measured by a Milnor exact sequence~\cite[Theorem 3.5.8]{Weibel94}, provided $C_\bullet(X^\infty, L_t;M)$ satisfies the Mittag-Leffler condition.

Since~$\lbrace C_k(X^\infty, L_t; M) \rbrace_{t \in \R}$ is indeed Mittag-Leffler for every $k$ (recall Remark~\ref{rem:IsMittagLeffler}),
an application of~\cite[Theorem 3.5.8]{Weibel94} results in the short exact sequence
\begin{equation}
\label{eq:MilnorSequence}
 0 \to \limone H_{k+1} (X^\infty, L_t;M) \to H_k\big( \ilim C_\bullet(X^\infty, L_t; M) \big) 
	\to  \ilim H_k(X^\infty, L_t; M)   \to 0.
	\end{equation}
By Claim~\ref{claim:InclusionInducedIsZero2}, the inverse system $H_{k+1} (X^\infty, L_t;M)$ is (trivially) Mittag-Leffler as the image stabilises to the zero submodule for $t$ sufficiently large. 
We deduce that~$\limone H_{k+1} (X^\infty, L_t;M) = 0$ \cite[Proposition 3.5.7]{Weibel94}.
The short exact sequence displayed in~\eqref{eq:MilnorSequence} and Claim~\ref{claim:InclusionInducedIsZero2} thus imply that
\[ H_k( C^-_\bullet(X^\infty; M) ) = \ilim H_k(X^\infty, L_t; M) = 0. \]
We have therefore established that $C^-_\bullet(X^\infty; M)$ is acyclic.
\end{proof}

Next, we prove the analogue of Proposition~\ref{prop:PMAcyclic} in cohomology.
\begin{proposition}\label{prop:PMAcyclicCohomology}
	The cochain complexes~$C_-^\bullet(X^\infty; M)$ and $C_+^\bullet(X^\infty; M)$ are acyclic. 
\end{proposition}
\begin{proof}
	We only consider the case~$C_-^\bullet(X^\infty;M)$, and fix a positive integer $k$.
	By Lemma~\ref{lem:nonSingularEval}, the evaluation induces an isomorphism
	\[ H^k(X^\infty, L_t; M) \xrightarrow{\sim} \Hom_{\text{right}-\F}\big( H_k(X^\infty, L_t; M)^\tr, \F\big). \]
	By Lemma~\ref{lem:FiniteGenerated}, the $\F$-vector space $H_k(X^\infty, L_t; M)$ is finite dimensional and so there exists finitely many $c_i \in H_k(X^\infty, L_t; M)$ with the following property: for every~$\alpha \in H^k(X^\infty, L_t;M)$ if $\alpha(c_i) = 0$ for each $i$, then $\alpha = 0$.
For instance, the $c_i$ can be taken to be a basis of the $\F$-vector space $H_k(X^\infty, L_t; M)$.

\begin{claim}
\label{claim:CohomVanishMap}
	For any $t \in \R$, there exists an integer $k > 0$ such that whenever $b > k$, then the map $H^k(X^\infty, L_{t-b}; M) \to H^k(X^\infty, L_{t};M)$ is zero.  
\end{claim}
For large enough $k$, the elements $c_i$ will lie in the kernel of the inclusion induced map $H_k(X^\infty, L_t; M) \to H_k(X^\infty, L_{t-b}; M)$. 
Consequently, any element~$\alpha$ in the image of $H^k(X^\infty, L_{t-b}; M) \to H^k(X^\infty, L_{t};M)$ will vanish on all the $c_i$, and therefore~$\alpha = 0$. 
This shows the claim. \claimbox

\begin{claim} 
\label{claim:CohomVanishMap2}
	For any $t \in \R$, there exists an integer $k > 0$ such that whenever $b > k$, the map $H^k(X^\infty, L_{t}; M) \to H^k(X^\infty, L_{t+b};M)$ is zero. 
\end{claim}
We proceed as above in the proof for the corresponding statement in homology (recall Claim~\ref{claim:InclusionInducedIsZero2}): the deck transformations allow us to translate Claim~\ref{claim:CohomVanishMap} to the desired statement. \claimbox

Claim~\ref{claim:CohomVanishMap2} implies that $\clim H^k(X^\infty, L_{t};M) = 0$.
As taking the colimit of a directed system of $R$--modules is exact~\cite[Theorem~2.6.15]{Weibel94}, we obtain that $H^k(C_-^\bullet(X^\infty; M) )=\clim H^k(X^\infty, L_{t};M)=~0$ and the proposition is established.
\end{proof}

\subsection{The homological geometric twisted Milnor pairing}
\label{sub:TwistedMilnorGeometricHomological}
We give another definition of the twisted Milnor pairing.
This approach is a twisted generalisation of an outline provided by Litherland~\cite[Appendix B]{LitherlandCobordism}: it relies on the infinite cyclic cover of $X$ instead of on the use of rings of power series.
\medbreak

In order to define the Milnor pairing using the (co)homology of $X^\infty$, we first review a non-compact version of Poincar\'e duality.
Assume that $X$ is a closed $n$-manifold; in short order we will return to the case $n=3$.
We describe a fundamental class representative $[X^\infty] \in C^\lf_n(X^\infty; \Z^\text{triv})$ for the infinite cyclic cover~$X^\infty$.
This class is obtained from a representative $[X] = \sum_i 1\otimes c_i \in C_n(X; \Z^\text{triv})$ as follows:
\begin{equation}
\label{eq:FundamentalCycle}
[X^\infty] := \sum_{i, [g]} 1\otimes g^{-1}\cdot c_i .
\end{equation}
This is indeed a fundamental class in the usual sense, since for each point $x \in X^\infty$, the class $[X^\infty]$ is sent to the positive generator of $H_n(X^\infty, X^\infty \sm \{x\}; \Z^\text{triv})$.
As we shall see in Proposition~\ref{prop:LocallyFinitePD} below, the $\F$-linear cap product
\[ \cap [X^\infty] \colon C^k(X^\infty; M) \xrightarrow{\sim} C_{n-k}^\lf (X^\infty; M) \]
induces an isomorphism in (co)homology. 
Temporarily taking this result for granted, we make the following definition.
\begin{definition}\label{def:PoincareLF}
The \textit{twisted locally finite Poincaré duality isomorphism}
\[ \PD^\lf_{X^\infty; M} \colon H_k^\lf(X^\infty;M) \to H^{n-k}(X^\infty; M)\] 
is the inverse of the $\F$-linear isomorphism $\cap [X^\infty]$.
\end{definition}
We now prove that $\cap [X^\infty]$ is indeed an isomorphism. The proof is a generalisation of \cite[Theorem~3.1]{Laitinen96}.

\begin{proposition}\label{prop:LocallyFinitePD}
	The cap product $\cap [X^\infty] \colon H^k(X^\infty; M) \to H_{n-k}^\lf(X^\infty; M)$ is an $\F$-isomorphism.
\end{proposition}
\begin{proof}
Recall that $F \colon X^\infty \to \R$ denotes a lift of the map $f \colon X \to S^1$.
For regular values $t > 0$ of $F \colon X^\infty \to \R$, consider the $n$--dimensional submanifolds $Y_t = F^{-1}([-t,t])$ and their interiors~$\Int Y_t = F^{-1}((-t,t))$. We will reduce the statement of the proposition to Poincaré duality on~$Y_t$.

Note that we have $Y_t \subset Y_{t+1}$ and $\bigcup_t Y_t=X^\infty$. 
We abbreviate the chain complex~$C_\bullet(X^\infty, X^\infty \sm \Int Y_t; M)$ by $C_\bullet(X^\infty | \Int Y_t;M)$. Recall that $C_k^\lf(X^\infty; M)$ is the inverse limit of~$C_k(X^\infty | \Int Y_t; M)$, and denote the corresponding maps by
\[ q_{Y_t} \colon C_k^\lf(X^\infty; M) \to C_k(X^\infty | \Int Y_t; M). \]
For every~$t>0$, the cap product gives rise to the commutative diagram 
\begin{equation}\label{eq:CapLFCommtative}
\begin{tikzcd}
	C^k(X^\infty; M) \times C^\lf_n(X^\infty; \Z^\triv) \ar{d}{\id \times q_{Y_t}}  \ar{r}{\cap} & C^\lf_{n-k}(X^\infty; M) \ar{d}{q_{Y_t}}\\
	C^k(X^\infty; M) \times C_n(X^\infty| \Int Y_t; \Z^\triv)  \ar{r}{\cap} & C_{n-k}(X^\infty| \Int Y_t; M).
\end{tikzcd}
\end{equation}
Take the homology of this diagram, and denote the image of $[X^\infty]$ by 
\[ [X^\infty]_t := \big( q_{Y_t} \big)_* [X^\infty] \in H_n(X^\infty|\Int Y_t; M). \] 
Diagram~\eqref{eq:CapLFCommtative} shows that (in homology) 
for every $\alpha \in H^k(X^\infty; M)$, we have
\[ \big( q_{Y_t}\big)_*\left(\alpha \cap [X^\infty]\right)  = \alpha \cap [X^\infty]_t .\]
Now consider the cap product~$\cap [X^\infty]_t \colon H^k(X^\infty; M) \to H_{n-k}(X^\infty | Y_t; M)$ for varying $t$: 
\begin{equation}\label{eq:CapLFCommtativeHomology}
	\begin{tikzcd}[column sep = 2cm]
	H^k(X^\infty; M) \ar{d}{\id}  \ar{r}{\cap [X^\infty]} & H^\lf_{n-k}(X^\infty; M) \ar{d}{\big( q_{Y_t}\big)_*}\\
	H^k(X^\infty; M)  \ar{r}{\cap [X^\infty]_t} & H_{n-k}(X^\infty| \Int Y_t; M).
\end{tikzcd}
\end{equation}

As \eqref{eq:CapLFCommtativeHomology} is compatible with the map~$(X^\infty, X^\infty \sm \Int Y_{t'}) \subset (X^\infty, X^\infty \sm \Int Y_{t})$ of pairs for $t < t'$, we can take the inverse limit of the bottom row, which results in the commutative diagram
\[
	\begin{tikzcd}[column sep = 2cm]
		H^k(X^\infty; M) \ar{d}{\id}  \ar{r}{\cap [X^\infty]} & H^\lf_{n-k}(X^\infty; M) \ar{d}{\ell_{n-k}}\\
	H^k(X^\infty; M) \ar{r}{ \ilim \big( \cap [X^\infty]_t \big)} & \ilim H_{n-k}(X^\infty| \Int Y_t; M),
\end{tikzcd}
\]
where we have made the abbreviation~$\ell_{n-k} := \ilim \big( q_{Y_t}\big)_*$.

Next, recall the Milnor exact sequence; see e.g.~\cite[Theorem~3.5.8]{Weibel94},
and fit it into the following commutative diagram:
\[
	\begin{tikzcd}
		& H^k(X^\infty; M) \ar{d}{\cap [X^\infty]} \ar{dr}{\ilim \big( \cap[X^\infty]_t \big)} &\\
		\limone H_{n-k+1}(X^\infty|  \Int Y_t; M) \ar[r] &  H^\lf_{n-k}(X^\infty; M) \ar[r,"\ell_{n-k}"] & \ilim H_{n-k}(X^\infty| \Int Y_t; M) \ar[r] & 0. 
	\end{tikzcd}
\]
By excision,~$H_{n-k+1}(X^\infty|  \Int Y_t; M) \cong H_{n-k+1}(Y_t, \partial Y_t; M)$, and since the space~$Y_t$ is a compact submanifold, the later $\F$-vector space is finite dimensional.
Since for every~$t > 0$, the vector space $H_{n-k+1}(X^\infty|  \Int Y_t; M)$ is finite dimensional, the term~$\limone H_{n-k+1}(X^\infty|  \Int Y_t; M)$ vanishes~\cite[Exercise~3.5.2]{Weibel94}. 
Since $\ell_{n-k}$ is an isomorphism, it remains to show that~$\ilim \big(\cap [X^\infty]_t\big)$ is an isomorphism. 

The compact submanifold $Y_t \subset X^\infty$ has codimension~$0$ and its fundamental class~$[Y_t]$ agrees with the image of $[X^\infty]$ under the map
\begin{align*}
H_n^\lf(X^\infty; \Z^{\op{triv}}) \xrightarrow{{q_{Y_t}}_*} H_n(X^\infty | \Int Y_t;\Z^{\op{triv}} )  \xrightarrow{\op{exc}} H_n(Y_t, \partial Y_t; \Z^{\op{triv}}),
\end{align*}
where $\op{exc}$ denotes the excision isomorphism.
Note also that by definition of $[X^\infty]_t$, we have $\op{exc} \circ {q_{Y_t}}_*([X^\infty])=\op{exc}([X^\infty]_t)=[Y_t]$.
Consequently, the following diagram is commutative
\begin{equation}
\label{eq:Defft}
 \begin{tikzcd}
	H^k (X^\infty; M) \ar[d] \ar{r}{\cap [X^\infty]_t} & H_{n-k}(X^\infty | \Int Y_t; M) \ar{d}{\op{exc}}[swap]{\cong}\\
	H^k (Y_t \subset X^\infty; M)\ar{r}{\cap [Y_t]}[swap]{\cong} \ar[dashed]{ur}{f_t} & H_{n-k}(Y_t \subset X^\infty, \partial Y_t; M),
\end{tikzcd}
\end{equation}
where $f_t:=\op{exc}^{-1} \circ \big(\cap [Y_t]\big)$.
The map~$\cap [Y_t]$ is an isomorphism by Poincaré duality for the compact manifold~$Y_t$, and since excision is an isomorphism, the map~$f_t$ is also an isomorphism. 

We apply the Milnor exact sequence to the chain module~$C^k(X^\infty;M)$, which is~$\Hom_\F(\clim C_k(Y_t \subset X^\infty;\Z[\pi]),M))\cong \ilim\Hom_\F(C_k(Y_t \subset X^\infty;\Z[\pi])^\tr,M))$.
In other words, we view $C^k(X^\infty;M)$ as $\ilim C^k(Y_t \subset X^\infty;M)$.
We argue as in homology that~$\limone H^{k-1}(Y_t\subset X^\infty; M) = 0$.
Consequently, the Milnor exact sequence shows that the following map is an isomorphism:
\[ \ell^k \colon H^k(X^\infty; M) \to \ilim H^k(Y_t \subset X^\infty; M).\] 
Using the diagram in~\eqref{eq:Defft} and the universal property of the inverse limit, we obtain the following commutative diagram:
\[
	\begin{tikzcd}
		H^k(X^\infty; M) \ar{r}{\ilim \big(\cap [X^\infty]_t\big)} 
		\ar{d}{\cong}[swap]{\ell^k} & 
\ilim H_{n-k}(X^\infty| \Int Y_t; M)    \\
		\ilim H^k(Y_t \subset X^\infty; M) \ar{ru}{\cong}[swap]{\ilim f_t}  . 
	\end{tikzcd}
\]
The map $\ilim f_t$ is an isomorphism, since each map $f_t$ is an isomorphism. Deduce from the diagram that $\ilim \big( \cap [X^\infty]_t\big)$ is an isomorphism. Above, we have seen that this implies that $\cap [X^\infty]$ is also an isomorphism, which proves the proposition.
\end{proof}

We record the corresponding isomorphism for cohomology with compact support.
\begin{remark}
\label{rem:PDCompactlySupported}
We construct the compactly supported Poincar\'e duality $\F$-isomorphism 
\[\PD_{X^\infty,\op{cs}}  \colon  H_{n-k}(X^\infty;M) \to H^k_{\op{cs}}(X^\infty;M).\]
As in Proposition~\ref{prop:LocallyFinitePD}, we set~$C_\bullet(X^\infty | \Int Y_t;M):=C_\bullet(X^\infty, X^\infty \sm \Int Y_t; M)$, and similarly in cohomology.
According to~\cite[Lemma 3.27]{Hatcher}, there is a fundamental class in $C_\bullet(X^\infty | \Int Y_t;\Z^{\op{triv}})$, which we denote $[X^\infty | \Int Y_t]$.
The cap product with~$\cap [X^\infty | \Int Y_t] $ yields a map
\[ \cap [X^\infty | \Int Y_t] \colon C^k(X^\infty | \Int Y_t;M) \to C_{n-k}(X^\infty;M).\]
As in~\cite[p.245]{Hatcher}, it can be verified that these cap products are compatible with the inclusions $Y_a\subset Y_b$ for $a \leq b$.
Taking the colimit, similar arguments as in the proof of Proposition~\ref{prop:LocallyFinitePD} show that the following map is an isomorphism:
\[
 \clim \cap  [X^\infty | \Int Y_t] \colon H^k_{\op{cs}}(X^\infty|Y_t;M) \to H_{n-k}(X^\infty;M).
\]
Here, taking the colimit in (co)homology is justified since colimits are exact.
We define the $\F$-isomorphism $\PD_{X^\infty,\op{cs}} $ as the inverse of $\limd \cap  [X^\infty | \Int Y_t]$.
\end{remark}

Let $X$ be a closed $3$-manifold, and let~$(M,b)$ be 
a unitary representation such that the $\F$-vector space $H_1(X^\infty;M)$ is finite dimensional.
In order to define a twisted Milnor pairing, consider the $\F$--linear map obtained as the composition
\begin{align*}
	H_1(X^\infty;M) 
	&\xrightarrow{\big(\beta^\lf\big)^{-1}} H_2^\lf(X^\infty;M) \\
	&\xrightarrow{\PD_{X^\infty}^\lf} H^1(X^\infty;M) \\
	&\xrightarrow{\op{ev}_M} \text{Hom}_{\text{right-}\F}( H_1(X^\infty;M)^\tr,M\otimes_{\Z[\pi^\infty]}M^\tr)\\
	&\xrightarrow{\op{quot}_*} \text{Hom}_{\text{right-}\F}( H_1(X^\infty;M)^\tr,M\otimes_{\Z[\pi]}M^\tr)\\
	&\xrightarrow{b_*} \text{Hom}_{\text{right-}\F}( H_1(X^\infty;M)^\tr, \F). \tag{GeoMilnor}\label{eq:GeoMilnor}
\end{align*}
of the five maps defined as follows.
Firstly, $\beta^\lf$ is the connecting $\F$--homomorphism that arises from the short exact sequence 
\[ 0 \to C_\bullet(X^\infty,M) \to C_\bullet^-(X^\infty,M) \oplus C_\bullet^+(X^\infty,M) \to C_\bullet^{\operatorname{lf}}(X^\infty,M) \to 0 \]
described in Lemma~\ref{lem:PMShortExact}; this Bockstein homomorphism is invertible thanks to Proposition~\ref{prop:PMAcyclic}.
Secondly, recall that the map $\PD_{X^\infty}^\lf$ is the twisted locally finite Poincar\'e duality $\F$-isomorphism from Definition~\ref{def:PoincareLF}.
Thirdly, $\op{ev}_M$ is the~$\F$-linear evaluation map described in Subsection~\ref{sub:EvaluationMaps}.
Fourthly, $\op{quot}_*$ is the $\F$-linear map induced by the quotient map $M\otimes_{\Z[\pi_1(X^\infty)]} M^\tr \to M\otimes_{\Z[\pi_1(X)]} M^\tr$. 
Fifthly and finally, $b_*$ is the $\F$-linear map induced by the pairing $b \colon M \otimes_{\Z[\pi_1(X)]} M^\tr \to \F$.

We can now extend Litherland's approach~\cite{LitherlandCobordism} to the twisted setting.
\begin{definition}
\label{def:GeometricTwistedMilnor}
Let $X$ be a closed $3$-manifold, and let $M$ be 
a unitary representation such that $H_1(X^\infty;M)$ is a finite dimensional $\F$-vector space.
The \emph{geometric twisted Milnor pairing}
\[ \mu_{X,M}^\infty \colon H_1(X^\infty;M) \times H_1(X^\infty;M) \rightarrow \F\]
is the $\F$--sesquilinear pairing defined by the composition in \eqref{eq:GeoMilnor}.
\end{definition}

\subsection{Relating the two homological pairings}
\label{sub:RelatingHomologicalPairings}
The goal of this subsection is to show that the twisted Milnor pairing of Definition~\ref{def:TwistedMilnorPairing} is isometric to the geometric twisted Milnor pairing of Definition~\ref{def:GeometricTwistedMilnor}.
\medbreak
Throughout this subsection, we fix a closed $3$--manifold $X$, and an epimorphism~$\psi \colon \pi_1(X) \twoheadrightarrow \Z$.
The first step to construct the isometry between the two Milnor pairings is to define the underlying isomorphism of $\Lambda$-modules.
Namely, we need to relate the $\Lambda$-module $H_k(X; M(\Lambda))$ described in Construction~\ref{cons:LambdaCoeff} with the $\Lambda$-module $H_k(X^\infty; M)$ of Construction~\ref{const:LambdaModule}.
This isomorphism is established in~\cite[Theorem 2.1]{KirkLivingston} but we recall the explicit map for later use.

\begin{lemma}[Shapiro]\label{lem:weakshapiro}
Let $M$ be a $(\F, \Z[\pi_1(X)])$--bimodule. 
The following map is an isomorphism of $\Lambda$--modules:
\begin{align*}
	\op{dst} \colon C_\bullet(X; M(\Lambda) ) &\rightarrow C_\bullet(X^\infty;M)\\
	 m t^{\psi(g)} \otimes_{\Z[\pi_1(X)]} c &\mapsto m\cdot g^{-1} \otimes_{\Z[\pi_1(X^\infty)]} (g \cdot c),
\end{align*}
\end{lemma}
\begin{proof}
This assignment is well-defined: first note, that by Construction~\ref{const:LambdaModule}, the right-hand side is exactly~$g \cdot (m \otimes_{\Z[\pi_1(X^\infty)]} c)$, which only depends on the coset~$\psi(g) \in \Z = \pi_1(X)/\pi_1(X^\infty)$. 
        Secondly, we have to verify that the assignment above is well-defined on the tensor product~$C_k(X; M(\Lambda))$. Recall the action on $M(\Lambda)$ from Construction~\ref{cons:LambdaCoeff} and note that 
        \[ m t^{\psi(g)} \otimes_{\Z[\pi_1(X)]} c = \big( (m \cdot g^{-1}) t^0\big) \cdot t^{\psi(g)} \otimes_{\Z[\pi_1(X)]} c = (m \cdot g^{-1}) t^{\psi(e)} \otimes_{\Z[\pi_1(X)]} g\cdot c. \]
Since the formula for $\op{dst}$ sends both sides of the equation to the same element, the map~$\op{dst}$ is indeed well-defined.
This also shows that the map is $\Z$--equivariant, and thus a $\Lambda$--module homomorphism.
To conclude that $\op{dst}$ is an isomorphism, verify that the inverse map is $m \otimes_{\Z[\pi_1(X^\infty)]}  c \mapsto m t^0 \otimes_{\Z[\pi_1(X)]}  c$.
\end{proof}

The $\Lambda$--chain isomorphism $\op{dst} \colon C_\bullet(X;M(\Lambda)) \to C_\bullet(X^\infty;M)$, which is described in Lemma~\ref{lem:weakshapiro}, extends to a $\Lambda$--chain map 
\begin{align*}
	\op{dst}^\lf \colon C_\bullet(X; M(\Gamma) ) &\xrightarrow{} C_\bullet^\lf(X^\infty;M)\\
	\Big( \sum  m t^{\psi(g)}\Big) \otimes_{\Z[\pi_1(X)]} c &\mapsto \sum m\cdot g^{-1} \otimes_{\Z[\pi_1(X^\infty)]} (g \cdot c).
\end{align*}
Indeed, the formula results in a locally finite chain, since for any pairs of compact sets $K, K' \subset X^\infty$, the set $\{ \ell \in \Z \colon \ell \cdot K' \cap K \neq \emptyset \}$ is finite. 
The next lemma shows that $\op{dst}^\lf$ induces a $\Lambda$--isomorphism on homology.

\begin{lemma}\label{lem:UndlerlyingIsoHomology}
	The map $\op{dst}^\lf \colon H_k(X; M(\Gamma)) \to H_k^\lf(X^\infty; M)$ is a $\Lambda$-isomorphism.
\end{lemma}
\begin{proof}
	The $\Lambda$-chain map $\op{dst}$ also extends to $\Lambda$-chain maps $C_\bullet(X; M(\Gamma^+)) \to C_\bullet^+(X; M)$ and $C_\bullet(X; M(\Gamma^-)) \to C_\bullet^-(X; M)$ that fit into the following commutative diagram of chain complexes:
\begin{equation} \label{eq:ChainComplexCommute}
	\begin{tikzcd}
		0\ar[r] & C_\bullet(X; M(\Lambda)) \ar[d,"\op{dst}"] \ar[r] & C_\bullet(X; M(\Gamma^-)) \oplus C_\bullet(X; M(\Gamma^+)) \ar[d,"\op{dst}^\pm"] \ar[r] & C_\bullet(X; M(\Gamma))\ar[r] \ar[d,"\op{dst}^\lf"] & 0\\
		0\ar[r] & C_\bullet(X^\infty; M)\ar[r] & C^-_\bullet(X^\infty; M) \oplus C^+_\bullet(X^\infty; M)\ar[r] & C^\lf_\bullet(X;M)\ar[r] & 0.
	\end{tikzcd}
\end{equation}
Here, the bottom row is the exact sequence constructed in \eqref{eq:ShortHomology},
and the top row is induced by the short exact sequence $0 \to M(\Lambda) \to M(\Gamma^-) \oplus M(\Gamma^+) \to M(\Gamma) \to 0$ of coefficients. 
Each row gives rise to connecting homomorphisms, which were denoted by~$\beta^\lf$ and~$\beta^\Gamma$ in Subsections~\ref{sub:TwistedMilnorHomological} and~\ref{sub:TwistedMilnorGeometricHomological}.
With these notations, the diagram displayed~\eqref{eq:ChainComplexCommute} gives rise to the following commutative square in homology:
\begin{equation}\label{eq:BocksteinCommute}
	\begin{tikzcd}
		0 & \ar[l] H_{k}(X; M(\Lambda)) \ar[d,"\op{dst}"]  & \ar[swap,l,"\beta^\Gamma"] H_{k+1}(X; M(\Gamma)) \ar[d,"\op{dst}^\lf"] & \ar[l] 0\\
		0 & \ar[l] H_{k}(X^\infty; M) & \ar[swap,l, "\beta^\lf"] H^\lf_{k+1}(X^\infty;M)& \ar[l] 0.
	\end{tikzcd}
\end{equation}
To obtain the left and right zeroes, we used Proposition~\ref{prop:PMAcyclic} and the fact that $\Gamma^+$ and $\Gamma^-$ are flat over $\Lambda$.
Since $\op{dst}$ is a $\Lambda$-isomorphism (by Lemma~\ref{lem:weakshapiro}), so is $\op{dst}^\lf$.
This concludes the proof of the lemma.
\end{proof}

In order to prove a cohomological version of Lemma~\ref{lem:UndlerlyingIsoHomology}, we introduce the cohomological analogue of the chain map $\operatorname{dst}$.
This is the $\Lambda$-chain map $\psi_\Gamma$, defined~by
\begin{align}\label{eq:PsiGamma}
	\psi_\Gamma \colon C^i(X^\infty;M) &\rightarrow C^i(X;M(\Gamma)) \\
\phi & \mapsto \Big( x  \mapsto \sum_{[g]} \big( g^{-1} \cdot \phi \big)(x)\, t^{\psi(g)} \Big), \nonumber
\end{align}
where the sum runs over the cosets~$\pi_1(X)/\pi_1(X^\infty)\cong \Z$.
We observe that the inverse~$g^{-1}$ in the expression on the right-hand side is forced upon us, since $\psi_\Gamma(\phi)(x)$ has to be $\pi_1(X)$--linear in $x$: 
\begin{align*}
\Psi_\Gamma(\phi)(h^{-1} \cdot x) 
&= \sum_{[g]} \big( g^{-1} \cdot \phi \big)(h^{-1}x) \, t^{\psi(g)} \\
&=\sum_{[g]} \big( g^{-1} \cdot \phi \big)(h^{-1}x)h^{-1}h \, t^{\psi(g)}\\
&=\sum_{[g]} \big( h g^{-1} \cdot \phi \big)(x)h \, t^{\psi(g)}\\
&=\sum_{[g]} \big( g^{-1} \cdot \phi \big)(x)h \, t^{\psi(g)} t^{\psi(h)}\\
&=\sum_{[g]} \Big( \big( g^{-1} \cdot \phi \big)(x) \, t^{\psi(g)} \Big) \cdot h\\
&=\big( \Psi_\Gamma(\phi)(x) \big) \cdot h.
\end{align*}

We now prove that $\Psi_\Gamma$ induces an isomorphism in cohomology.

\begin{lemma}
	The map~$\psi_\Gamma$ induces a $\Lambda$-isomorphism on the homology level.
\end{lemma}
\begin{proof}
The $\Lambda$-chain map $\psi_\Gamma$ restricts to $\Lambda$-chain maps on $C_+^\bullet(X^\infty; M)$ and $C_-^\bullet(X^\infty;M)$.
Furthermore,~\cite[Example~4.61]{Ran02} shows that $\psi_\Gamma$ restricts to a $\Lambda$-chain isomorphism $ \psi_\cps \colon C^\bullet_\cps(X^\infty;M) \xrightarrow{\sim} C^\bullet(X;M(\Lambda))$ on compactly supported cochains.
All these maps fit into the following commutative diagram of $\Lambda$-chain complexes:
\[
	\begin{tikzcd}
		0\ar[r] & C^\bullet(X; M(\Lambda)) \ar[r] & C^\bullet(X; M(\Gamma^-)) \oplus C^\bullet(X; M(\Gamma^+))  \ar[r] & C^\bullet(X; M(\Gamma))\ar[r]  & 0\\
		0\ar[r] & C^\bullet_\cps(X^\infty; M)\ar[r] \ar[u,"\psi_\cps"] & C^\bullet_-(X^\infty; M) \oplus C^\bullet_+(X^\infty; M)\ar[u,"\psi_\pm"]\ar[r] & C^\bullet(X;M)\ar[u,"\psi_\Gamma"]\ar[r] & 0.
	\end{tikzcd}
\]
Let~$\beta_{\cps}$ and $\beta_\Gamma$ be the connecting homomorphisms in cohomology.
Arguing as in Lemma~\ref{lem:UndlerlyingIsoHomology}, we use Proposition~\ref{prop:PMAcyclicCohomology} and the fact that $\Gamma^-$ and $\Gamma^+$ are flat over $\Gamma$, we obtain the following commutative diagram in homology:
\begin{equation}\label{eq:BetaCPS}
	\begin{tikzcd}
		0\ar[r] & H^{k}(X; M(\Lambda))   \ar[r,"\beta_\Gamma"] & H^{k+1}(X; M(\Gamma)) \ar[r] & 0\\
		0 \ar[r]&  H^{k}(X^\infty; M)  \ar[u,"\psi_\cps"]\ar[r, "\beta_\cps"]&  H^{k+1}(X^\infty;M)\ar[u,"\psi_\Gamma"] \ar[r]&  0.
	\end{tikzcd}
\end{equation}
As the connecting homomorphisms and $\psi_\cps$ are isomorphisms, so is $\psi_\Gamma$, as desired.
\end{proof}

We will now use the $\Lambda$-chain maps $\operatorname{dst}^{\text{lf}}$ and $\Psi_\Gamma$ to show that the geometric and homological twisted Milnor pairing are isometric.
The necessary additional input is that~$\op{dst}^\lf$ and~$\psi_\Gamma$ are compatible with locally finite Poincaré duality.
\begin{lemma}\label{lem:PsiDistPD}
The following diagram commutes:
\begin{equation}
\label{eq:PsiDistPD}
\begin{tikzcd}
	C_2(X; M(\Gamma))\ar[d, "\op{dst}^\lf"]  & \ar[swap,l, "{\cap [X]}"] C^1(X;M( \Gamma)) \\
	C_2^\lf (X^\infty; M) & \ar[swap,l, "{\cap [X^\infty]}"] C^1(X^\infty; M). \ar[u,"\psi_\Gamma"] 
\end{tikzcd}
\end{equation}
\end{lemma}
\begin{proof}
As in the discussion preceding Definition~\ref{def:PoincareLF}, we choose chain representatives~$[X] = \sum_i c_i \in C_3(X; \Z^{\op{triv}})$ and $[X^\infty] = \sum_i \sum_{[g]} g\cdot c_i$ for the fundamental classes of $X$ and $X^\infty$.
For an element $\alpha \in C^1(X^\infty; M)$, the definition of the twisted cap product yields
	\begin{align}
	\label{eq:alphacapXinfty}
		\alpha \cap [X^\infty] &= \sum_{i} \sum_{[g]} \alpha(g\cdot \lrcorner c_i) \otimes g\cdot  \llcorner c_i,
	\end{align}
 	where the sum is over the cosets $[g]$ of $\pi_1(X)/\pi_1(X^\infty)\cong \Z$.
We now compute the composition around the top of the diagram displayed in~\eqref{eq:PsiDistPD}.
In order to compute~$\psi_\Gamma(\alpha) \cap [X]$, we recall the definition of $\psi_\Gamma(\alpha)$ and use the definition of the twisted cap product:
	\begin{align*}
		\psi_\Gamma(\alpha) & = \Big( x \mapsto \sum_{[g]} (g^{-1} \cdot \alpha)(x)t^{\psi(g)} \Big),\\
		\psi_\Gamma(\alpha) \cap [X] & = \sum_{[g]} \sum_i (g^{-1} \cdot \alpha)(\lrcorner c_i )t^{\psi(g)} \otimes \llcorner c_i.
	\end{align*}
Finally, using the definition of $\op{dst}^\lf$ and the $\Lambda$-module structure of $C^\bullet(X^\infty;M)$ (recall Construction~\ref{const:LambdaCohomology}), we can describe the image of $\alpha$ under the top route as
	\begin{align*}
		\op{dst}^\lf \big( \psi_\Gamma(\alpha) \cap [X] \big) 
        &=\sum_{[g]} \sum_i \Big( (g^{-1} \cdot \alpha ) (\lrcorner c_i ) \Big) \cdot g^{-1} \otimes g \cdot \llcorner c_i\\
		&=\sum_{[g]} \sum_i \alpha(g  \cdot \lrcorner c_i) \otimes g \cdot \llcorner c_i,
	\end{align*}
	which agrees with the computation of $\alpha \cap [X^\infty]$ from~\eqref{eq:alphacapXinfty}. 
	It follows that the diagram commutes, and the lemma is proved.
\end{proof}

We can now relate the homological twisted Milnor from Definition~\ref{def:TwistedMilnorPairing} with the geometric Milnor pairing in the cover~$X^\infty$ from Definition~\ref{def:GeometricTwistedMilnor}.

\begin{proposition}\label{prop:Litherland}
	The $\Lambda$-isomorphism~$\op{dst} \colon H_1(X; M(\Lambda)) \to H_1(X^\infty; M)$ defines an isometry between the pairings~$\widetilde{\mu}_{X,M}$ and $\mu_{X,M}^\infty$.
\end{proposition}
\begin{proof} 
Both pairings are compositions of a Bockstein homomorphism, a Poincar\'e duality isomorphism and an evaluation map.
The strategy of the proof is therefore to fit all these maps into a large commutative diagram.
For space reasons, we subdivide this large diagram into smaller diagrams.
As a first step, we consider the portions involving the Bockstein and Poincar\'e duality isomorphisms:
\[
\begin{tikzcd}
	H_1(X; M(\Lambda)) \ar[r,"\big(\beta^\Gamma\big)^{-1}"] \ar{d}{\op{dst}}[swap]{\cong} & H_2(X; M(\Gamma))\ar[r, "\PD_X"] \ar[d, "\op{dst}^\lf"] & H^1(X;M( \Gamma)) \\
	H_1(X^\infty; M) \ar[r, "\big( \beta^\lf \big)^{-1}"] & H_2^\lf (X^\infty; M) \ar[r,"\PD^\lf_{X^\infty}"] & H^1(X^\infty; M). \ar[u,"\psi_\Gamma"] 
\end{tikzcd}\]
The first square commutes by~\eqref{eq:BocksteinCommute} and the second by Lemma~\ref{lem:PsiDistPD}.
Next, we relate the segments of~$\op{Mil}$ and $\op{Mil}^\infty$ that involve evaluation maps:
\begin{equation}\label{eq:HomEv} 
\begin{tikzcd}[column sep=3.5cm]
	H^1(X; M(\Gamma) ) \ar[r, "(\op{const} \circ \wt b)_* \circ \operatorname{ev}_{M(\Gamma)}"] & \text{Hom}_{\text{right-}\F}\big(H_1(X;M( \Lambda))^\tr, \F \big) \\
	H^1(X^\infty; M) \ar[r, "b_* \circ \operatorname{ev}_{M}"] \ar[u,"\psi_\Gamma"] & \text{Hom}_{\text{right-}\F} \big( H_1(X^\infty; M)^\tr, \F\big). \ar[u, "\op{dst}^*"]
\end{tikzcd}
\end{equation}
We claim that this diagram already commutes on the chain level. 
By linearity, it is enough to check this on elementary tensors of the form $mt^k \otimes c$. 
Since $\psi$ is surjective, we may choose~$h \in \pi$ with $\psi(h) = k$. 
In the following calculations, we view the evaluation maps $\operatorname{ev}_{M(\Gamma)}$ and $\operatorname{ev}_{M}$,  (recall Definition~\ref{def:EvaluationMap}) as pairings: for instance, when $\alpha \in H^1(X;M)$, we write $\langle \alpha, c \rangle := \op{ev}_M(\alpha)(c)$.
We compute the top route of~\eqref{eq:HomEv}.
First, apply $\psi_\Gamma$ and then $\operatorname{ev}_{M(\Gamma)}$:
\begin{align*}
	\big \langle \psi_\Gamma(\alpha),m t^k \otimes c \big\rangle 
	&= \Big \langle \sum_{[g]} ( g^{-1} \cdot \alpha) t^{\psi(g)}, m t^k \otimes c \Big\rangle\\
	&= \sum_{[g]}  (g^{-1} \cdot \alpha)(c)t^{\psi(g)} \otimes mt^{k}.
\end{align*}
To obtain an element in $\F$, postcompose with $\op{const} \circ \widetilde{b}$:
\begin{align*}
	\big( \op{const} \circ \wt b\big) \big( \langle \psi_\Gamma(\alpha),mt^k  \otimes c\rangle \big)
	&= \op{const} \Big( \sum_{[g]} b\big( (g^{-1} \cdot \alpha)(c), m \big) t^{k-\psi(g)} \Big) \\
	&= b\big( (h^{-1} \cdot \alpha )(c) , m \big).
\end{align*}

Now we compare this with the result obtained from the bottom route, namely the composition~$\op{dst}^* \circ b_* \circ \op{ev}_M$.
First, for $\alpha \in H^1(X;M)$, apply the definition of $\op{dst}$ (from Lemma~\ref{lem:weakshapiro}) and then the definition of $\op{ev}_M$ to obtain
\begin{align*}
	\langle \alpha, \op{dst}( m t^k \otimes c )\rangle
	&= \langle \alpha,  m\cdot h^{-1} \otimes  h\cdot c )\rangle\\
	&= \alpha(h\cdot c) \otimes_{\Z[\pi^\infty]}  (m\cdot h^{-1}).
\end{align*}
We can now apply the pairing $b$ to this result and obtain
\begin{align*}
	b \big( \langle \alpha, \op{dst}( m t^k \otimes c )\rangle \big) 
        &= b( \alpha(h\cdot c), m h^{-1})\\
	&= b\big( \alpha(h \cdot c) \cdot h , m \big)\\
	&= b\big( (h^{-1} \cdot \alpha)(c) , m \big).
\end{align*}
In the penultimate equation, use that $b(x\cdot h, y) = b(x, h \cdot y)$ for $x,y \in M$ and $h \in \pi$.
The result of the two paths in the diagram agree, 
and so $\op{dst}$ defines an isometry from~$\widetilde{\mu}_{X,M}$ to $\mu_{X,M}^\infty$.
This concludes the proof of the proposition.
\end{proof}

\subsection{The cohomological geometric twisted Milnor pairing}
\label{sub:TwistedMilnorGeometricCohomological}
In this subsection, we define a fourth version of the twisted Milnor pairing.
This is the version of the pairing defined by Kirk and Livingston in~\cite[Section 7]{KirkLivingston}. 
\medbreak

Let $X$ be a closed $3$-manifold, and let $(M,b)$ be 
a unitary representation such that~$H_1(X^\infty; M)$ is a finite dimensional $\F$--vector space. 
Consider the $\F$--linear map obtained as the composition
\begin{align*}
	H^1(X^\infty; M) \times H^1(X^\infty;M)^\tr 
	&\xrightarrow{ \id \times ( \beta_{\cps} )^{-1} } H^1(X^\infty;M) \times H^2_{\op{cs}}(X^\infty;M)^\tr \\ 
	& \xrightarrow{\cup} H^3_{\op{cs}}(X^\infty; M \boxtimes M)\\
	& \xrightarrow{\PD_{X^\infty}^{-1}} H_0(X^\infty;M \boxtimes M)\\
	& \xrightarrow{\cong} M \otimes_{\Z[\pi_1(X^\infty)]}M^\tr\\
	& \xrightarrow{\op{quot}} M \otimes_{\Z[\pi_1(X)]}M^\tr\\
	& \xrightarrow{b} \F  \tag{GeoCoMil}\label{eqn:GeoCoMil}
\end{align*}
of the six maps defined as follows.
Firstly, $\beta^\cps$ is the connecting $\F$-homomorphism that arises from the short exact sequence 
\[ 0 \to C^\bullet(X^\infty,M) \to C_-^\bullet(X^\infty,M) \oplus C^\bullet_+(X^\infty,M) \to C_\cps^\bullet(X^\infty,M) \to 0 \]
described in Lemma~\ref{lem:PMShortExact}; this Bockstein homomorphism is invertible thanks to Proposition~\ref{prop:PMAcyclicCohomology}.
Secondly, $\cup$ is the $\F$-linear twisted cup product described in Subsection~\ref{sub:Cup}.
Thirdly, $\PD_{X^\infty}$ is the twisted Poincar\'e duality $\F$-isomorphism described in Remark~\ref{rem:PDCompactlySupported}.
Fourthly, we use the $\F$-isomorphism of Lemma~\ref{lem:Ccoinvariants}.
Fifthly, $\op{quot}$ denotes the quotient map $M \otimes_{\Z[\pi_1(X^\infty)]}M^\tr \to M \otimes_{\Z[\pi_1(X)]}M^\tr$.
Sixthly and finally,~$b$ is our fixed pairing.

\begin{definition}
\label{def:MilnorPairingCohologyInfiniteCylic}
Let $X$ be a closed $3$-manifold, and let $M$ be 
a unitary representation such that $H_1(X^\infty; M)$ is a finite dimensional $\F$-vector space.
The \emph{geometric cohomological twisted Milnor pairing}
\[ \mu_{X,M} \colon H^1(X^\infty; M) \times H^1(X^\infty; M) \to \F  \]
is the $\F$-sesquilinear pairing defined by the composition in \eqref{eqn:GeoCoMil}. 
\end{definition}

This Milnor pairing is the pairing defined by Kirk-Livingston~\cite[Section 7]{KirkLivingston} and which generalises Milnor's construction in the untwisted case~\cite{MilnorInfiniteCyclic}.

\subsection{Relating the two cohomological pairings}
\label{sub:RelatingCohomologicalPairings}
We relate the cohomological Milnor pairings of Definitions~\ref{def:TwistedMilnorPairingCohomology} and~\ref{def:MilnorPairingCohologyInfiniteCylic}.
\medbreak

We will show that the $\Lambda$-isomorphism~$\psi_\Gamma$ from~\eqref{eq:PsiGamma} is the required isometry.
In order to show that $\Psi_\Gamma$ is compatible with cup products, consider the $\F$--linear map
\begin{align*} 
	\delta_\boxtimes \colon M(\Gamma) \boxtimes M(\Lambda) &\to (M\boxtimes M)(\Lambda)\\
	m t^g \boxtimes n t^h &\mapsto 
	\begin{cases}
		(m\boxtimes n) t^h & g = h\\
		0 & g \neq h.
	\end{cases}
\end{align*}
Using $\delta_\boxtimes$, the relation between the twisted cup products is expressed as follows.

\begin{lemma}\label{lem:PsiCup}
	The following diagram commutes:
	\[ 
	\begin{tikzcd}
		C^k(X; M(\Gamma)) \times C^l(X; M(\Lambda) ) 
		\ar[r, "\cup"] 
		& C^{k+l}(X; M(\Gamma)\boxtimes M(\Lambda))
        \ar[r, "{\delta_\boxtimes}_*"] 
		&C^{k+l}(X; (M\boxtimes M)(\Lambda) ).\\
		C^k(X^\infty; M) \times C^l_{\cps}(X^\infty; M) \ar[u, "\psi_\Gamma \times \psi_{\cps}"] \ar[r, "\cup"] & C^{k+l}_{\cps}(X^\infty; M\boxtimes M) \ar[ru, "\psi_\Gamma"] 
	\end{tikzcd} \]
\end{lemma}
\begin{proof}
	Let $\alpha \in C^k(X^\infty; M),\alpha' \in C^l_{\op{cs}}(X^\infty; M)$ and $c \in C_{k+l}(X^\infty; M \boxtimes M)$. 
    Using consecutively the definition of ${\delta_\boxtimes}_*$, the definition of the twisted cup product, the definition of $\Psi_\Gamma$,
the definition of $\delta_\boxtimes$ (and of the module structures), and the definition of $\Psi_\Gamma$, we obtain
	\begin{align*}
            {\delta_\boxtimes}_*(\psi_\Gamma(\alpha) \cup \psi_\cps(\alpha'))(c) 
            &= {\delta_\boxtimes} \big( \psi_\Gamma(\alpha) \cup \psi_\cps(\alpha')(c) \big)\\
            &= {\delta_\boxtimes} \big( \psi_\Gamma(\alpha)( \lrcorner c) \boxtimes \psi_\cps(\alpha')(   \llcorner c ) \big)\\
            &= {\delta_\boxtimes} \Big( \sum_{[g], [h]} (g^{-1} \cdot \alpha)(\lrcorner c)t^{\psi(g)}
			\boxtimes (h^{-1} \cdot \alpha')(\llcorner c)t^{\psi(h)}\Big)\\
		&=  \sum_{[g]}\Big( (g^{-1} \cdot \alpha)(\lrcorner c) 
			\boxtimes (g^{-1} \cdot \alpha')(\llcorner c) \Big) t^{\psi(g)}\\
		&= \psi_\Gamma (\alpha \cup \alpha') (c).
	\end{align*}
This shows the commutativity of the diagram and establishes the proposition.
\end{proof}

We can now relate the cohomological twisted Milnor pairing from Definition~\ref{def:TwistedMilnorPairingCohomology} with the geometric Milnor pairing in the cover~$X^\infty$ from Definition~\ref{def:MilnorPairingCohologyInfiniteCylic}.

\begin{proposition}
\label{prop:CohomologicalPairingsAgree}
	The $\Lambda$-isomorphism $\psi_\Gamma \colon H^1(X^\infty; M) \to H^1(X; M(\Gamma))$ given by~$\psi_\Gamma$
	induces an isometry between the pairings $\mu_{X,M}$ and $\widetilde{\mu}^\cup_{X,M}$.
\end{proposition}
\begin{proof}
Both pairings are defined by first composing a Bockstein homomorphism and a cup product pairing.
As a consequence, we first consider the following diagram which commutes by the diagram displayed in~\eqref{eq:BetaCPS}:
	\[ \begin{tikzcd}
			H^1(X^\infty; M) \times H^1(X^\infty; M) \ar[r, "\id \times (\beta_\cps)^{-1}"] \ar[d, "\psi_\Gamma \times \psi_\Gamma"]
			& H^1(X^\infty; M) \times H^2_\cps(X^\infty; M)
			\ar[d, "\psi_\Gamma \times \psi_\cps"]\\
			H^1(X; M(\Gamma)) \times H^1(X; M(\Gamma)) \ar[r, "\id \times (\beta_\Gamma)^{-1}"] 
			& H^1(X; M(\Gamma)) \times H^2(X; M(\Lambda)).
	\end{tikzcd}\]
Next, Lemma~\ref{lem:PsiCup} shows that the following diagram commutes:
	\[ \begin{tikzcd}
			H^1(X^\infty; M) \times H^2_\cps(X^\infty; M) \ar[r, "\cup"] \ar[d, "\psi_\Gamma \times \psi_\cps"]
			& H^3_\cps(X^\infty; M\boxtimes M)
			\ar[d, "\psi_\Lambda"]\\
            H^1(X; M(\Gamma)) \times H^2(X; M(\Lambda)) \ar[r, "{\delta_\boxtimes} \,\circ\, \cup"] 
			& H^3(X; (M\boxtimes M)(\Lambda)).
	\end{tikzcd}\] 
The two pairings are now obtained by taking Poincar\'e duality and identifying the $0$-th homology terms appropriately.
To make this precise, we note that the augmentation map $\Lambda \to \F$ induces an $\F$-linear map $\op{aug} \colon (M\boxtimes M)(\Lambda)  \to M\boxtimes M$  which fits into the following commutative diagram:
\[ \begin{tikzcd}
		H^3_\cps(X^\infty; M\boxtimes M) \ar[r, "\PD_{X^\infty}"] \ar[d, "\psi_\Lambda"]
			& H_0(X^\infty; M\boxtimes M) \ar[r]& M\otimes_{\Z[\pi_1(X)]} M^\tr\\
		H^3(X; (M\boxtimes M)(\Lambda)) \ar[r, "\PD_X"] & H_0( X; (M\boxtimes M)(\Lambda) ) \ar[u, "\op{dst}"] \ar[r,"\op{aug}_*"] & H_0(X; (M\boxtimes M)). \ar[u, "\cong"]
			&  
	\end{tikzcd}\]
If we combine each of the three previous diagrams, then we obtain a large commutative whose top row is (the adjoint of) $\mu_{X,M}$ and  whose bottom row is the adjoint of the following pairing:
	\begin{align*}
		H^1(X; M(\Gamma)) \times H^1(X; M(\Gamma)) 
		&\xrightarrow{\id \times (\beta_\Gamma)^{-1} }
			H^1(X; M(\Gamma)) \times H^2(X; M(\Lambda))\\
		&\xrightarrow{\cup} H^3(X; M(\Gamma)\boxtimes M(\Lambda))\\
        &\xrightarrow{ {\delta_\boxtimes} } H^3(X; (M\boxtimes M)(\Lambda))\\
		&\xrightarrow{ \op{PD}_X^{-1} } H_0(X; (M\boxtimes M)(\Lambda))\\
		&\xrightarrow{ \op{aug}_* } H_0(X; M\boxtimes M)\\
		&\xrightarrow{ \cong } M\otimes_{\Z[\pi_1(X)]} M^\tr.
	\end{align*}
In order to conclude the proof of the proposition, it only remains to show that this pairing agrees with $\widetilde{\mu}_{X,M}^\cup$.
By Lemma~\ref{lem:PDCoeff}, Poincaré duality commutes with change of coefficients, it only remains to notice that	$\op{const}\circ \wt b = b \circ \op{aug} \circ {\delta_\boxtimes}$.
This concludes the proof of the proposition.
\end{proof}

Theorem~\ref{thm:PairingsAgree} from the introduction now follows by combining Proposition~\ref{prop:CohomologicalPairingsAgree}  with Proposition~\ref{prop:IsometryTwistedPairings}.
Indeed these propositions show that the pairing $\mu_{X,M}$ defined by Kirk and Livingston and the pairing $\widetilde{\mu}_{X,M}$ from Definition~\ref{def:TwistedMilnorPairing} agree.
In particular, from now on, we use $\mu_{X,M}$ to denote any version of the twisted Milnor pairing.

\section{The Milnor pairing for fibered manifolds}
\label{sec:Fibered}

In this section, we prove Theorem~\ref{thm:MilnorIntersectionMonodromy}.
This result describes the twisted Milnor pairing for fibered manifolds, and its proof makes use of the homological definition of the twisted Milnor pairing from Definition~\ref{def:GeometricTwistedMilnor}.
In Subsection~\ref{sub:Cross}, we discuss the twisted cross product.
In Subsection~\ref{sub:Fibered}, we describe the twisted Milnor pairing for fibered $3$-manifolds.

\subsection{The twisted cross product}
\label{sub:Cross}
In this subsection, we describe twisted cross products on homology and cohomology.
We fix a commutative ring $R$, topological spaces $X$ and $Y$, and universal covers~$\wt X$ of $X$, $\wt Y$ of $Y$, and $\wt X \times \wt Y$ of $X \times Y$.
\medbreak
Given two singular simplices~$\sigma_X \colon \Delta^k \to \wt X$ and $\sigma_Y \colon \Delta^l \to \wt Y$, one can form a new chain 
$\sigma_X \times \sigma_Y \colon \Delta^k \times \Delta^l \to \wt X \times \wt Y$ as in~\cite[Theorem 16.1]{Bredon93}.
While the proof of this result involves acyclic models, the intuition behind the cross product is as follows: while the product~$\Delta^k \times \Delta^l$ is not itself a simplex, it can be decomposed into $(k+l)$--dimensional simplices, whose sum gives the chain~$\sigma_X \times \sigma_Y$ of the cross product. 
Since the cross product is natural~\cite[Theorem~16.1~(2)]{Bredon93}, 
we obtain a~$(\Z[\pi_1(X)], \Z[\pi_1(Y)])$-linear map 
\[ \times \colon C(\widetilde{X},\Z[\pi_1(X)]) \times C(\widetilde{Y},\Z[\pi_1(Y)]) \to C(\widetilde{X} \times \widetilde{Y} \ ,\Z[\pi_1(X \times Y)]). \]
Fix an $(R,\Z[\pi_1(X)])$--bimodule $A$ and an $(R,\Z[\pi_1(Y)])$--bimodule $B$.
After tensoring with the modules~$A$ and $B$, one obtains well defined $(R,R)$-linear map
\begin{align*}
	\times \colon C_k(X; A) \times C_l(Y; B)^\tr &\to C_{k+l}(X\times Y; A \boxtimes B)\\
	\Big( a \otimes_{\Z[\pi_1(X)]} \wt \sigma_X, b \otimes_{\Z[\pi_1(Y)]} \wt \sigma_Y \Big) &\mapsto \big(a \boxtimes b \big) \otimes_{\Z[\pi_1(X\times Y)]} \widetilde{\sigma}_X \times \widetilde{\sigma}_Y.
\end{align*}
A similar procedure can be carried out in cohomology.
For any spaces $X,Y$ and any cochains $\alpha,\beta \in C^k(X), C^\ell(Y)$, a cross product $\alpha \times \beta \in C^{k+l}(X \times Y)$ is defined in~\cite[Chapter VI, Section 3]{Bredon93}.
This cross product construction is natural, which can be deduced from the relation between the cup and cross products~\cite[p.~327]{Bredon93},
and therefore induces a $(\Z[\pi_1(X)],\Z[\pi_1(Y)])$-linear map on the chain complexes of the universal covers.
As in homology, one then obtains maps on the twisted cochain complexes.
These pairings  descend to (co)homology.

\begin{definition} \label{def:TwistedCross}
Let $X,Y$ be spaces, let $A$ be a $(R,\Z[\pi(X)])$-bimodule, and let let~$B$ be a $(R,\Z[\pi(Y)])$-bimodule. 
The cross product construction induces the~$(R,R)$-linear \emph{twisted cross products}
\begin{align*}
&\times \colon H^i(X;A) \otimes_\Z H^j(Y;B)^\tr \to H^{i+j}(X \times Y;A \boxtimes B), \\
&\times \colon H^i(X;A) \otimes_\Z H_k(Y;B)^\tr \to H_{k-i}(X \times Y ;A \boxtimes B).
\end{align*}
\end{definition}

We record some properties of the cross product in Lemma~\ref{lem:CrossProduct} below. 
The element~$1_Y \in C^0(Y;\Z)$ is the augmentation cocycle, defined by the property~$1_Y\big(\wt \sigma\big) =~1$ for every singular simplex~$\wt \sigma$.
Also, recall that we can identify~$A \boxtimes \Z = A$.

\begin{lemma} \label{lem:CrossProduct}
	Let $X,Y$ be two topological spaces. Let $A$ be a $\Z[\pi_1(X)]$--module, and $B$ a $\Z[\pi_1(Y)]$--module. Let $p_X \colon \wt X \times \wt Y \to \wt X$ denote the projection, which is~$\pi_1(X) \times \pi_1(Y)$--equivariant. For $\alpha \in H^*(X;A)$ and $\beta \in H^*(Y; B)$ and $a,a' \in H_*(X; A)$ and $b \in H_*(Y;B)$, we have 
\begin{enumerate}
	\item $p_X^*(\alpha)=\alpha \times 1_Y
		\in H^*(X\times Y; A\boxtimes \Z)\cong H^*(X\times Y; A)$,
\item $(\alpha \times \beta) \cap (a \times b)=(-1)^{\operatorname{deg}(\beta)\operatorname{deg}(a)} (\alpha \cap a) \times (\beta \cap b)\in H^*(X\times Y; A\boxtimes B)$,
\item $p_X^*(a') \cap (a \times b)=(a' \cap a) \times b \in H^*(X\times Y; A\boxtimes B)$.
\end{enumerate}
\end{lemma}
\begin{proof}
Apply the proof of the untwisted case to the singular simplices in the universal covers. 
For a proof of the first assertion; see~\cite[p.~324]{Bredon93}. The proof of the second assertion 
can also be found in Bredon~\cite[Theorem~5.4]{Bredon93}. The third assertion is a consequence of the first two.
\end{proof}

\subsection{Fibered $3$--manifolds}\label{sub:Fibered}
In this subsection, we describe the twisted Milnor pairing for fibered $3$-manifolds. 
The result is expressed in terms of the twisted intersection form of the fiber and the effect of the monodromy on twisted homology.
\medbreak

First, we recall some facts about fibered $3$-manifolds.
For a surface~$\Sigma$ and a diffeomorphism~$\varphi \colon \Sigma \to \Sigma$ that fixes a basepoint~$x_0 \in \Sigma$, construct the mapping~cylinder
\[ X = \Sigma \times \R / ( \varphi^{k} (z), t) \sim (z, t+k). \] 
The $3$-manifold~$X$ is called \emph{fibered} with \emph{fiber}~$\Sigma$ and \emph{monodromy}~$\varphi$.
The infinite cyclic cover of $X$ with respect to the projection induced map $\pi_1(X) \to \pi_1(S^1) = \Z$ is diffeomorphic to $\Sigma \times \R$.
The deck transformation group of this cover is~$\Z$ and given $(z, t) \in \Sigma \times \R$, an integer $k \in \Z$ acts by the formula
$k \cdot (z, t) = (\varphi^{-k}(z), t+k)$.
Recall that the fundamental group of $X$ has the presentation
\[ \pi_1(X) = \big \langle \pi_1(\Sigma), t \ \vert \ tgt^{-1} = (\varphi^{-1})_*(g), \quad  g \in \pi_1(\Sigma) \big \rangle, \]
where $t$ denotes the path $\gamma(s) = (x_0, s)$.
We now fix some notation.

\begin{notation}
\label{not:SetUpFibered}
Fix a fibered $3$-manifold $X$ with fiber surface~$\Sigma$, monodromy~$\varphi$ and projection map $f \colon X \to S^1$.
Let $M$ be 
a unitary representation such that~$H_1(X^\infty;M)$ is~$\Lambda$-torsion.
\end{notation}

The goal of the next construction is to endow $H_1(\Sigma;M)$ with the structure of a $\Lambda$-module and to describe the automorphism induced by $\varphi$ on $H_1(\Sigma;M)$.

\begin{construction}\label{const:FibredHomology}
Since $\pi_1(X^\infty) = \pi_1(\Sigma)$, we can also consider the restriction of $M$ to $\pi_1(X^\infty)$ as a $\pi_1(\Sigma)$--module. 
Fix a universal cover~$\wt \Sigma$ for $\Sigma$, and observe that the manifold~$\wt \Sigma \times \R$ is a universal cover for~$X$.
The  inclusion~$\iota \colon \wt \Sigma \to \wt X$ by $x \mapsto x \times \{0\}$ is $\pi_1(X^\infty)$--equivariant, and consequently this map induces $\F$-isomorphisms
\begin{equation} \label{eq:SigmaIso} 
	H_k(\Sigma; M) \xrightarrow{\sim} H_k(\Sigma \subset X^\infty; M) \xrightarrow{\sim}
	H_k(X^\infty; M).
\end{equation}
For the last isomorphism, note that $\wt \Sigma \times \R$ deformation retracts to $\wt \Sigma$ in a $\pi_1(X^\infty)$--equivariant way.
Note that the projection~$\pi_\Sigma \colon \wt \Sigma \times \R \to \wt \Sigma$ is $\pi_1(\Sigma)$--equivariant, and a retraction for the inclusion $\iota \colon \wt \Sigma \to \wt X$.

The $\F$-vector space~$H_k(X^\infty;M)$ is a $\Lambda$--module.
We now equip $H_k(\Sigma;M)$ with the structure of a $\Lambda$--module via an automorphism $\varphi_M^{-1}$ such that the composition in \eqref{eq:SigmaIso} is an isomorphism of $\Lambda$--modules. 
The automorphism~$\varphi_M$ will play a key role in Theorem~\ref{thm:MilnorIntersectionMonodromy}, which expresses the Milnor pairing on $X^\infty$ in terms of $\Sigma$.

Let $\wt \varphi \colon \wt \Sigma \to \wt \Sigma$ be the unique lift that preserves the basepoint $\wt x_0 \in \wt \Sigma$. 
The aforementioned map $\varphi_M$ is defined by
\begin{align*}
 \varphi_M^{-1} \colon &H_k(\Sigma; M) \to H_k(\Sigma; M) \\
        &m \otimes x \mapsto m \cdot t^{-1} \otimes \big (\wt \varphi \big)^{-1} (x). 
\end{align*}
\end{construction}

The next remark uses the cross product to relate the twisted $1$-chains of $\Sigma$ to the locally finite twisted $2$-chains of $X^\infty$.

\begin{remark}\label{rem:CrossProduct}
A locally finite fundamental class $[\R]\in C_1^{\operatorname{lf}}(\R)$ for~$\R$ is obtained by setting~$[\R]:= \sum_{k \in \Z} [k, k+1]$.
Taking the cross product with $[\R]$ gives rise to an~$\F$-linear map
	\[ \times [\R] \colon C_1(\Sigma;M) \to C_2^{\operatorname{lf}}(X^\infty;M).\] 
	Indeed, given $a < b$ and a simplex~$\wt \sigma \in C_1(\wt \Sigma)$, only finitely summands of the infinite sum~$\wt \sigma \times [\R]$ have their image intersect~$\Sigma \times [a,b]$; recall Definition~\ref{def:EndHomology}.
	Similarly, the locally finite classes $[(-\infty,0)],[(0,\infty)]\in C_1^{\operatorname{lf}}(\R)$ induces $\F$-linear maps
	\begin{align*}
& \times [(-\infty,0)] \colon C_1(\Sigma;M) \to C_2^{-}(X^\infty;M), \\
& \times [(0,\infty)] \colon C_1(\Sigma;M) \to C_2^{+}(X^\infty;M).
	\end{align*}
\end{remark}

The adjoint of the Milnor pairing of Definition~\ref{def:GeometricTwistedMilnor} is obtained as the composition of the inverse of a Bockstein homomorphism with Poincar\'e duality and an evaluation map.
Here recall that the aforementioned inverse Bockstein
\[ \beta^{-1} \colon H_1(X^\infty;M) \to H_2^{\operatorname{lf}}(X^\infty;M)\]
arises from the short exact sequence in~\eqref{eq:ShortHomology}.
The next lemma relates the inclusion induced isomorphism $\iota \colon H_1(\Sigma;M) \to H_1(X^\infty;M)$ to 
$\beta$ and to the cross product~$\times [\R].$ 

\begin{lemma}
\label{lem:MilnorMapCross}
The composition $ \beta^{-1} \circ \iota$ agrees with the map $\times [\R]$.
\end{lemma}
\begin{proof}
Let $[x] \in H_1(\Sigma;M)$. 
The chain $\iota(x)$ bounds the chains $\iota(x) \times [0,\infty) \in C_2^+(X;M)$ and $\iota(x) \times (-\infty,0] \in C_2^-(X;M)$.
 The definition of $\beta^{-1}$ (as a connecting homomorphism) and the fact that $\iota(x) \times [\R] = \iota(x) \times [0,\infty) + \iota(x) \times (-\infty,0]$ now imply that~$\beta^{-1}(\iota(x))=x \times [\R]$.
\end{proof}

Next, we relate locally finite Poincar\'e duality on $X^\infty$ with Poincar\'e duality on~$\Sigma$. 
\begin{lemma}
\label{lem:CrossProductMagic}
If $\pi_\Sigma \colon \wt \Sigma \times \R \to \wt \Sigma$ denotes the canonical projection, then one has
\[ \PD^{\text{lf}}_{X^\infty;M}([a] \times [\R])= (\pi_\Sigma)^*\operatorname{PD}_{\Sigma}([a]). \]
for all $[a] \in H_1(\Sigma;M)$.
\end{lemma}
\begin{proof}
Recall that the inverse of (locally finite) Poincar\'e duality is obtained by capping with the (locally finite) fundamental class.
The statement of the lemma therefore reduces to showing that 
\[a \times [\R] = (\pi_\Sigma)^*\operatorname{PD}_{\Sigma}(a)  \cap [X^\infty].\]
\begin{claim}
\label{claim:FundamentalClassXInfinity}
The class~$[\Sigma] \times [\R]$ is a locally finite fundamental class for~$X^\infty$.
\end{claim}
The fundamental class~$[X^\infty] \in C^{\text{lf}}(X^\infty; \Z^{\text{triv}})$ is characterised by the property that for every $x \in X^\infty$, the homology of class $[X^\infty]$ is sent to a generator of~$H_3(X^\infty,X^\infty \setminus \lbrace x \rbrace; \Z^{\text{triv}})$.
Let $F \colon X^\infty \to \R$ denote the map lifting our fixed map~$X \to S^1$, and suppose that $x \in F^{-1}(]-\infty,a] \cup[b,\infty[)=:L_a \cup R_b$ for some~$a,b \in \R$.
Observe that we have an isomorphism
\begin{align*} 
H_3^{\text{lf}}(X^\infty, X^\infty \setminus \lbrace x \rbrace  ; \Z)
& \cong H_3(\Sigma \times [a,b], \Sigma \times [a,b] \setminus \lbrace x \rbrace  ; \Z) \\
& \cong H_3(\Sigma \times [a,b], \Sigma \times \{a,b\}; \Z).
  \end{align*}
 Since $[\Sigma]\times [\R]$ is sent to~$[\Sigma] \times [a,b]$ under this composition, and since $[\Sigma] \times [a,b]$ is the generator of $H_3(\Sigma \times [a,b], \Sigma \times \{a,b\}; \Z) $, we have shown that the chain~$[\Sigma] \times [\R]$ is a fundamental chain of~$X^\infty$.
This concludes the proof of the claim. \claimbox

Using successively Claim~\ref{claim:FundamentalClassXInfinity}, the third item of Lemma~\ref{lem:CrossProduct}, and the fact that $\PD_\Sigma$ and $-\cap [\Sigma]$ are mutual inverses, we obtain
\begin{align*}
	(\pi_\Sigma)^*\operatorname{PD}_{\Sigma}(a)  \cap [X^\infty]
&=(\pi_\Sigma)^*\operatorname{PD}_{\Sigma}(a)  \cap ([\Sigma] \times [\R]) \\
&=  (\operatorname{PD}_{\Sigma}(a) \cap [\Sigma])  \times [\R] \\
&=  a \times [\R].
\end{align*}
This concludes the proof of the lemma.
\end{proof}

We can now conclude our proof of the computation of the Milnor pairing.
As in Lemma~\ref{lem:nonSingularEval}, for $Y =\Sigma, X^\infty$, we obtain an evaluation map
\[ \langle -,-\rangle_{Y,M} \colon H^1(Y;M) \times H_1(Y;M) \to \F. \]
Recall furthermore that there is an inclusion induced map
 $\iota \colon H_1(\Sigma;M) \to H_1(X^\infty;M)$ and a projection induced map 
$(\pi_\Sigma)_* \colon H_1(X^\infty ;M) \to H_1(\Sigma; M)$.
The next proposition describes the twisted Milnor pairing on $X$ in terms of the twisted intersection form $\langle -,-\rangle_{\Sigma,M}$ of the fiber surface $\Sigma$.
\begin{proposition}\label{prop:MilnorIntersection}
	For $a \in H_1(\Sigma;M)$ and $c \in H_1(X^\infty; M)$, one has
	\[ \mu_{X,M}(\iota(a), c)=\langle a,  (\pi_\Sigma)_* c\rangle_{\Sigma,M}. \]
\end{proposition}
\begin{proof}
The definition of the Milnor pairing from Definition~\ref{def:GeometricTwistedMilnor}, Lemma~\ref{lem:MilnorMapCross}, Lemma~\ref{lem:CrossProductMagic}, properties of the evaluation map and the fact that $\pi_\Sigma \circ \iota=\id_\Sigma$ show that
\begin{align*}
\Big \langle \PD^{\text{lf}}_{X^\infty,M}\big( \beta^{-1} (\iota(a))\big),c \Big\rangle_{X^\infty}
 &=   \langle \PD^{\text{lf}}_{X^\infty,M}(a \times [\R]),c \rangle_{X^\infty} \\
 &= \langle (\pi_\Sigma)^*\operatorname{PD}_{\Sigma,M}(a) ,c \rangle_{X^\infty} \\
 &= \langle \operatorname{PD}_{\Sigma,M}(a), (\pi_\Sigma)_*(c) \rangle_{\Sigma}.
\end{align*}
This concludes the proof of the proposition.
\end{proof}

We can now prove our main result.
To state this result, we recall from Construction~\ref{const:FibredHomology} that the monodromy $\varphi \colon \Sigma \to \Sigma$ of a fibered $3$-manifold $f \colon M \to S^1$ induces an automorphism~$\varphi_M \colon H_1(\Sigma; M) \to H_1(\Sigma; M)$.
Also, the inclusion induces a map~$\iota \colon H_1(\Sigma;M) \to H_1(X^\infty;M)$.

\begin{theorem}\label{thm:MilnorIntersectionMonodromy}
Let $f \colon X \to S^1$ be a closed fibered $3$-manifold with fiber surface~$\Sigma$ and monodromy~$\varphi$. Let $M$ be a unitary representation such that $H_1(X^\infty;M)$ is~$\F[t^{\pm 1}]$-torsion. 
For $x,y \in H_1(\Sigma; M)$, the twisted Milnor pairing can be expressed~as
\[ \mu_{X,M}(\iota(x), t^k \cdot \iota(y))= \big\langle x,  \varphi_M^{-k} ( y )\big \rangle_{\Sigma, M}. \]
\end{theorem}
\begin{proof}
For an elementary tensor $m \otimes c \in C_1(X^\infty; M)$, the $\Lambda$-action described in Construction~\ref{const:LambdaModule} gives 
$t \cdot (m\otimes c) = m\cdot t^{-1} \otimes_{\Z[\pi_1(X^\infty)]} t\cdot c$.
As the deck transformation associated to $t \in \pi_1(X)$ on $\widetilde{X}=\widetilde{\Sigma} \times \R$ is given by 
$(x,s) \mapsto \big( (\wt \varphi)^{-1}(x), s+1)$, we then obtain
\begin{align*}
	(\pi_\Sigma)_* \big( t \cdot (m\otimes c) \big) &= m\cdot t^{-1} \otimes_{\Z[\pi_1(\Sigma)]} (\pi_\Sigma)_* \big( t \cdot c \big) \\
	&= m \cdot t^{-1} \otimes_{\Z[\pi_1(\Sigma)]} (\wt \varphi)^{-1}\big( (\pi_\Sigma)_*(c) \big). 
\end{align*}
Now write the chain $\iota(y) \in C_1(X^\infty;M)$, as $\iota(y)=\sum_i m_i \otimes \iota(y_i)$ and use this equality as well as the description of $\varphi_M$ from Construction~\ref{const:FibredHomology} to get 
\[ (\pi_\Sigma)_* \big( t \cdot \iota(y) \big) = \sum_i m_i \cdot t^{-1} \otimes_{\Z[\pi_1(\Sigma)]}(\wt \varphi)^{-1} (y_i) = \varphi_M^{-1}(y).\]
Use this equality and apply Proposition~\ref{prop:MilnorIntersection} to obtain:
\[ \mu_{X,M}(\iota(x), t \cdot \iota(y)) = \big\langle x,  (\pi_\Sigma)_*(t \cdot \iota(y)) \big \rangle_{\Sigma, M} = \big\langle x, \varphi_M^{-1}(y) \big \rangle_{\Sigma, M}. \]
The proof for $k>1$ is identical and this concludes the proof of the theorem.
\end{proof}

\begin{example}[Levine-Tristram signature]
\label{ex:ErleLT}
Let $K$ be a fibered knot and $\omega \in S^1$. 
Denote the exterior of $K$ by~$X_K$, its fundamental group by~$\pi$, and the abelianization map by $f \colon \pi \to \Z$.
Define the $(\C, \Z[\pi])$--bimodule~$\C^\omega$, which has $\C$ as the underlying vector space, and action $\gamma \cdot z = \omega^{f(\gamma)}z$ for $\gamma \in \pi$ and $z \in \C$.

Pick a fiber surface~$\Sigma$, which has genus~$g = g_3(K)$ of $K$.
Note that if we restrict $\C^\omega$ to $\pi^\infty = [\pi, \pi]$, we simply obtain $\C$ with the trivial action. Thus, $H = H_1(\Sigma; \C^\omega) = \C^{2g}$. 
	Let $A$ be a Seifert matrix of $\Sigma$, which can be assumed to be invertible (recall that any Seifert matrix associated to a fiber surface of a fibered knot is invertible over $\Z$).
The intersection paring $\langle -, - \rangle_\Sigma$ is represented by the matrix~$A - A^\tr$, i.e. it is isometric to $(x,y) \mapsto x^\tr (A-A^\tr)\overline{y}$. 
As the action of the monodromy on~$H_1(\Sigma;\Z)$ is represented by $A^{-1}A^\tr$~\cite[Lemma~8.1]{SavelievLectures}, Construction~\ref{const:FibredHomology}, implies that the action of~$t \in \Lambda$ on $H_1(\Sigma;\C^\omega)$ is given by $\big( \omega A^{-1}A^\tr \big)^{-1} \in \op{Mat}(\C,2 g)$.
We therefore obtain a skew-isometric structure $(\C^{2g}, A-A^\tr, \omega^{-1} A^{-\tr }A)$; see Definition~\ref{def:IsometricStructure} below.

We compute the signature of the form~$b_{K,\C^\omega}(x,y)=\langle t\cdot x, y \rangle_\Sigma - \langle x, t\cdot y\rangle_\Sigma$. 
Using the matrices described above, this pairing is represented by
    \begin{align*} 
    b_{K,\C^\omega} &\sim \omega^{-1} \big( A^{-\tr}A \big)^\tr (A - A^\tr) - \omega(A - A^\tr) A^{-\tr}A\\
    &= \omega^{-1} A^\tr A^{-1} (A - A^\tr) - \omega(A - A^\tr) A^{-\tr}A\\
    &\sim \omega^{-1} (A - A^\tr)^{-\tr} A^{\tr}A^{-1} - \omega (A - A^\tr)^{-\tr}(A - A^\tr) A^{-\tr}A (A - A^\tr)^{-1}\\
    &\sim \omega (A - A^\tr)^{-1}(A - A^\tr) A^{-\tr}A (A - A^\tr)^{-1} - \omega^{-1} (A - A^\tr)^{-1} A^{\tr}A^{-1}\\
    &\sim \omega A^{-\tr}A (A - A^\tr)^{-1} - \omega^{-1} (A - A^\tr)^{-1} A^{\tr}A^{-1}\\
    &\sim \omega A (A - A^\tr)^{-1} A - \omega^{-1} A^\tr (A - A^\tr)^{-1} A^{\tr}.
    \end{align*}
We can assume that $A - A^\tr = J$ with $J$ the standard symplectic form, which fulfills the relations~$J^{-1} = J^\tr = -J$. Compute
	\begin{align*} 
    b_{K,\C^\omega} & \sim \omega A(A - A^\tr)^{-1} A - \omega^{-1} A^\tr (A - A^\tr)^{-1} A^\tr \\
                    & = -\omega A J A + \omega^{-1} A^\tr J A^\tr\\
                    & = -\omega (J+A^\tr) J A + \omega^{-1} A^\tr J (A-J)\\
                    & = \big( \omega A + \omega^{-1} A^\tr \big) - \omega A^\tr J A  +\omega^{-1} A^\tr J A\\
                    & = \big( \omega A + \omega^{-1} A^\tr \big) + \big(  A^\tr (\omega^{-1}J - \omega J) A \big).
	\end{align*}
For $\omega=1$, we deduce that $b_{K,\C^\omega}$ is represented by the matrix $A+A^\tr$, recovering a result of Erle in the fibered case~\cite{Erle}.
In particular, $\sign(b_{K,\C^1})=\sigma_K(-1)$.
For general $\omega \in S^1$ (regardless of whether the knot~$K$ is fibered), we can show that~$\sign(b_{K,\C^\omega})=\sigma_K(-\omega)-\sigma_K(\omega)$ by using the signature jumps from~\cite{BorodzikConwayPolitarczyk} and Theorem~\ref{thm:MainMain} below, but how to derive this from the matrix calculations above remains unclear.
\end{example}

\section{Relation to the twisted Blanchfield pairing}
\label{sec:BlanchfieldMilnor}

In this section, we relate the twisted Milnor pairing to the twisted Blanchfield pairing; this is done by using a trace map first introduced by Litherland~\cite{LitherlandCobordism}.
We then deduce that the twisted Milnor pairing is skew-Hermitian and non-singular.

\subsection{The trace map}
\label{sub:Trace}
We recall the definition of the trace map due to Litherland~\cite[Appendix A]{LitherlandCobordism}; see also~\cite[Section 28.E]{RanickiHighDimensionalKnotTheory}.
\medbreak

For a field $\F$ of characteristic zero, we set $\Lambda:=\LF,\Gamma:=\F[[t]]$ and consider the rings $\Lambda \subset \Gamma^-,\Gamma^+ \subset~\Gamma$ just as in~\eqref{eq:GammaPlusMinus}.
Each of these rings is also a $\Lambda$-module. 
Recall from Remark~\ref{rem:GammaFields} that $\Gamma^\pm$ is a field.
        The universal property of the field of fractions~$Q$, and the canonical inclusions~$\Lambda \to~\Gamma^+$ and $\Lambda \to \Gamma^-$ give rise to~$\Lambda$-linear morphisms
\begin{align*}
&i_{+}\colon Q \rightarrow \Gamma^+, \\
&i_{-}\colon Q \rightarrow \Gamma^-,
\end{align*}
which send~$p(t)/q(t) \mapsto p(t)\cdot q(t)^{-1}$ where the inverse is taken in the field~$\Gamma^\pm$.
Observe that the difference~$i_+-i_-$ of these morphisms vanishes on~$\Lambda$.
Since $\Gamma^+$ and~$\Gamma^-$ include into $\Gamma$, this difference induces an~$\Lambda$--linear map~$Q/\Lambda\to \Gamma$. 
 We use~$\text{const} \colon \Gamma \rightarrow~\F$ to denote the~$\F$--linear map that takes a power series to its constant term. 
The main definition of this subsection, which is due to Litherland, is the following.
\begin{definition}
\label{def:Trace}
The \emph{trace map}~$\chi \colon Q/\Lambda \rightarrow \F$ is the $\F$-linear map~
\[ \chi := \text{const}  \circ (i_+-i_-) \colon Q/\Lambda \to \F.\]
\end{definition}

This definition of the trace map first appears in~\cite[Appendix A]{LitherlandCobordism}, but we refer to~\cite[Section 28.E]{RanickiHighDimensionalKnotTheory} 
for related discussions.
Computations involving $\chi$ appear in Subsection~\ref{sub:RelationToJump}, but here, we only list some of its properties, referring to~\cite[Proposition~A.3]{LitherlandCobordism} for a proof.

\begin{proposition}\label{prop:trace}
	The trace map~$\chi\colon Q/\Lambda \rightarrow \F$ has the following properties:
\begin{enumerate}
	\item\label{item:ChiOverline} $\chi( \overline x)=-\overline{\chi(x)}$ for all $x \in Q/\Lambda$;
	\item\label{item:TorsionHom} the induced map $\chi \colon \Hom_{\Lambda}(H,Q/\Lambda) \xrightarrow{\sim} \Hom_\F(H,\F)$ is an $\F$-isomorphism for every $\Lambda$--torsion module~$H$.
\end{enumerate}
\end{proposition}

If $\lambda \colon H \times H \to \F(t)/\LF$ is a linking form, then $(x,y) \mapsto \chi\big( \lambda(x,y) \big)$ is a skew-Hermitian form; see Proposition~\ref{prop:trace} \eqref{item:ChiOverline}. 
Moreover, since~$\lambda$ is sesquilinear, we also have $\chi\big( \lambda(tx,ty)\big)=\chi \big( \lambda(x,y) \big)$.
In the language of Subsection~\ref{sub:SkewIsometric} below, the triple~$(H,\chi \circ \lambda,t)$ is an example of a skew-isometric structure over $\F$.

\subsection{Milnor and Blanchfield pairings}

In this subsection, we briefly recall the definition of the twisted Blanchfield pairing and relate it to the twisted Milnor pairing.
As a consequence, we deduce that the twisted Milnor pairing is skew-Hermitian and non-singular.
\medbreak
Let $X$ be a closed $3$-manifold, and let $(M,b)$ be
a unitary representation such that~$H_1(X; M(\Lambda))$ is $\Lambda$-torsion. 
The short exact sequence $0 \to \Lambda \to Q \to Q/\Lambda \to 0 $ of coefficients gives rise to a long exact sequence in homology with connecting homomorphism
$\beta^{Q/\Lambda} \colon H_2(X;M(Q/\Lambda))  \to H_1(X;M(\Lambda)). $
Since $H_1(X; M(\Lambda))$ is~$\Lambda$-torsion,~$\beta^{Q/\Lambda}$ is an isomorphism.
Consider the composition
\begin{align*}
	H_1(X;M(\Lambda)) 
&\xrightarrow{ \big( \beta^{Q/\Lambda} \big)^{-1} } H_2(X; M(Q/\Lambda)) \\
&\xrightarrow{ \PD_X } H^1(X; M( Q/\Lambda)) \\
&\xrightarrow{ \op{ev}_{M(\Lambda)} } 
\Hom_{\text{right-}\Lambda}\left( H_1(X; M(\Lambda))^\tr , M(Q/\Lambda) \otimes_{\Z[\pi_1(X)]} M(\Lambda)^\tr \right)\\
&\xrightarrow{\wt b_*} \Hom_{\text{right-}\Lambda}\left( H_1(X; M(\Lambda)) , Q/\Lambda \right)  \tag{Blanchfield} \label{eqn:BlanchfieldPairing}
\end{align*}
of the following four $\Lambda$-linear maps: the Bockstein homomorphism, Poincar\'e duality, the evaluation map,
and the $(\Lambda, \Lambda)$--morphism~$\wt b \colon M(Q/\Lambda) \otimes_{\Z[\pi]} M(\Lambda)^\tr \to Q/\Lambda$ obtained by extending $b$.
The \emph{twisted Blanchfield pairing} 
\[ \op{Bl}_{X,M} \colon H_1(X;M(\Lambda)) \times H_1(X;M(\Lambda)) \to Q/\Lambda \]
is the $\Lambda$--sesquilinear pairing defined by the composition in \eqref{eqn:BlanchfieldPairing}.
The Blanchfield pairing is known to be non-singular (since $\LF$ is a PID) and Hermitian~\cite{PowellBlanchfield}.
The next theorem relates the twisted Blanchfield pairing to the twisted Milnor pairing.

Recall that we fix map~$f \colon X \to S^1$ that induces a surjection~$f_* \colon \pi_1(X) \twoheadrightarrow \Z$.
Associated to $f_*$ is an infinite cyclic cover $X^\infty$ with $\pi_1(X^\infty)=\ker(f_*)$.
From this data, we built in Construction~\ref{cons:LambdaCoeff} the module~$M(\Lambda)$.
\begin{theorem}
\label{thm:BlanchfieldMilnor}
Let $X$ be a closed $3$-manifold equipped with a map $f \colon X \to S^1$.
Let $M$ be a unitary representation
such that $H_1(X;M(\Lambda))$ is~$\F[t^{\pm 1}]$-torsion.
The twisted Milnor pairing and the twisted Blanchfield pairing are related by the formula 
\[ \chi \circ \op{Bl}_{X,M} = \mu_{X,M}.\]
\end{theorem}

\begin{proof}
We work with the definition of the Milnor pairing from Subsection~\ref{sub:TwistedMilnorHomological}.
Throughout this proof, we use $i_+$ and $i_-$ to denote the maps $Q \to \Gamma^+$ and~$Q \to \Gamma^-$ described in Subsection~\ref{sub:Trace} as well as the various maps they induce on homology and cohomology.
The commutative diagram
\[
\begin{tikzcd}
	0 \ar[r]&  M(\Lambda)  \ar[r]\ar[d,"\id"] &M(Q) \ar[r] \ar{d}{(i_-,i_+)}  & M(Q /\Lambda) \ar{d}{i_+-i_-}\ar[r]&   0 \\
0 \ar[r] &  M(\Lambda)  \ar{r}{\bsm 1\\ 1\esm} &M(\Gamma_-) \oplus M(\Gamma_+) \ar{r}{\bsm -1 \ 1\esm}  &M(\Gamma)  \ar[r] & 0  
\end{tikzcd}
\]
implies that the map ${\beta^\Gamma}^{-1}$ decomposes as the composition $(i_+-i_-) \circ {\beta^{Q/\Lambda}}^{-1}$. The left pentagon of the following diagram therefore commutes, while the commutativity of the right square is clear:
\[ \begin{tikzcd}
& & \op{Hom}_\F(H_1(X;M(\Lambda)),\F)^\tr \\
	H_2(X;M(\Gamma)) \ar[r,"\PD_X"] &  H^1(X;M(\Gamma))  \ar{r}{\widetilde{b} \circ \op{ev}_{M(\Lambda)}}& \text{Hom}_\Lambda(H_1(X;M(\Lambda)),\Gamma)^\tr  \ar[u,"\text{const}_*"] &\\
	H_1(X; M(\Lambda)) \ar{u}{{(\beta^\Gamma})^{-1}} \ar{rd}[swap]{\PD_X}&   H^1(X;M(Q/\Lambda)) \ar{u}[swap]{(i_+-i_-)^*}  \ar[r, "\widetilde{b} \circ \op{ev}_{M(\Lambda)}"]& \text{Hom}_\Lambda(H_1(X;M(\Lambda)),Q/\Lambda)^\tr. \ar[u,"(i_+-i_-)^*"] \\
	& H^2(X;M(\Lambda)) \ar{u}[swap]{\op{(\beta^{Q/\Lambda}})^{-1}}
\end{tikzcd} \]
The uppermost route gives the twisted Milnor pairing and the lower most route gives the twisted Blanchfield pairing postcomposed with~$\chi$. 
We thus established that~$\chi \circ \op{Bl}_{X,M} = \mu_{X,M}$, concluding the proof of the theorem.
\end{proof}

In the untwisted case, Theorem~\ref{thm:BlanchfieldMilnor} recovers~\cite[Theorem A.1]{LitherlandCobordism} up to a sign.
\begin{remark}
\label{rem:LitherlandSign}
Using an argument involving intersections of singular chains, Litherland showed in the untwisted setting that $\chi \circ \op{Bl}_{X} =-\mu_{X}$.
Such arguments appear to be difficult to export to the twisted setting.
Regarding the sign: if we apply the trace to the formula for $\op{Bl}_X$ in~\cite[Theorem 1.3]{FriedlPowell}, we obtain $\chi \circ \op{Bl}_X \sim A-A^T$ which, according to Theorem~\ref{thm:MilnorIntersectionMonodromy} also represents $\mu_X$ in the fibered case.
\end{remark}

Next, we deduce some elementary properties of the twisted Milnor pairing.
\begin{corollary}
\label{cor:SkewSym}
The twisted Milnor pairing is non-singular and skew-Hermitian.
\end{corollary}
\begin{proof}
The fact that the Milnor pairing is skew-Hermitian follows from Theorem~\ref{thm:BlanchfieldMilnor} by additionally using the first item of Proposition~\ref{prop:trace} and the fact that the Blanchfield pairing is Hermitian.
The fact that the Milnor pairing is non-singular follows from Theorem~\ref{thm:BlanchfieldMilnor} by additionally using the second item of Proposition~\ref{prop:trace} and the fact that the Blanchfield pairing is non-singular (here we are using that the pairing~$b \colon M \otimes_{\Z[\pi]} M^\tr \to \F$ is non-singular). 
\end{proof}

\section{Signatures associated to the twisted Milnor pairing}
\label{sec:Isometric}

Since we have established that the Milnor pairing is skew-Hermitian, we can extract signature invariants by symmetrizing it.
As the Milnor pairing is also non-singular, it determines a ``skew-isometric structure".
The aim of this section is to discuss signature invariants associated to skew-isometric structures, and describe how these definitions can be understood from the viewpoint of Witt theory.
In Subsection~\ref{sub:SkewIsometric}, we review some basics on skew-isometric forms and their symmetrisation.
In Subsection~\ref{sub:WittSkewIsom}, we recall how the primary decomposition gives rise to additional signature invariants.

\subsection{Skew-isometric structures}
\label{sub:SkewIsometric}

Let $\F$ be a field with involution $x \mapsto \overline{x}$.
A sesquilinear form $\mu \colon H \times H \to \F $ is \emph{$\varepsilon$-Hermitian} if
$\ol{\mu(y,x)} = \varepsilon \mu(x,y)$.
For $\varepsilon = 1$, the form is called \emph{Hermitian} and for $\varepsilon = -1$, it is called \emph{skew-Hermitian}.

\begin{definition}
\label{def:IsometricStructure}
An \emph{$\varepsilon$-isometric structure} consists of a triple $(H,\mu,t)$, where
\begin{enumerate}
\item $H$ is a finite dimensional $\F$-vector-space;
\item $\mu \colon H \times H \to \F $ is a non-singular $\varepsilon$-Hermitian form over $\F$;
\item $t \colon H \to H$ is an isometry of $\mu$.
\end{enumerate}
\end{definition}

We are following Litherland's terminology~\cite[Appendix A]{LitherlandCobordism} (see also~\cite[Paragraph 6]{LevineInvariants}). On the other hand, Ranicki refers to isometric structures as ``autometric structures"~\cite[Example 28.12]{RanickiHighDimensionalKnotTheory}.

In practice, we will often choose bases and think of the $\varepsilon$-Hermitian form and isometry as matrices.
Here, it is understood that an $\varepsilon$-Hermitian matrix $A$ defines an $\varepsilon$-Hermitian form by setting~$\mu(x,y):= x^T A \overline{y}.$
We start with an example of a skew-isometric structure over $\LR$.

\begin{example}
\label{ex:Elementarye1R}
Over $\F=\R$ with the trivial involution, $\left(\R^2,\bsm 0&1 \\ -1&0 \esm,\bsm 0&-1 \\ 1&2\operatorname{Re}(\xi) \esm \right)$ is a skew-isometric structure for each $\xi \in S^1$. 
We denote it by $\mathbf{e}(1,1,\xi,\R)$ which, for~$n \in \N$,
matches with the notation for the linking form~$\mathfrak{e}(n,\epsilon,\xi,\R)$ from~\cite{BorodzikConwayPolitarczyk}.
\end{example} 

The next example shows that the role of the involution must not be neglected: for~$\F=\C$ with the trivial involution, the $\C$-vector space $H$ must be even dimensional (since it supports a non-singular skew-symmetric form), while with conjugation as an involution, odd-dimensional vector spaces arise.
\begin{example}
\label{ex:Elementarye1C}
Over $\C$ with conjugation as an involution, $(\C, 2i |\operatorname{Im}(\xi))|,(\xi))$ is a skew-isometric structure for each $\xi \in S^1$.
We denote it by $\mathbf{e}(1,1,\xi,\C)$.
\end{example} 

From now on, $\C$ will be understood to be endowed with the conjugation as an involution.
Next, we move on to our main example.
\begin{example} \label{ex:ClassicalMilnor}
	Let $X$ be a closed $3$-manifold with an epimorphism $\pi_1(X) \twoheadrightarrow \Z$.
	Let~$M$ be a unitary representation such that~$H:=H_1(X;M(\Lambda))$ is~$\Lambda$-torsion.
	Note that $H$ is a finite dimensional vector space over $\F$; recall Remark~\ref{rem:Module}.
	The action by $t \in \Lambda$ defines an automorphism~$t \colon H \to H$, which preserves the twisted Milnor pairing $\mu:=\mu_{X,M}$. 
	Thus,~$(H, \mu, t)$ defines a skew-isometric structure by Corollary~\ref{cor:SkewSym}.
\end{example}

Next, we discuss a symmetrization process for skew-isometric structures.
\begin{definition} \label{def:Symmetrization}
Let $(H,\mu,t)$ be a skew-isometric structure over $\F$, for which the~$\LF$-module $H$ has no $(t \pm 1)$-primary summand. 
The \emph{symmetrization} of~$(H,\mu,t)$ is the isometric structure $\operatorname{Sym}(H,\mu,t):=(H,b_\mu,t)$, where the Hermitian form $b_\mu$ is given by
\[ b_\mu(x,y):=\mu(tx,y)-\mu(x,ty). \]
The \emph{signature} of such a skew-isometric structure $(H,\mu,t)$ is the signature of~$b_\mu$:
\[  \sign(H,\mu,t):=\sign(b_\mu).  \]
\end{definition}

It is not difficult to prove that $b_\mu$ is Hermitian and that $t$ is an isometry of $b_\mu$.
To show that $b_\mu$ is non-singular, note that $b_\mu(x,y) = \mu( (t-t^{-1}) x,y)$ and use the fact that $(t \pm 1)$ induces an automorphism of~$H$.

\begin{remark}
\label{rem:SymmetrizationSingular}
If we allow the Hermitian form $b$ of an isometric structure $(H,b,t)$ to be singular, then symmetrization can be defined for \emph{any} skew-isometric form; we need not assume the absence of $(t \pm 1)$-primary summands.
Moreover, the signature~$\sign(b_\mu)$ is a well defined invariant of $(H,\mu,t)$ and so can be employed as a knot invariant via Milnor pairings~\cite{MilnorInfiniteCyclic, KawauchiSignature}. 
\end{remark}

Next, we describe some examples of the symmetrization process.
\begin{example}\label{ex:SymmetrizationExamples}
We consider the symmetrizations and signatures of the skew-isometric structures defined in 
Examples~\ref{ex:Elementarye1R} and~\ref{ex:Elementarye1C} for~$\xi \neq \pm 1$:
\begin{enumerate}
\item The symmetrization of the skew-isometric structure $\mathbf{e}(1,1,\xi,\R)$ is
\[ 
	\operatorname{Sym}\mathbf{e}(1,1,\xi,\R)=\left(\R^2,
	\begin{pmatrix}
	-2&2\operatorname{Re}(\xi) \\ 2\operatorname{Re}(\xi)&-2
	\end{pmatrix},
	\begin{pmatrix} 0&-1 \\ 1&2\operatorname{Re}(\xi) \end{pmatrix} \right).
\]
We deduce that the signature of $\mathbf{e}(1,1,\xi,\R)$ is
\[ \sign \mathbf{e}(1,1,\xi,\R)=-2. \]
\item The symmetrization of the skew-isometric structure $\mathbf{e}(1,1,\xi,\C)$ is
\[ \operatorname{Sym}\mathbf{e}(1,1,\xi,\C)=\left(\C,
(-\operatorname{sign}(\operatorname{Im}(\xi)) \cdot 4 \operatorname{Im}(\xi)^2),
(\xi) \right).\]
We deduce that the signature of $\mathbf{e}(1,1,\xi,\C)$ is
\[ \sign \mathbf{e}(1,1,\xi,\C)=-\operatorname{sign}(\operatorname{Im}(\xi)). \]
\end{enumerate}
\end{example}

Two $\varepsilon$-isometric structures~$(H_1,\mu_1,t_1)$ and $(H_2,\mu_2,t_2)$ are \emph{isomorphic} if there is an~$\F$-isomorphism $f \colon H_1 \to H_2$ 
with $\mu_2(f(x),f(y))=\mu_1(x,y)$ and~$t_2 \circ f=f \circ t_1$.
The set of isomorphism classes of $\varepsilon$-isometric structures over $\F$ is a commutative monoid under the direct sum.

As the isometric structure coming from a twisted Milnor pairing can have~$(t\pm 1)$-primary summands, we do not apply the symmetrization directly, but first consider its primary decomposition. After doing so, we see that symmetrization actually descends to Witt groups for the appropriate primary summands.

\subsection{Signatures of skew-isometric structures}
\label{sub:WittSkewIsom}
In this subsection, we use the primary decomposition to extract signature invariants from the symmetrization of~$(H,\mu,t)$.
Background references on the topic include~\cite{LitherlandCobordism, RanickiHighDimensionalKnotTheory}.
\medbreak
A Laurent polynomial~$p \in \F[t^{\pm 1}]$ is \emph{weakly symmetric} if $p \stackrel{.}{=} \ol p$. 
Let $(H, \mu, t)$ be an $\varepsilon$-isometric structure, and $p$ a weakly symmetric and irreducible polynomial in~$\LF$.
Denote by~$H_p$ the $p$-primary summand of $H$. 
The restriction~$\mu|_{H_p}$ is again a non-singular $\varepsilon$-Hermitian structure, and $t|_{H_p}$ is an isometry of $H_p$. Consequently, the $p$-\emph{primary part} $(H_p, \mu|_{H_p}, t_{H_p})$ is again an $\varepsilon$--isometric structure.

We describe frequently occurring examples of weakly symmetric and irreducible polynomials.
For $\F=\R$, set $p_\xi(t):=t-2\operatorname{Re}(\xi)+t^{-1}$, with~$\xi \in S^1 \cap \lbrace \operatorname{Im}(z)>0 \rbrace$, while for~$\F=\C$, set $p_\xi(t):=t-\xi$ with $\xi \in S^1$.
These polynomials are weakly symmetric and irreducible, and are referred to as \emph{basic polynomials}.

If $p_\xi \neq t\pm 1$, then $t\pm 1$ acts by isomorphisms on the primary summand $H_{p_\xi}$ and we can therefore consider the symmetrization of $(H_{p_\xi}, \mu|_{H_{p_\xi}}, t_{H_{p_\xi}})$.  

\begin{definition}
\label{def:MilnorSignature}
Given $\xi \in S^1 \setminus \lbrace \pm 1 \rbrace$, the \emph{Milnor $\xi$-signature}~$\sigma_{(H,\mu,t)}(\xi)$ of a skew-isometric structure $(H,\mu,t)$ is the signature of its $p_\xi$-primary part $(H_{p_\xi}, \mu|_{H_{p_\xi}}, t_{H_{p_\xi}})$: 
\[ \sigma_{(H,\mu,t)}(\xi)=\sign (b_{(\mu|_{H_{p_\xi}})}). \]
\end{definition}

The next remark comments on the reason for which we excluded $\xi =\pm 1$ and describes an equivalent definition of the Milnor $\xi$-signature~\cite{MilnorInfiniteCyclic,KirkLivingston}.
\begin{remark}
\label{rem:AlternativeMilnorSignature}
Since symmetrization commutes with the canonical projection to primary summands, $\sigma_{(H,\mu,t)}(\xi)$ 
can also be defined by restricting the symmetrization of the skew-Hermitian form $\mu$ to the $p_\xi$-primary summand $H_{p_\xi}$ of $H$:
	\[ \sigma_{(H,\mu,t)}(\xi) = \sign (b_\mu|_{H_{p_\xi}}). \]
Using Remark~\ref{rem:SymmetrizationSingular}, note that this definition can also be made for $\xi = \pm 1$.
In this case however, the Hermitian $b_\mu$ form might be singular.
\end{remark}
\begin{remark}
As $H$ has finite $\F$-dimension, for all but finitely many~$\xi \in S^1 \sm \{\pm 1\}$ the subspace~$H_{p_\xi}$ is trivial and the Milnor $\xi$--signature is $0$.
In fact, if $\xi$ is not a root of $\op{ord}_{\LF}H$, then the Milnor $\xi$--signature is $0$.
\end{remark}

Next, we introduce the Witt group of $\varepsilon$-isometric forms and show that the Milnor $\xi$-signature is well defined on Witt classes.
For an $\varepsilon$-isometric structure~$(H,\mu,t)$ over $\F$, define the following: a subspace~$L \subset H$ is \emph{invariant} if $t(L) = L$;
a \emph{metabolizer} is an invariant subspace~$L$ such that $L=L^\perp$ with respect to $\mu$;
an~$\varepsilon$-isometric structure is \emph{metabolic} if $H$ contains a \emph{metabolizer}. 
The set of isomorphism classes of metabolic $\varepsilon$-isometric structures forms a submonoid of the monoid of isomorphism classes of $\varepsilon$-isometric structures. 

\begin{definition}
\label{def:Isometric}
Given a field $\F$, the \emph{Witt group} $W_\varepsilon\operatorname{Aut}(\F)$ of $\varepsilon$-isometric structures is the quotient of the monoid of isomorphism classes of $\varepsilon$-isometric structures by the submonoid of isomorphism classes of metabolic $\varepsilon$-isometric structures.
\end{definition}

In Ranicki's monograph on high dimensional knot theory, this Witt group is denoted by~$L\operatorname{Aut}^0(\F,\varepsilon)$~\cite[Section 28.D]{RanickiHighDimensionalKnotTheory}.
$W_\varepsilon\operatorname{Aut}(\F)$ is known to be an abelian group: the inverse of $[(H,\mu,t)]$ is represented by $(H,-\mu,t)$.

{Given an irreducible weakly symmetric polynomial $p \in \LF$,} use $W_\varepsilon\operatorname{Aut}(\F,p)$ to denote the Witt group of $\varepsilon$-isometric structures defined on $p$-primary modules. 
Since $\LF$ is a PID, the primary decomposition of torsion $\LF$-modules leads to the following decomposition on the level of Witt groups~\cite[Appendix~A]{LitherlandCobordism}.
\begin{proposition}
\label{prop:PrimaryDecomposition}
Let $\mathcal{S}$ denote the set of irreducible weakly symmetric polynomials over $\LF$. For each $p \in \mathcal{S}$, the map $\pi_p \colon W_\varepsilon\operatorname{Aut}(\F) \to  W_\varepsilon\operatorname{Aut}(\F,p)$ that sends~$[(H,\mu, t)] \mapsto [(H_p,\mu|_{H_p}, t_{H_p})]$ is well defined. The collection of these maps induces an isomorphism
\[ W_\varepsilon\operatorname{Aut}(\F) \cong \bigoplus_{p \in \mathcal{S}} W_\varepsilon\operatorname{Aut}(\F,p). \]
\end{proposition}

Next, we observe that symmetrization carries metabolisers to metabolisers.
\begin{lemma} \label{lem:SymmetrizationWitt}
Let $(H,\mu,t)$ be a skew-isometric structure with no $(t \pm 1)$-primary summands.
If $L \subset H$ is a metabolizer, then $L$ is also a metabolizer of~$\operatorname{Sym}(H,\mu,t)$.
\end{lemma}


As a corollary, we deduce that symmetrization descends to Witt groups.
\begin{corollary} \label{cor:SymmetrizationPrimary}
Let $p \in \mathcal{S}$ be a weakly symmetric irreducible polynomial.
If~$p \neq t \pm 1$, then symmetrization descends to a well defined homomorphism
\[ \operatorname{Sym} \colon W_-\operatorname{Aut}(\F,p) \to W\operatorname{Aut}(\F,p).\]
\end{corollary}

In the next corollary, we use $\operatorname{forget}$ to denote the map that sends an isometric structure to its underlying Hermitian form.
Observe that this map descends to a homomorphism $\operatorname{forget} \colon W\operatorname{Aut}(\F,p_\xi) \to  W(\F)$ on the level of Witt groups.
Using Corollary~\ref{cor:SymmetrizationPrimary}, we obtain the following reformulation of the Milnor signatures.

\begin{corollary}\label{cor:SymmSignature}
Let $p_\xi \neq t \pm 1 \in \LF$ be a basic polynomial.
The Milnor $\xi$-signature~$\sigma_{(H,\mu,t)}(\xi)$ agrees with the following composition:
\[
\sigma(\xi) \colon 
W_-\operatorname{Aut}(\F) 
\xrightarrow{\pi_{p_\xi}} W_-\operatorname{Aut}(\F,p_\xi) 
\xrightarrow{\operatorname{Sym}} W\operatorname{Aut}(\F,p_\xi) 
\xrightarrow{\operatorname{forget}} W(\F)  \xrightarrow{\sign} \Z. \]
In particular, $\sigma_{(H,\mu,t)}(\xi)$ depends only on the Witt class of $(H,\mu,t)$.
\end{corollary}

\section{Milnor signatures and signatures jumps}
\label{sec:MilnorBlanchfieldSignatures}

In this final section, we use the trace-relation between the Milnor pairing and the Blanchfield pairing that was established in Theorem~\ref{thm:BlanchfieldMilnor} in order to relate the Milnor signatures of Kirk-Livingston to the more computable signature jumps of~\cite{BorodzikConwayPolitarczyk}. 
In Subsection~\ref{sub:WittLinkingForm}, we review some facts about the Witt group of linking forms.
In Subsection~\ref{sub:RelationToJump}, we relate the signatures of skew-isometric structures (described in Section~\ref{sec:Isometric}) to signature jumps of linking forms.

\subsection{The Witt group of linking forms and signature jumps}
\label{sub:WittLinkingForm}
We review some facts about the Witt group $W(\F(t),\LF)$ of linking forms over $\LF$, referring to~\cite[Appendix A]{LitherlandCobordism} and~\cite[Section 4]{BorodzikConwayPolitarczyk} for further details.
\medbreak
Given an irreducible polynomial $p$ in $\LF$, we denote 
by~$W(\F(t),\LF,p)$ the Witt group of linking forms defined over $p$-primary modules.
As $\LF$ is a PID, the Witt group $W(\F(t),\LF)$ is known to decompose as follows~\cite[Appendix~A]{LitherlandCobordism}.
\begin{proposition}
\label{prop:Primary}
Let $\mathcal{S}$ denote the set of irreducible weakly symmetric polynomials over $\LF$. The primary decomposition induces the isomorphism
\[ W(\F(t),\LF)\cong\bigoplus_{p \in \mathcal{S}} W(\F(t),\LF,p).\]
\end{proposition}

For later use, we describe some linking forms which turn out to be explicit generators of these  primary summands of $W(\F(t),\LF)$~\cite[Sections~2 and~4]{BorodzikConwayPolitarczyk}.
\begin{example}
\label{ex:BasicPairingXiS1RealComplex}
We recall two examples of linking forms from~\cite{BorodzikConwayPolitarczyk}.
\begin{enumerate}
\item Fix a complex number $\xi \in S^1$ with $\re (\xi)>0$, a positive integer $n$, and $\varepsilon= \pm 1$.
We associate to the basic polynomial $p_\xi(t):= (t-\xi)(1-\overline{\xi}t^{-1})$ the linking form~$\mathfrak{e}(n,\varepsilon,\xi,\R)$ defined on~$\LR/p_\xi(t)^n$ by the formula
\begin{equation}\label{eq:e_n_k_form_real}
\begin{split}
  \LR / p_\xi(t)^n \times\LR / p_\xi(t)^n& \to\R(t)/\LR \\
  (x,y) &\mapsto \frac{\varepsilon x \ol y}{p_\xi(t)^{n}}. 
\end{split}
\end{equation}
\item Fix a complex number $\xi\in S^1$ and a positive integer $n$.
We associate to the basic polynomial~$p_{\xi}(t):= t-\xi$ the linking form $\mathfrak{e}(n,\varepsilon,\xi,\C)$:
\begin{align*}
  \LC / p_{\xi}(t)^n \times \LC / p_{\xi}(t)^n &\to \C(t)/\LC,\\
  (x,y)  &\mapsto \frac{\varepsilon x \ol y}{p_{\xi}(t)^{\frac{n}{2}}p_{\bar{\xi}}(t^{-1})^{\frac{n}{2}}} &&\text{if $n$ is even}, \\
 (x,y)   &\mapsto \frac{\operatorname{sgn}(\operatorname{Im}(\xi)) \varepsilon (1-\xi t) x \ol y}{p_{\xi}(t)^{\frac{n+1}{2}}p_{\bar{\xi}}(t^{-1})^{\frac{n-1}{2}}}     &&\text{if $n$ is odd.}
\end{align*}
\end{enumerate}
\end{example}

We review dévissage
for linking forms on $p$-torsion modules in detail, before 
giving the general statement in Proposition~\ref{prop:Devissage} below.
\begin{construction}\label{const:pTorsionForm}
Let $p \in \LF$ be a weakly symmetric irreducible polynomial, 
and let $(H,\lambda)$ be a linking form, with $H$ a $p$--torsion $\LF$--module.
Consider the following $\LF$--submodule of $\F(t)/\LF$:
\begin{equation}\label{eq:Devissage}
	(\F(t)/\LF)_{/p}:=\lbrace [f] \in \F(t)/\LF  \ | \ f \in \F(t), \ pf  \in \LF \rbrace \subset \F(t)/\LF.
\end{equation}
It inherits the structure of a vector space over $R := \LF/p$, and multiplication by~$p$ defines an isomorphism $m_ p \colon (\F(t)/\LF)_{/p} \xrightarrow{\sim} R$ of vector spaces.  
Since $H$ is $p$-torsion, note that the form~$\lambda \colon H \times H \to \F(t)/\LF$ takes values in~$(\F(t)/\LF)_{/p}$. Similarly to $(\F(t)/\LF)_{/p}$ above, $H$ also inherits the structure of an $R$--vector space. Consequently, $(H, m_p \circ \lambda)$ is a form over $R$. 

A more detailed inspection reveals that this form is $\wt u$--Hermitian, where $\wt u \in R$ is obtained as follows: suppose $\ol p = up$ for a $u \in \LF$, which implies that $u \cdot \ol u = 1$. Denote the reduction of $u$ mod~$p$ by $\wt u$, and verify that $m_p \circ \lambda$ is $\wt{u}$-Hermitian:
\[ \ol{ (m_p \circ \lambda)(y,x) } = \ol p \cdot \ol{ \lambda(x,y)} = \wt u \cdot (m_p \circ \lambda)(x,y) .\]
\end{construction}

Construction~\ref{const:pTorsionForm} associates a Hermitian form~$(H, m_p \circ \lambda)$ to a linking form $(H,\lambda)$, when $H$ is a $p$-torsion $\LF$-module. The next proposition shows that this procedure is well defined on Witt groups and can be generalized from $p$-torsion modules to~$p$-primary modules. 
Use~$W_{\wt u}(R)$ to denote the Witt group of non-singular $\wt{u}$-Hermitian forms.
If $u=1$, then we write~$W(R)$ instead of~$W_1(R)$.

\begin{proposition}[d\'evissage]
\label{prop:Devissage}
Let $p$ be an irreducible weakly symmetric polynomial with $p=u \ol p$, and
denote the reduction mod~$p$ of $u$ by $\widetilde{u} \in \LF/p$. 
There is an isomorphism
\[ d \colon W(\F(t),\LF,p) \cong W_{\widetilde{u}}(\LF/p)\]
that sends $[H, \lambda]$ to $[H, m_p \circ \lambda]$, if $H$ is a $p$-torsion module; see Construction~\ref{const:pTorsionForm}.
\end{proposition}

One of the more difficult aspects of Proposition~\ref{prop:Devissage} is to prove that if $(H,\lambda)$ is a linking form, where $H$ is $p$-primary, then $(H,\lambda)$ is Witt equivalent to $(H',\lambda')$ where~$H'$ is $p$-torsion; we refer to~\cite[Chapter II.B.2]{Bourrigan} for a proof.
\begin{remark}
\label{rem:Odd}
This latter observation can be used to show that the basic linking form $\mathfrak{e}(2n+1,\varepsilon,\xi,\F)$ is Witt equivalent to $\mathfrak{e}(1,\varepsilon,\xi,\F)$ and that $\mathfrak{e}(2n,\varepsilon,\xi,\F)$ is metabolic~\cite[Section 4]{BorodzikConwayPolitarczyk}.
\end{remark}

We describe d\'evissage explicitly on the elementary linking forms of~\cite{BorodzikConwayPolitarczyk}.
\begin{example}
\label{ex:Devissage}
We illustrate Proposition~\ref{prop:Devissage} on the linking forms of Example~\ref{ex:BasicPairingXiS1RealComplex}.
\begin{enumerate}
	\item We apply d\'evissage to $\lambda:=\mathfrak{e}(1,\varepsilon,\xi,\R)$. 
For $\F = \R$ the polynomial~$p_\xi(t)$ is symmetric, and so~$u=1$.
Note furthermore that $\LR / p_\xi = \C$.
Since Construction~\ref{const:pTorsionForm} implies that $d(\lambda)(x,y)= \frac{p_\xi \cdot x\overline{y}}{p_\xi} = x\overline{y}$, we deduce that~$d(\lambda)=(1)$.

\item We apply d\'evissage to $\lambda:=\mathfrak{e}(1,\varepsilon,\xi,\C)$.
The polynomial $p_\xi = (t-\xi)$ is $u$-symmetric with $u=-t^{-1}\overline{\xi}$, and so $\widetilde{u}=-\overline{\xi}^2$.
Note furthermore that $\LC/p \cong \C$.
Example~\ref{ex:BasicPairingXiS1RealComplex} and 
Construction~\ref{const:pTorsionForm} imply that 
\[ d(\lambda)(x,y)=\operatorname{sign(Im}(\xi))  \left[ x\overline{y}(t-\xi)\frac{1-\xi t}{t-\xi} \right]
=\operatorname{sign(Im}(\xi))   x\overline{y}(1-\xi^2).\]
We conclude  that~$d(\lambda)=\operatorname{sign(Im}(\xi)) (1-\xi^2)$.
\end{enumerate}
\end{example}

Using the primary decomposition theorem and dévissage, we can now introduce a~$\Z$-valued invariant on the Witt group~$W(\F(t),\LF)$ of linking forms.
Note that if~$\widetilde{u}=\eta \cdot \overline{\eta}^{-1}$ for $\eta \in R$, then multiplication by $\eta$ induces an isomorphism~$(\eta \cdot) \colon W_{\widetilde{u}}(R) \xrightarrow{\sim} W(R)$.
In particular, multiplication by $\eta=i\overline{\xi}$ induces the following isomorphism~$(\eta \cdot) \colon W_{-\overline{\xi}^2}(\C) \xrightarrow{\sim}  W(\C)$.
\begin{definition}
\label{defn:SignatureJumpBasic}
Let $\F=\R,\C$ and let $p_\xi$ be a basic polynomial. Set $\eta := 1$ if~$\F=\R$ and $\eta := i \ol \xi$ if $\F=\C$. Define the homomorphism~$(\delta \sigma)(\xi) \colon W(\F(t),\LF) \to \Z$ as the composition
\[
	\delta \sigma(\xi) \colon W(\F(t),\LF) \xrightarrow{\pi_{p_\xi}} W(\F(t),\LF,p_\xi) \xrightarrow{d}  W_{\widetilde{u}}(\C) \xrightarrow{\eta \cdot} W(\C) \xrightarrow{\sign} \Z,
\]
where $\pi_{p_\xi}$ denotes the canonical projection, and $d$ denotes the d\'evissage isomorphism of Proposition~\ref{prop:Devissage}.
\end{definition}

From now on, we assume that $\F=\R,\C$.
We compare the homomorphism $\delta \sigma (\xi)$ from Definition~\ref{defn:SignatureJumpBasic} with the signature jumps introduced in~\cite[Definition 5.1]{BorodzikConwayPolitarczyk}.
\begin{remark}
\label{rem:BCPSign}
Given a linking form $(H,\lambda)$ over $\LF$, we note that $\delta \sigma_{(H,\lambda)} (\xi)$ agrees with the signature jump at $\xi$ if $\F=\R$, and with its opposite when $\F=\C$.
As every linking form is Witt equivalent to a direct sum of $\mathfrak{e}(1,\varepsilon,\xi,\F)$~\cite[Theorem~4.7 and 4.10]{BorodzikConwayPolitarczyk}, this follows by combining~\cite[Definition 5.1]{BorodzikConwayPolitarczyk} with Lemma~\ref{lem:SignatureJumpBasic}~below.
\end{remark}

Next, we compute the values of~$\delta \sigma(\xi)$ on the linking forms from Example~\ref{ex:BasicPairingXiS1RealComplex}.
\begin{lemma}\label{lem:SignatureJumpBasic}
For $\xi \in S^1 \setminus \lbrace \pm 1\rbrace$, the linking form $\mathfrak{e}(1,1,\xi,\F)$ of Example~\ref{ex:BasicPairingXiS1RealComplex} satisfies 
	\[ (\delta \sigma_{\mathfrak{e}(1,1,\xi,\F)})(\xi) = 1 .\]
\end{lemma}
\begin{proof}
	For $\F=\R$, Example~\ref{ex:Devissage} established that $d \big( \mathfrak{e}(1,1,\xi,\R) \big)=(1)$ and this proves the assertion in the real case since $\eta=1$.
	For $\F=\C$, Example~\ref{ex:Devissage} established that~$d \big( \mathfrak{e}(1,1,\xi,\C) \big)=(\operatorname{sign(Im}(\xi))(1-\xi^2))$. 
Multiplying by $\eta=i\overline{\xi}$, we obtain that~$i\overline{\xi} \operatorname{sign(Im}(\xi))(1-\xi^2)=2\operatorname{sign(Im}(\xi)) \operatorname{Im}(\xi)$.
Since this real number is always positive, the assertion is also proved in the complex case. 
\end{proof}

The next proposition is implicit in~\cite[Section 4]{BorodzikConwayPolitarczyk}.

\begin{proposition}\label{prop:WittIsZ}
For $\xi \in S^1 \setminus \lbrace \pm 1\rbrace$,
the abelian group~$W(\F(t),\LF,p_\xi) \cong \Z$ is freely generated 
by~$\mathfrak{e}(1,1,\xi,\F)$.
\end{proposition}
\begin{proof}
	Generators of the Witt group of linking forms $W(\F(t),\LF, p_\xi)$ are the elementary linking forms~$\mathfrak{e}(2n+1,\xi,\varepsilon,\F)$ described in Example~\ref{ex:BasicPairingXiS1RealComplex}; see\cite[Theorems 4.7 and 4.10]{BorodzikConwayPolitarczyk}

	We have seen in Remark~\ref{rem:Odd}, that $\mathfrak{e}(2n+1,\xi,\varepsilon,\F)$ is Witt equivalent to $\mathfrak{e}(1,\xi,\varepsilon,\F)$. Since $\mathfrak{e}(1,-1,\xi,\F)$ is the
opposite of $\mathfrak{e}(1,1,\xi,\F)$, the form~$\mathfrak{e}(1,1,\xi,\F)$ generates the whole of~$W(\F(t),\LF, p_\xi)$.
By Lemma~\ref{lem:SignatureJumpBasic}, the signature jump~$(\delta \sigma)(\xi)$ induces an isomorphism~$W(\F(t),\LF,p_\xi)\cong \Z$, and $\mathfrak{e}(1,1,\xi,\F)$ generates it freely.
\end{proof}

\subsection{Relation to signature jumps}\label{sub:RelationToJump}

We relate the jumps of skew-isometric structures to the jumps of linking forms via the trace map $\chi \colon \F(t)/\LF \to \F$.
As a first step, we show that $\chi$ maps~$\mathfrak{e}(1,1,\xi,\F)$ to 
$\mathbf{e}(1,1,\xi,\F)$.
\begin{lemma}\label{lem:ChiFormToForm}
For $\xi \in S^1 \setminus \lbrace \pm 1\rbrace$,
	let $\mathbf{e}(1,1,\xi,\F)) \in W_- \op{Aut}(\F, p_\xi)$ be the skew-isometric structure of Examples~\ref{ex:Elementarye1R} and \ref{ex:Elementarye1C}. Let $\mathfrak{e}(1,1,\xi,\F) \in W(\F(t), \LF, p_\xi)$ be the linking form defined in Example~\ref{ex:BasicPairingXiS1RealComplex}. Then
\[ \chi(\mathfrak{e}(1,1,\xi,\F)) = \mathbf{e}(1,1,\xi,\F).\]
\end{lemma}
\begin{proof}
	First we focus on the case~$\F = \R$.
Recall the formula for the elementary linking form $\lambda := \mathfrak{e}(1,1,\xi,\R)$:
\begin{align*}
\mathfrak{e}(1,1,\xi,\R) \colon \frac{\LR}{(t-\xi)(1-\overline{\xi} t^{-1})} \times \frac{\LR}{(t-\xi)(1-\overline{\xi} t^{-1})}  &\to \R(t)/\LR  \\
(x,y) &\mapsto  \frac{x \cdot \ol y}{(t-\xi)(1-\overline{\xi} t^{-1})}.
\end{align*}
Note that $(t-\xi)(1-\overline{\xi} t^{-1})=t-2\operatorname{Re}(\xi)+t^{-1}$ is a real symmetric polynomial.
For conciseness, we set $a:=2\operatorname{Re}(\xi)$.
We start with a preliminary computation:
\begin{align}
\label{eq:PrelimComputElementarye1R}
&i_+\left( \frac{1}{t-a+t^{-1}} \right)
=i_+\left( \frac{1}{t^{-1}(1-(at-t^2)} \right)
=t \sum_{k \geq 0} (at-t^2)^k, \\
&i_-\left( \frac{1}{t-a+t^{-1}} \right)
=i_-\left( \frac{1}{t(1-(at^{-1}-t^{-2})} \right)
=t^{-1} \sum_{k \leq 0} (at-t^2)^k. \nonumber
\end{align}
We now compute the skew-Hermitian form $\chi( \mathfrak{e}(1,1,\xi,\R))$.
Viewing $H$ as a real vector space of dimension $2$, we choose $1$ and $t$ as generators.
Observe that since the expressions in~\eqref{eq:PrelimComputElementarye1R} have no constant terms, we immediately deduce that
\begin{align*}
&\mu(t,t)=\mu (1,1)
=\chi \lambda (1,1)
=0.
\end{align*}
It therefore only remains to compute the value of $\mu(t,1)$ since $\mu(1,t)=-\mu(1,t)$.
First, we observe that~$\mu(t,1)=(\chi \circ \lambda) (t,1)=\chi\big( t \cdot \lambda (1,1)\big)$.
We therefore compute
\begin{align*}
&i_+\left( \frac{t}{t-a+t^{-1}} \right)
=t^2 \sum_{k \geq 0} (at-t^2)^k
=\ldots+0+\ldots, \\
&i_-\left( \frac{t}{t-a+t^{-1}} \right)
= \sum_{k \leq 0} (at-t^2)^k
=\ldots+1+\ldots .
\end{align*}
We deduce that 
\[ \mu(t,1)=(\chi \circ \lambda) (t,1)= \chi\big( t \lambda (1,1)\big)=0-(1)=-1. \]
Therefore, we have shown that the skew-Hermitian form $\mu:= \chi( \mathfrak{e}(1,1,\xi,\R))$ is represented by~$\bsm 0&1 \\ -1&0 \esm$.
We conclude by computing the isometry: we have $t \cdot 1=t$ and, since $t^2-2t\operatorname{Re}(\xi)+1=0$, we deduce that $t\cdot t=t^2=2t\operatorname{Re}(\xi)-1$.
This shows that the isometry is represented by~$\bsm 0&-1 \\ 1&2\operatorname{Re}(\xi) \esm$.

We move on to the complex case~$\F = \C$.
Recall the formula for the elementary linking form  $\mathfrak{e}(1,1,\xi,\C)$:
\begin{align*}
\lambda:=\mathfrak{e}(1,1,\xi,\C) \colon \frac{\LC}{(t-\xi)} \times \frac{\LC}{(t-\xi)}  &\to \C(t)/\LC  \\
(x,y) &\mapsto \operatorname{sign}(\operatorname{Im}(\xi))\frac{1-\xi t}{t-\xi} x\ol y.
\end{align*}
We start with a preliminary computation:
\begin{align*}
&i_+\left( \frac{1}{t-\xi} \right)
=i_+\left( \frac{1}{-\xi(1-t\overline{\xi})} \right)
=-\overline{\xi} \sum_{k \geq 0} \overline{\xi}^{k}t^k, \\
&i_-\left( \frac{1}{t-\xi} \right)
=i_-\left( \frac{1}{t(1-\xi t^{-1})} \right)=t^{-1}\sum_{k \geq 0} \xi^kt^{-k}.
\end{align*}
We now multiply by $1-\xi t$ and compute constant terms:
\begin{align*}
&i_+\left( \frac{1-\xi t}{t-\xi} \right)
=-(1-\xi t)\overline{\xi} \sum_{k \geq 0} \overline{\xi}^{k}t^k=(t-\overline{\xi})\sum_{k \geq 0} \overline{\xi}^{k}t^k= +\ldots-\overline{\xi}+\ldots, \\
&i_-\left( \frac{1-\xi t}{t-\xi} \right)
=t^{-1}(1-\xi t) \sum_{k \geq 0} \xi^kt^{-k} 
=(t^{-1}-\xi) \sum_{k \geq 0} \xi^kt^{-k}
=\ldots -\xi+\ldots.
\end{align*}
We deduce that 
\[  \chi \left(\frac{1-\xi t}{t-\xi}\right)=-\overline{\xi}-(-\xi)=2i\operatorname{Im}(\xi).  \]
Multiplying by $\operatorname{sign}(\operatorname{Im}(\xi))$ shows that 
$\chi ( \mathfrak{e}(1,1,\xi,\C))$ is represented by~$\mathbf{e}(1,1,\xi,\C)$.
We conclude by computing the isometry: since we have $t-\xi=0$ in $\LC/(t-\xi)$, we deduce that~$t \cdot 1=t=\xi$, as claimed. This concludes the proof of the claim.
\end{proof}

Combining the result below with Theorem~\ref{thm:BlanchfieldMilnor} immediately yields Theorem~\ref{thm:SignaturesAreEqual} from the introduction.
\begin{theorem}\label{thm:MainMain}
Let $(H,\lambda)$ be a linking form over $\LF$ and let $p_\xi$ be a basic polynomial, with $\xi \in S^1 \setminus \lbrace \pm 1 \rbrace$.
The signature jump of $\lambda$ at $\xi$ coincides with the Milnor signature of the skew-isometric structure $\chi \circ \mu$ at $\xi$:
\[ 
\sigma_{(H,\chi \circ \mu,t)}(\xi)=
\begin{cases}
 -2\delta \sigma_{(H,\lambda)}(\xi)  &\quad \text{ if } \F=\R, \\
\operatorname{sign}(\operatorname{Im}(\xi))\delta \sigma_{(H,\lambda)}(\xi) &\quad \text{ if } \F=\C.  
\end{cases}
\]
\end{theorem}

\begin{proof}
As both signatures are invariant under Witt equivalence; we can show the equality on $W(\F(t),\LF)$.
As in Subsection~\ref{sub:WittLinkingForm}, we set $\widetilde{u}:=1,\eta:=1$ if~$\F=\R$ and $\widetilde{u}:=-\overline{\xi}^2,\eta:=i\overline{\xi}$ if~$\F=\C$.
We also set $c_\R:=-2$ and~$c_\C:=-\operatorname{sign}(\operatorname{Im}(\xi))$.
We claim that the theorem, that is the equation~$\sigma_{(H,\chi \circ \lambda,t)}(\xi)=c_\F \cdot \delta \sigma_{(H,\lambda)}(\xi)$, will follow
if we manage to show that the following diagram commutes:
\begin{equation}
\label{eq:DiagramWitt}
\begin{tikzcd}
W(\F(t),\LF) \ar[r,"\pi_{p_\xi}"] \ar{d}{\chi,\cong} 
& W(\F(t),\LF,p_\xi) \ar{r}{d,\cong} \ar{d}{\chi,\cong} 
& W_{\widetilde{u}}(\C) \ar{r}{\eta,\cong}
& W(\C) \ar{r}{\sign,\cong}
& \Z \ar{d}{\cdot c_\F}
\\
W_-\operatorname{Aut}(\F) \ar{r}{\pi_{p_\xi}}
& W_-\operatorname{Aut}(\F,p_\xi) \ar{r}{\operatorname{Sym}}
& W\operatorname{Aut}(\F,p_\xi) \ar{r}{\operatorname{forget}}
& W(\F)  \ar{r}{\sign,\cong}
&\Z.
\end{tikzcd}
\end{equation}
If we travel through the top route, then we obtain~$c_\F \delta \sigma_{(H,\lambda)}(\xi)$ by Definition~\ref{defn:SignatureJumpBasic}, while traveling through the bottom route produces 
$\sigma_{(H,\chi \circ \lambda,t)}(\xi)$ by Corollary~\ref{cor:SymmSignature}. 
If the diagram commutes, then we obtain $\sigma_{(H,\chi \circ \lambda,t)}(\xi)=c_\F \delta \sigma_{(H,\lambda)}(\xi)$, as desired.
 
We check that the diagram commutes for every linking form by verifying it on the generator~$\mathfrak{e}(1,1,\xi,\F) \in W(\F(t),\LF,p_\xi)$; see Proposition~\ref{prop:WittIsZ}.
Using Lemma~\ref{lem:SignatureJumpBasic}, if we follow the top route of the diagram for $\mathfrak{e}(1,1,\xi,\F)$, then we obtain
\[ c_\F \cdot \delta \sigma_{\mathfrak{e}(1,1,\xi,\F)}(\xi)=c_\F.\]

We must therefore show that the image of $\mathbf{e}(1,1,\xi,\F)$ via the bottom route of the diagram in~\eqref{eq:DiagramWitt} also equals $c_\F$. 
The relation~$\chi(\mathfrak{e}(1,1,\xi,\F)) =\mathbf{e}(1,1,\xi,\F)$ holds by Lemma~\ref{lem:ChiFormToForm}, and we then use Example~\ref{ex:SymmetrizationExamples} 
to deduce that
\[
\sigma_{\mathbf{e}(1,1,\xi,\F)}(\xi) =
\sign \circ \operatorname{forget} \circ \operatorname{Sym}\mathbf{e}(1,1,\xi,\F)= c_\F.
\] 
\end{proof}

\bibliographystyle{alpha}
\bibliography{hom}

\begin{thebibliography}{Kea79b}

\bibitem[AT16]{AbeTagami}
Tetsuya Abe and Keiji Tagami.
\newblock Fibered knots with the same 0-surgery and the slice-ribbon
  conjecture.
\newblock {\em Math. Res. Lett.}, 23(2):303--323, 2016.

\bibitem[Bak16]{Baker}
Kenneth~L. Baker.
\newblock A note on the concordance of fibered knots.
\newblock {\em J. Topol.}, 9(1):1--4, 2016.

\bibitem[BCP18]{BorodzikConwayPolitarczyk}
Maciej Borodzik, Anthony Conway, and Wojciech Politarczyk.
\newblock {Twisted Blanchfield pairings, twisted signatures and Casson-Gordon
  invariants}.
\newblock {\em ArXiv 1809.08791}, 2018.

\bibitem[Bou13]{Bourrigan}
Maxime Bourrigan.
\newblock Quasimorphismes sur les groupes de tresses et forme de blanchfield.
\newblock {\em PhD thesis}, 2013.

\bibitem[Bre93]{Bredon93}
Glen Bredon.
\newblock {\em Topology and Geometry}.
\newblock Springer-Verlag, New York, 1993.

\bibitem[CG78]{CassonGordon1}
Andrew~J. Casson and Cameron~McA. Gordon.
\newblock On slice knots in dimension three.
\newblock In {\em Algebraic and geometric topology \textup{(}{P}roc. {S}ympos.
  {P}ure {M}ath., {S}tanford {U}niv., {S}tanford, {C}alif., 1976\textup{)},
  {P}art 2}, Proc. Sympos. Pure Math., XXXII, pages 39--53. Amer. Math. Soc.,
  Providence, R.I., 1978.

\bibitem[CG86]{CassonGordon2}
Andrew~J. Casson and Cameron~McA. Gordon.
\newblock Cobordism of classical knots.
\newblock In {\em \`A la recherche de la topologie perdue}, volume~62 of {\em
  Progr. Math.}, pages 181--199. Birkh\"auser Boston, Boston, MA, 1986.
\newblock With an appendix by P. M. Gilmer.

\bibitem[CKP19]{ConwayKimPolitarczyk}
Anthony Conway, Min~Hoon Kim, and Wojciech Politarczyk.
\newblock {Non-slice linear combinations of iterated torus knots}.
\newblock {\em Arxiv 1910.01368}, 2019.

\bibitem[Con17]{ConwayThesis}
Anthony Conway.
\newblock {\em Invariants of colored links and generalizations of the {B}urau
  representation}.
\newblock PhD thesis, Universit\'e de Gen\`eve, 2017.
\newblock
  \href{http://dx.doi.org/10.13097/archive-ouverte/unige:99352}{10.13097/archive-ouverte/unige:99352}.

\bibitem[Erl69]{Erle}
Dieter. Erle.
\newblock Die quadratische {F}orm eines {K}notens und ein {S}atz \"uber
  {K}notenmannigfaltigkeiten.
\newblock {\em J. Reine Angew. Math.}, 236:174--218, 1969.

\bibitem[FK06]{FriedlKim}
Stefan Friedl and Taehee Kim.
\newblock The {T}hurston norm, fibered manifolds and twisted {A}lexander
  polynomials.
\newblock {\em Topology}, 45(6):929--953, 2006.

\bibitem[FP17]{FriedlPowell}
Stefan Friedl and Mark Powell.
\newblock A calculation of {B}lanchfield pairings of 3-manifolds and knots.
\newblock {\em Mosc. Math. J.}, 17(1):59--77, 2017.

\bibitem[Hat02]{Hatcher}
Allen Hatcher.
\newblock {\em Algebraic topology}.
\newblock Cambridge University Press, Cambridge, 2002.

\bibitem[HHL19]{HeHubbardTruong}
Dongtai He, Diana Hubbard, and Truong Linh.
\newblock On the upsilon invariant of fibered knots and right-veering open
  books.
\newblock {\em ArXiv 1911.05776}, 2019.

\bibitem[HKL12]{HeddenKirkLivingston}
Matthew Hedden, Paul Kirk, and Charles Livingston.
\newblock Non-slice linear combinations of algebraic knots.
\newblock {\em J. Eur. Math. Soc. \textup{(}JEMS\textup{)}}, 14(4):1181--1208,
  2012.

\bibitem[HM06]{Harvey}
Reese Harvey and Giulio Minervini.
\newblock Morse {N}ovikov theory and cohomology with forward supports.
\newblock {\em Math. Ann.}, 335(4):787--818, 2006.

\bibitem[HS97]{HiltonStammbach}
Peter Hilton and Urs Stammbach.
\newblock {\em {A course in homological algebra. 2nd ed.}}
\newblock New York, NY: Springer, 2nd ed. edition, 1997.

\bibitem[Kaw77]{KawauchiOnQuadratic}
Akio Kawauchi.
\newblock On quadratic forms of {$3$}-manifolds.
\newblock {\em Invent. Math.}, 43(2):177--198, 1977.

\bibitem[Kaw86]{KawauchiSignature}
Akio Kawauchi.
\newblock The signature invariants of infinite cyclic coverings of closed
  odd-dimensional manifolds.
\newblock In {\em Algebraic and topological theories ({K}inosaki, 1984)}, pages
  52--85. Kinokuniya, Tokyo, 1986.

\bibitem[Kea79a]{KeartonCompound}
Cherry. Kearton.
\newblock The {M}ilnor signatures of compound knots.
\newblock {\em Proc. Amer. Math. Soc.}, 76(1):157--160, 1979.

\bibitem[Kea79b]{KeartonSignature}
Cherry Kearton.
\newblock Signatures of knots and the free differential calculus.
\newblock {\em Quart. J. Math. Oxford Ser. (2)}, 30(118):157--182, 1979.

\bibitem[KL99]{KirkLivingston}
Paul Kirk and Charles Livingston.
\newblock Twisted {A}lexander invariants, {R}eidemeister torsion, and
  {C}asson-{G}ordon invariants.
\newblock {\em Topology}, 38(3):635--661, 1999.

\bibitem[Lai96]{Laitinen96}
Erkki Laitinen.
\newblock End homology and duality.
\newblock {\em Forum Math.}, 8(1):121--133, 1996.

\bibitem[Lan85]{Lang85}
Serge Lang.
\newblock {\em Complex analysis}, volume 103 of {\em Graduate Texts in
  Mathematics}.
\newblock Springer-Verlag, New York, second edition, 1985.

\bibitem[Lev69]{LevineInvariants}
Jerome Levine.
\newblock Invariants of knot cobordism.
\newblock {\em Invent. Math.}, 8:98--110; addendum, ibid. 8 (1969), 355, 1969.

\bibitem[Lev77]{LevineKnotModules}
Jerome Levine.
\newblock Knot modules. {I}.
\newblock {\em Trans. Amer. Math. Soc.}, 229:1--50, 1977.

\bibitem[Lev89]{LevineMetabolicHyperbolic}
Jerome Levine.
\newblock Metabolic and hyperbolic forms from knot theory.
\newblock {\em J. Pure Appl. Algebra}, 58(3):251--260, 1989.

\bibitem[Lit79]{LitherlandIterated}
Richard Litherland.
\newblock Signatures of iterated torus knots.
\newblock In {\em Topology of low-dimensional manifolds \textup{(}{P}roc.
  {S}econd {S}ussex {C}onf., {C}helwood {G}ate, 1977\textup{)}}, volume 722 of
  {\em Lecture Notes in Math.}, pages 71--84. Springer, Berlin, 1979.

\bibitem[Lit84]{LitherlandCobordism}
Richard Litherland.
\newblock Cobordism of satellite knots.
\newblock In {\em Four-manifold theory ({D}urham, {N}.{H}., 1982)}, volume~35
  of {\em Contemp. Math.}, pages 327--362. Amer. Math. Soc., Providence, RI,
  1984.

\bibitem[LM85]{LivingstonMelvin}
Charles Livingston and Paul Melvin.
\newblock Abelian invariants of satellite knots.
\newblock In {\em Geometry and topology ({C}ollege {P}ark, {M}d., 1983/84)},
  volume 1167 of {\em Lecture Notes in Math.}, pages 217--227. Springer,
  Berlin, 1985.

\bibitem[Mat77]{Matumoto}
Takao Matumoto.
\newblock On the signature invariants of a non-singular complex sesquilinear
  form.
\newblock {\em J. Math. Soc. Japan}, 29(1):67--71, 1977.

\bibitem[Mil68]{MilnorInfiniteCyclic}
John~W. Milnor.
\newblock Infinite cyclic coverings.
\newblock In {\em Conference on the {T}opology of {M}anifolds ({M}ichigan
  {S}tate {U}niv., {E}. {L}ansing, {M}ich., 1967)}, pages 115--133. Prindle,
  Weber \& Schmidt, Boston, Mass., 1968.

\bibitem[Mis17]{Misev}
Filip Misev.
\newblock On families of fibred knots with equal seifert forms.
\newblock {\em Communications in Analysis and Geometry.}, 2017.
\newblock (to appear).

\bibitem[MP18]{MillerPowell}
Allison~N. Miller and Mark Powell.
\newblock Symmetric chain complexes, twisted blanchfield pairings, and knot
  concordance.
\newblock {\em Algebr. Geom. Topol.}, 18(6):3425--3476, 2018.

\bibitem[MS19]{MisevSpano}
Filip Misev and Gilberto Spano.
\newblock Tight fibred knots without l-space surgeries.
\newblock {\em ArXiv 1906.11760}, 2019.

\bibitem[Nos16]{Nosaka}
Takefumi Nosaka.
\newblock Twisted cohomology pairings of knots ii; to classical invariants.
\newblock {\em Arxiv 1602.01131}, 2016.

\bibitem[Pow16]{PowellBlanchfield}
Mark Powell.
\newblock Twisted {B}lanchfield pairings and symmetric chain complexes.
\newblock {\em Q. J. Math.}, 67(4):715--742, 2016.

\bibitem[Ran98]{RanickiHighDimensionalKnotTheory}
Andrew Ranicki.
\newblock {\em High-dimensional knot theory}.
\newblock Springer Monographs in Mathematics. Springer-Verlag, New York, 1998.
\newblock Algebraic surgery in codimension 2, With an appendix by Elmar
  Winkelnkemper.

\bibitem[Ran02]{Ran02}
Andrew Ranicki.
\newblock {\em Algebraic and geometric surgery}.
\newblock Oxford University Press, Oxford, UK, 2002.

\bibitem[Sav99]{SavelievLectures}
Nikolai Saveliev.
\newblock {\em Lectures on the topology of {$3$}-manifolds}.
\newblock De Gruyter Textbook. Walter de Gruyter \& Co., Berlin, 1999.
\newblock An introduction to the Casson invariant.

\bibitem[Wei94]{Weibel94}
Charles~A. Weibel.
\newblock {\em An introduction to homological algebra}, volume~38 of {\em
  Cambridge Studies in Advanced Mathematics}.
\newblock Cambridge University Press, Cambridge, 1994.

\end{thebibliography}
\end{document}